\documentclass[12pt,a4paper]{article}
\usepackage{hyperref}
\usepackage{mathtools}
\usepackage[nottoc]{tocbibind}
\usepackage[utf8x]{inputenc}
\usepackage{amsfonts,longtable,amssymb,amsmath,latexsym,dsfont,mathabx}
\usepackage[main=english,polish]{babel}
\usepackage{wrapfig}
\usepackage[title]{appendix}
\usepackage{amsthm, amsmath, amssymb, amsfonts, amscd, amscd, geometry}
\newsavebox{\ssa}
\usepackage{bbm}
\usepackage{float}
\usepackage{tikz}
\usepackage{xcolor} 
\usepackage{lipsum}
\usepackage{cjhebrew}
\usepackage[math]{cellspace}
\DeclareFontFamily{U}{rcjhbltx}{}
\DeclareFontShape{U}{rcjhbltx}{m}{n}{<->rcjhbltx}{}
\DeclareSymbolFont{hebrewletters}{U}{rcjhbltx}{m}{n}

\DeclareMathOperator{\ulam}{Ulam}
\DeclareMathOperator{\pulam}{pseudo-Ulam}
\DeclareMathOperator{\culam}{c-Ulam}
\DeclareMathOperator{\eone}{E_1}
\DeclareMathOperator{\etwo}{E_2}
\DeclareMathOperator{\oone}{O_1}
\DeclareMathOperator{\otwo}{O_2}

\DeclareMathSymbol{\kaf}{\mathord}{hebrewletters}{107}
\DeclareMathSymbol{\vav}{\mathord}{hebrewletters}{119}
\DeclareMathSymbol{\lamed}{\mathord}{hebrewletters}{108}
\DeclareMathSymbol{\mem}{\mathord}{hebrewletters}{163}

\newcommand{\floor}[1]{\left\lfloor#1\right\rfloor}
\newcommand{\ceil}[1]{\left\lceil#1\right\rceil}
\newcommand{\abs}[1]{\left|#1\right|}
\newcommand{\norm}[1]{\left\|#1\right\|}

\makeatletter
\newcounter{savesection}
\newcounter{apdxsection}
\renewcommand\appendix{\par
  \setcounter{savesection}{\value{section}}%
  \setcounter{section}{\value{apdxsection}}%
  \setcounter{subsection}{0}%
  \gdef\thesection{\@Alph\c@section}}
\newcommand\unappendix{\par
  \setcounter{apdxsection}{\value{section}}%
  \setcounter{section}{\value{savesection}}%
  \setcounter{subsection}{0}%
  \gdef\thesection{\@arabic\c@section}}
\makeatother

\newtheorem{theorem}{Theorem}[section]
\newtheorem{corollary}[theorem]{Corollary}

\newtheorem{lemma}[theorem]{Lemma}
\newtheorem{prop}[theorem]{Proposition}

\theoremstyle{definition}
\newtheorem{definition}[theorem]{Definition}
\newtheorem{remark}[theorem]{Remark}
\newtheorem{conjecture}[theorem]{Conjecture}
\newtheorem{case}{Case}
\newtheorem{subcase}{Case}
\numberwithin{subcase}{case}
\usepackage{verbatim}

\usepackage[capitalize]{cleveref}
	\crefname{claim}{Claim}{Claims}
	\Crefname{claim}{Claim}{Claims}
	\crefname{app-corollary}{Corollary}{Corollaries}
	\Crefname{app-corollary}{Corollary}{Corollaries}
	\crefname{app-definition}{Definition}{Definitions}
	\Crefname{app-definition}{Definition}{Definitions}
	\crefname{figure}{Figure}{Figures}
	\Crefname{figure}{Figure}{Figures}
	\crefname{lemma}{Lemma}{Lemmata}
	\Crefname{lemma}{Lemma}{Lemmata}
	\crefname{definition}{Definition}{Definitions}
	\Crefname{definition}{Definition}{Definitions}
	\crefname{app-lemma}{Lemma}{Lemmata}
	\Crefname{app-lemma}{Lemma}{Lemmata}
	\crefname{app-proposition}{Proposition}{Proposition}
	\Crefname{app-proposition}{Proposition}{Proposition}
	\crefname{app-theorem}{Theorem}{Theorems}
	\Crefname{app-theorem}{Theorem}{Theorems}
\usepackage{amsmath}
\usepackage{color,colortbl}
\usepackage[center]{caption}
\usepackage{subcaption}

\def \N {\mathbb{N}}

\def \Z {\mathbb{Z}}

\def \LL {\mathcal{L}}

\def \PP {\mathcal{P}}

\def \UU {\mathcal{U}}

\def \e {\varepsilon}

\def \l {\lambda}

\providecommand{\keywords}[1]
{
  \small	
  \textbf{\textit{Keywords---}} #1
}

\title{On fractal patterns in Ulam words}

\author{Andrei Mandelshtam\footnote{Department of Mathematics, Stanford University, Stanford, CA 94305, USA}\hspace{0.1in}\footnote{\href{mailto:andman@stanford.edu}{andman@stanford.edu}}}
\date{}

\begin{document}
\maketitle
\selectlanguage{english}

\begin{abstract}
%In 1964, \begin{otherlanguage}{polish}Stanis"law Ulam \end{otherlanguage} introduced a curious type of sequence with many unusual yet unproven properties related to periodicity, self-similarity, density, and concentration. In 2017, Steinerberger and Kravitz generalized the Ulam sequence to other settings, observing and proving similar properties. In particular, Ulam words with one $1$ form a Sierpinski gasket. 
% We initiate the study of fractal patterns in Ulam words. 
% We show that already a seemingly simple set of Ulam words -- those with two $1$'s -- possess an intricate intrinsic structure. We create a logarithmic-time algorithm to determine whether any given such word is Ulam, uncovering properties such as biperiodicity and various parity conditions, as well as sharp bounds on the number of $0$'s outside the two $1$'s. We also discover and prove that sets of Ulam words indexed by the number $y$ of $0$'s between the two $1$'s have an inherent dual hierarchical structure, determined by the arithmetic properties of $y.$ In particular, this allows us to construct an infinite family of self-similar fractals $\Tilde{U}(y)$ indexed by the set of $2$-adic integers $y,$ containing for example the outward Sierpinski gasket as $\Tilde{U}(-1).$
Ulam words are binary words defined recursively as follows: the length-$1$ Ulam words are $0$ and $1$, and a binary word of length $n$ is Ulam if and only if it is expressible uniquely as a concatenation of two shorter, distinct Ulam words. We discover, fully describe, and prove a surprisingly rich structure already in the set of Ulam words containing exactly two $1$'s. In particular, this leads to a complete description of such words and a logarithmic-time algorithm to determine whether a binary word with two $1$'s is Ulam. Along the way, we uncover delicate parity and biperiodicity properties, as well as sharp bounds on the number of $0$'s outside the two $1$'s. We also show that sets of Ulam words indexed by the number $y$ of $0$'s between the two $1$'s have intricate tensor-based hierarchical structures determined by the arithmetic properties of $y$. This allows us to construct an infinite family of self-similar Ulam-word-based fractals indexed by the set of $2$-adic integers, containing the outward Sierpinski gasket as a special case.
% 
% Generalizing this structure, we obtain fractal shapes related in particular to the Sierpinski gasket.
% \lipsum[15] \lipsum[16]
\end{abstract}

% \tableofcontents

\section{Introduction}
Many integer sequences have curious properties that are difficult to understand. In 1964, \begin{otherlanguage}{polish}%Stanis"law 
Ulam \end{otherlanguage} introduced one such sequence---the \textit{Ulam sequence} \cite{ulam-paper}. It is inductively defined with $U_1=1,U_2=2,$ and $U_n$ being the smallest positive integer not yet in the sequence that is expressible \textit{uniquely} as a sum of two distinct previous terms. The first few terms of the sequence are $1,2,3,4,6,8,11,\dots.$ This sequence displays several interesting periodicity and concentration properties that cannot be due to random sampling. In 2015, Steinerberger performed a Fourier analysis on the first ten million terms of the sequence and discovered that there is a value $\l$ such that all but four terms of the Ulam sequence fall into the central-half interval modulo $\l$ \cite{hidden-signal}. Moreover, it is conjectured that for every $\e>0$ all but finitely many Ulam numbers fall within $\e$ of the central third. These empirical observations give rise to an almost linear-time algorithm to calculate successive Ulam numbers and also explain some curious facts, such as that there are only four pairs of Ulam numbers differing by $1$ \cite{ulam-gaps}.

More recently, it has become of interest to consider Ulam-type sequences and sets in different contexts. For instance, Ulam sets in two dimensions display a fascinating mixture of order and chaos \cite{noah-stefan}. Ulam-type multiplicative sets can also be defined on the complex plane, which naturally leads to examining Ulam sets in groups $\Z\times\Z/n\Z.$ In contrast with the original Ulam sequence, Ulam sets in $\Z\times\Z/n\Z$ can be finite \cite{noah-stefan}.

In this paper, we will look specifically at the \textit{binary Ulam set} $\UU,$ a set defined on the free group on two letters, conveniently denoted by $0$ and $1.$ 
\begin{definition}[\cite{student-paper}]
Define $\UU(1)=\{0,1\}.$ For each $n>1,$ inductively define $\UU(n)$ as the set of all length-$n$ words that are expressible uniquely as a concatenation of two distinct words in $\UU(i)$ and $\UU(j)$ for some $i,j<n.$ Then define $\UU=\bigcup\limits_{n=1}^\infty\UU(n).$ Words in $\UU$ are called \textit{Ulam words}.
\end{definition}
\begin{figure}[H]
    \centering
     \begin{subfigure}[b]{0.32\textwidth}
         \centering
         \includegraphics[width=\textwidth]{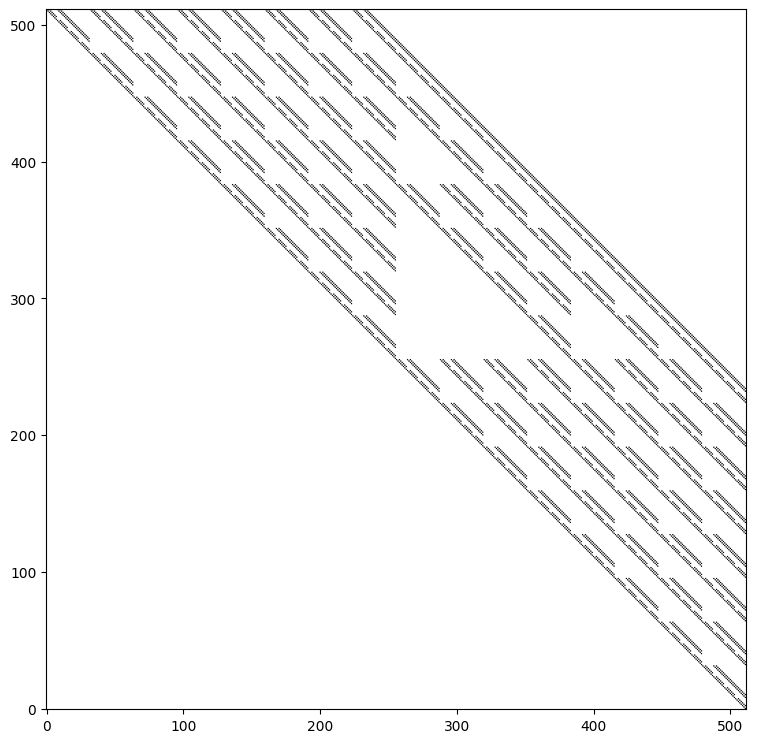}
         \caption{$(x,z)$ such that $0^x10^{276}10^z$ is Ulam}
         \label{fig:276}
     \end{subfigure}
     \hfill
    %  \hspace{-0.9in}
     \begin{subfigure}[b]{0.32\textwidth}
         \centering
         \includegraphics[width=\textwidth]{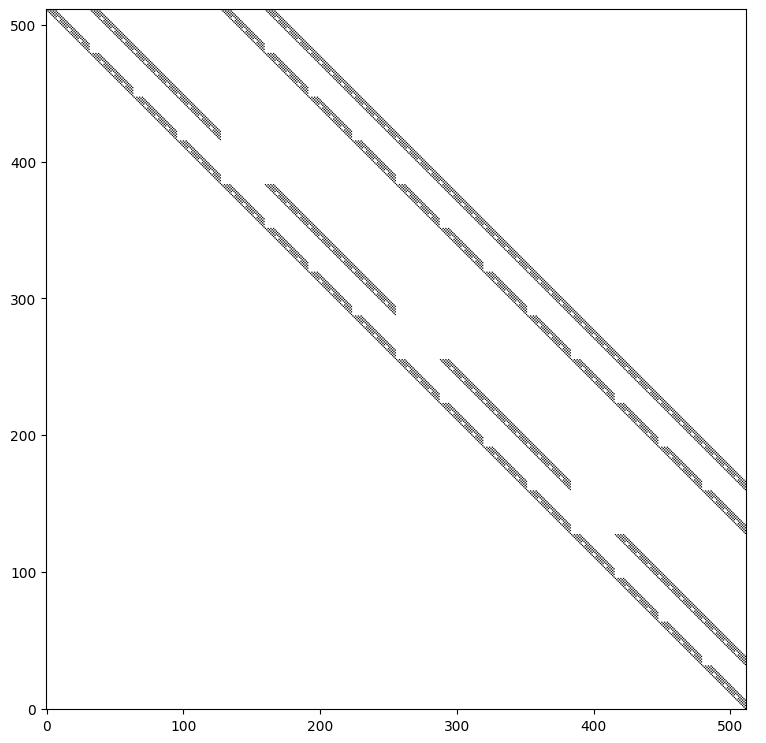}
         \caption{$(x,z)$ such that $0^x10^{344}10^z$ is Ulam}
         \label{fig:344}
     \end{subfigure}
     \hfill
    %  \hspace{-0.9in}
     \begin{subfigure}[b]{0.32\textwidth}
         \centering
         \includegraphics[width=\textwidth]{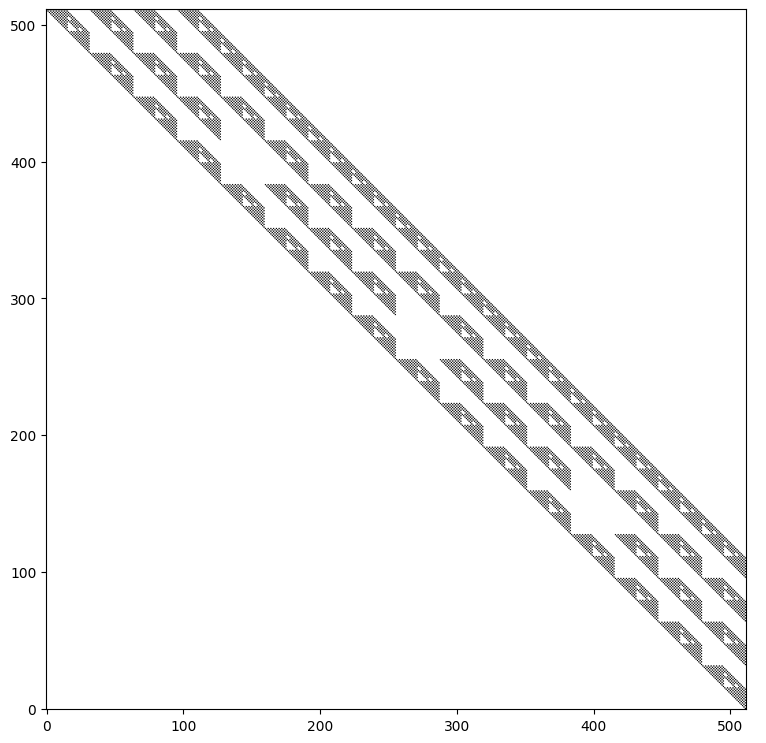}
         \caption{$(x,z)$ such that $0^x10^{400}10^z$ is Ulam}
         \label{fig:400}
     \end{subfigure}
        \caption{Fractal-type structures arising in the study of Ulam words with two $1$'s. For complete descriptions of these phenomena, see \Cref{hierarchy}.}
        \label{fig:intro ulam}
\end{figure}
% Ulam words are similar to the classical Ulam sequence in many ways. For example, both have (conjectured) nonzero asymptotic density - the Ulam sequence has one approximately equal to $0.07398$ and the binary Ulam sequence has one approximately equal to $0.2.$
Observe that the set of Ulam words is constructed in batches. This is possible because two words of lengths $a,b>0$ can only be concatenated to form a word of length strictly greater than both $a$ and $b.$ Thus, the Ulamness of all words of any given length $n>0$ is independent. As we will see, it is reasonable to consider other types of batches to construct Ulam words. Some batch decompositions do have elements within one batch interacting with one another, but in a much more controlled fashion. For example, if we layer words based on the number of $1$'s (or $0$'s) that they contain, then the Ulamness of a word $w$ with $k$ $1$'s depends only on that of words with fewer $1$'s and two special words with $k$ $1$'s---one formed by removing the leftmost $0$ of $w$ and the other formed by removing the rightmost $0$ (\Cref{main-theorem}).

Ulam words differ in many ways from the classical Ulam sequence. For instance, the above batch-like structure of Ulam words allows for easy discovery of classes of words that are known to be or not to be Ulam (e.g., palindromes of odd length \cite{student-paper}), whereas Ulam numbers can only be constructed linearly, and thereby it is much harder to determine any of their persistent or forbidden properties.

% At the same time, however, there is the idea of `thinning' that both the classical Ulam sequence and Ulam words undergo. The Ulam sequence begins fairly densely populated but terms become increasingly sparse, so much so that Ulam initially conjectured that the sequence's asymptotic density is zero \cite{ulam-paper}. Since then, however, further calculations seem to imply that the density is actually nonzero, equal to approximately $0.07398$ \cite{ulam-gaps}. Similarly, the set of Ulam words undergoes a thinning. Words of length near $20$ are concentrated with a density around $0.2,$ whereas a length of $40$ already decreases the density to $0.17,$ and a length of $80$ decreases it only further to $0.13.$ A length of $200$ decreases the density to $0.09$ (we calculated the four, previously unknown, densities via a probabilistic algorithm, see \Cref{conjecture}). Perhaps, as with the Ulam sequence, further calculations may lead us to conclude that the density is in fact nonzero. However, our empirical observations suggest that Ulam words in fact decrease polynomially in density, which may be related to fractal shapes that appear in their study (as we will discuss further). %{\color{red}{perhaps remove this part, add to either preliminary information or analysis after proving sierpinski thing}}

The layers inherent to the binary Ulam set suggest that it may be possible to explicitly describe its elements. However, even the counts of Ulam words of each given length do not seem to follow any known patterns \cite[A337361]{oeis}. Nonetheless, restricting to specific types of words reveals very interesting properties of the binary Ulam set. For example, we may choose to focus on the number of instances of a letter in Ulam words. Note that Ulamness is preserved under the symmetries of reversing and taking complements, so we may consider specifically the count of $1$'s. A word with no $1$'s (or no $0$'s) is Ulam if and only if its length is $1$ (i.e., it is $0$ or $1$). More interestingly, a word with exactly one $1,$ i.e. of the form $\underbrace{00\cdots0}_m1\underbrace{0\cdots00}_n=:0^m10^n$ is Ulam if and only if $\binom{m+n}{m}$ is odd \cite{student-paper}, or, equivalently, if the binary representations of $m$ and $n$ have no entries both equal to $1$ (\Cref{lucas}). By symmetry, this also gives a characterization of all Ulam words with a single $0.$ We also see that this provides an algorithm of logarithmic complexity (based on the length of a word) to calculate whether the given word of such a type is Ulam. 

In this paper, we consider the significantly more complicated case of Ulam words with two $1$'s. We discover and fully describe fascinating hierarchical structures embedded within the study of these words, as well as a family of exceptions that we are also able to completely describe. We construct an algorithm to completely classify all Ulam words with two $1$'s, which works in logarithmic time to determine the Ulamness of any particular word, a significant improvement over the $O(n^3)$ time required for general words of length $n.$% (\Cref{conjecture}).

As it turns out, Ulam words with two $1$'s can be determined more simply than just checking all possible concatenations. In fact, aside from $0$'s appended from either side, the Ulamness of a word with two $1$'s is determined by splitting it between the two $1$'s and checking the resulting Ulamness of the two words with one $1$ each.

Then, in order to gain a better understanding of general Ulam sets, we can look at them in finer detail---through what we call \textit{c-Ulam labelings}. Precisely, $\culam[y]:\N^2\to\{0,1,2\}$ labels every coordinate $(x,z)$ with one of three `caveman' labels: $0,$ meaning zero representations of $0^x10^y10^z$ as a concatenation of Ulam words with one $1$; $1,$ meaning one such representation; and $2,$ meaning two or more. %Note that we also consider in this coloring representations by concatenations of equal words. %It is also simple in most cases to go from a $\culam$ labeling to the actual Ulam set; if $y$ is not a Zumkeller number, then $\ulam[y]$ is simply the set of $(x,z)$ such that $x+z$ is odd and $\culam[y](x,z)=1.$ Here are examples of some $\culam$ colorings (with a square marked $0$ being colored in white, $1$ in grey, and $2$ in black):
% In order to state the most general theorems, we look at the Ulam set graphically. We fix a $y$ and label each point $(x,z)$ in $\N^2$ with an integer from $0$ to $2,$ indicating how many representations the word $0^x10^y10^z$ has as a concatenation of Ulam words with one $1$ ($2$ or more representations are denoted by $2$). Call this labeling $\textit{c-Ulam}[y].$ 
\begin{remark}
The word `culam' (\cjRL{kwlm}) means `everybody' in Hebrew. This can be interpreted as sets $\culam[y]$ being responsible for the entire rich structure of Ulam words with two $1$'s.
%us finding a complete description of $\culam[y]$ for all $y.$
\end{remark}
Let $\textit{Ulam}[y]\subset\N^2$ be the set of points $(x,z)$ for which the corresponding word $0^x10^y10^z$ is Ulam. Then $\ulam[y]$ may seem to consist of $(x,z)$ labeled with $1$ by $\culam[y].$ However, we should also take care of representations either consisting of two equal words or those where one word is $0$ and the other has both $1$'s. It turns out, however, that the structure of $\ulam[y]$ can be easily deduced by that of $\culam[y].$ The exact structure depends on whether or not the binary representation of $y$ possesses a special property.
\begin{definition}\cite[A089633]{oeis}
A positive integer $y$ is called a \textit{Zumkeller number} if it has at most one $0$ in its binary representation.
\end{definition}
\begin{remark}
The name for these numbers was introduced by Shevelev in 2011 \cite{first-zumkeller-paper}.
\end{remark}
\begin{theorem}\label{culam-to-ulam}
    If $y$ is not a Zumkeller number, then $\ulam[y]$ is simply the set of $(x,z)$ where $x+z\equiv1\pmod2$ and $\culam[y](x,z)=1.$
\end{theorem}
% On the other hand, when $y$ is a Zumkeller number, the issue of words needing to be distinct begins to play a role. We fully describe how to determine $\ulam[y]$ for such $y$ in \Cref{special-case}.
Now, it is not difficult to observe that:
\begin{lemma}\label[lemma]{culam-biperiodic}
$\culam[y]$ is biperiodic with biperiod $2^{k+1}\times2^{k+1},$ where $2^k\leq y<2^{k+1}.$
\end{lemma}
Yet, again, it is not immediately evident how we can translate this to $\ulam[y].$ The next theorem shows us that indeed we can, but the exact result depends on whether $y$ is Zumkeller.
\begin{theorem}\label{intro-biperiodic}
Suppose that $2^k\leq y<2^{k+1}.$
\begin{itemize}
    \item The set $\ulam[y]$ is biperiodic if and only if $y$ is not a Zumkeller number, and in this case it has a biperiod of $2^{k+1}\times2^{k+1}.$
    \item If $y$ is a Zumkeller number, then $\ulam[y]$ is eventually biperiodic with a biperiod of $2^{k+1}\times2^{k+1}$; moreover, the only aperiodicities (called ``impurities") occur in the bottom-left block: $(x,z)$ satisfying $x,z<2^{k+1}.$
\end{itemize}
    % Suppose that $y$ is not a Zumkeller number and $2^k\leq y<2^{k+1}.$ Then the set $\ulam[y]$ is biperiodic with biperiod $2^{k+1}\times2^{k+1}.$
\end{theorem}
% Now that $\ulam[y]$ has been fully described for all Zumkeller numbers $y,$ we can fully ignore this case and assume that the set is fully biperiodic.
%In particular, together these two cases imply biperiodicty or almost-biperiodicity of $\ulam[y],$ a result we prove separately in \Cref{biperiodic}. 
% Almost-biperiodicity (except for the bottom-leftmost block) also holds in the case when $y$ is a Zumkeller number, which is once again fully understood by \Cref{special-case}. 
This biperiodicity is useful in a number of other properties, one of which is particularly intriguing.
\begin{theorem}\label{odd-even-equal-intro}
If $y$ is odd and not a Zumkeller number, then $\ulam[y]=\ulam[y-1].$ These sets also coincide for Zumkeller $y$ outside of points $(x,z)$ that are impurities as in \Cref{intro-biperiodic}.
\end{theorem}
This result is surprising because the corresponding $\culam$ labelings look quite different for odd and even $y,$ as shown in \Cref{self-referencer-odd-even}.
\begin{figure}[H]\label[figure]{self-referencer-odd-even}
    \centering
    \captionsetup[subfigure]{labelformat=empty}
    % \subfloat[\Cref{self-referencer-odd-even}: Here and in the future, squares labeled $0$ are in white, $1$ in gray, and $2$ in black.]{
    \begin{subfigure}[b]{0.4\textwidth}
         \centering
         \includegraphics[width=\textwidth]{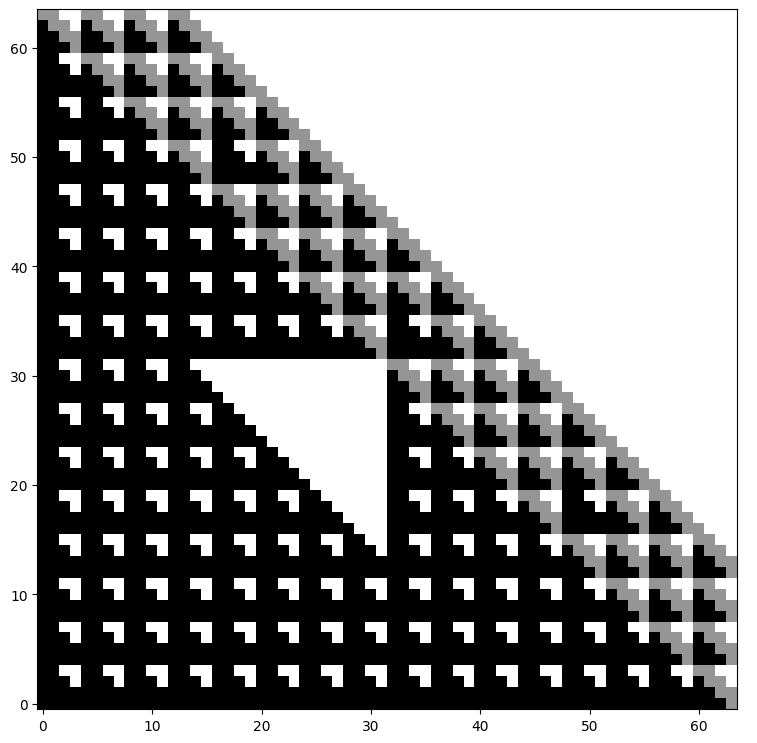}
         \caption{$\culam[50]$}
         \label{fig:culam50}
     \end{subfigure}
    %  \hfill
    %  \hspace{1.0437545in}
     \begin{subfigure}[b]{0.4\textwidth}
         \centering
         \includegraphics[width=\textwidth]{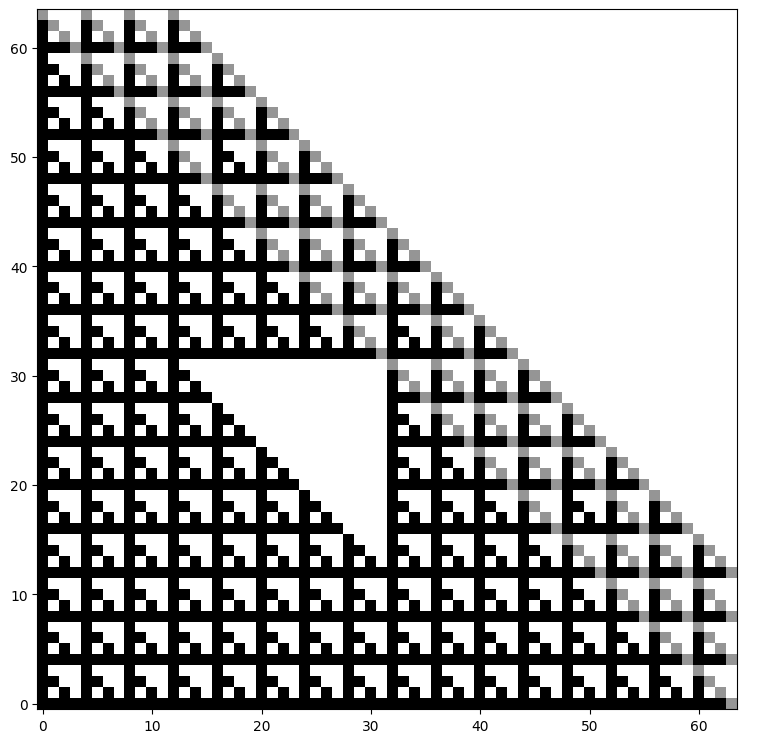}
         \caption{$\culam[51]$}
         \label{fig:culam51}
     \end{subfigure}
    \caption{Here and in the future, squares labeled $0$ are in white; $1$ in gray; and $2$ in black.}
    \label{fig:odd-even}
\end{figure}
\begin{figure}[H]\ContinuedFloat
    \centering
    \captionsetup[subfigure]{labelformat=empty}
%     \noindent\subfloat{
    \begin{subfigure}[b]{0.4\textwidth}
        \centering
        \includegraphics[width=\textwidth]{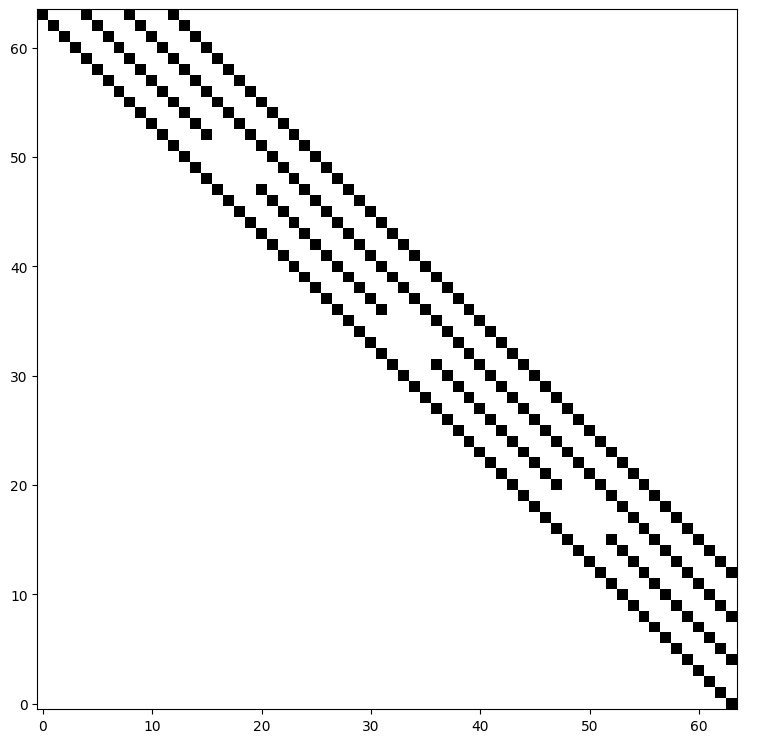}
        \caption{$\ulam[50]$}
        \label{fig:50}
    \end{subfigure}
    % \hfill
    % \hspace{-1.0437545in}
    \begin{subfigure}[b]{0.4\textwidth}
        \centering
        \includegraphics[width=\textwidth]{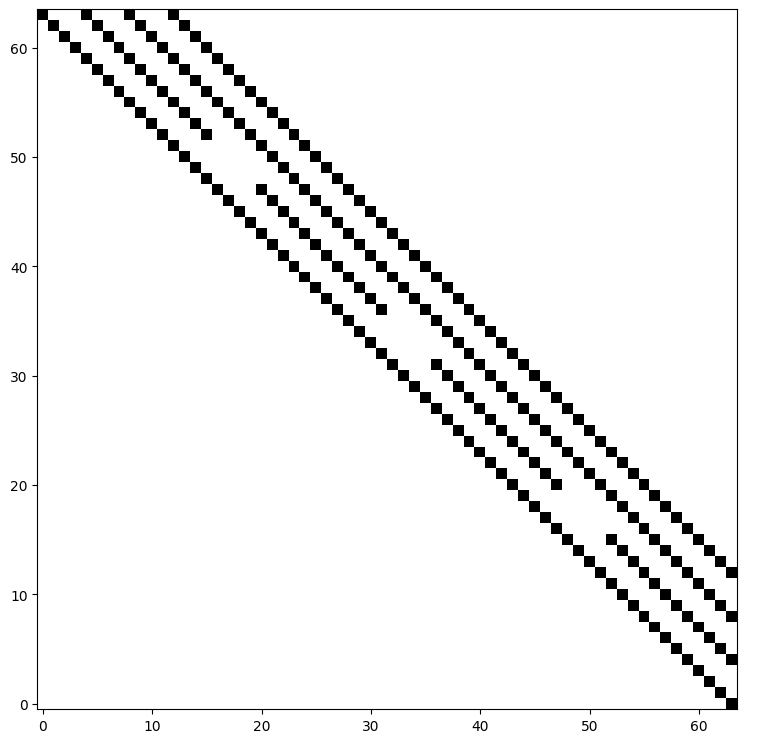}
        \caption{$\ulam[51]$}
        \label{fig:51}
    \end{subfigure}
    % }
    % \qquad
\end{figure}
Indeed, more points seem to be labeled $1$ in $\culam[50]$ than in $\culam[51].$ Yet, \Cref{odd-even-equal-intro} implies that the \textit{final} configurations of points with exactly one representation as a concatenation of (distinct) Ulam words are the same.

Note that all above descriptions of $\ulam[y]$ list Zumkeller $y$ with $x,z<2^{k+1}$ as exceptions. This corresponds to the bottom left block of $\ulam[y],$ which we are able to fully describe for all such $y.$
\begin{theorem}[The case of Zumkeller $y$]\label{zumkeller}
    \hfill
    \begin{itemize}
        \item The set $\ulam[2^{k+1}-1]$ consists of $(x,z)$ such that $x+z\equiv-1\pmod{2^{k+1}}$ and $x+z\neq2^{k+1}-1.$
        \item The set $\ulam[2^{k+1}-2]$ consists of $(x,z)$ such that either $x+z\equiv-1\pmod{2^{k+1}}$ and $x+z\neq2^{k+1}-1$; or $x+z=2^{k+1}-2$ and $x,z\equiv0\pmod2$; or $x+z=2^{k+1}$ and $x,z>0.$
        \item Suppose $y=2^{k+1}-2^a-1$ is a Zumkeller number for some $0<a<k.$ Then $\ulam[y]$ is the set of $(x,z)$ such that either $x+z=y$ and $x\pmod{2^{a+1}}<2^a$; or $x+z\equiv-1\pmod{2^{k+1}}$; or $x\pmod{2^{k+1}}+z\pmod{2^{k+1}}=2^{k+1}+2^a-1.$
        % 
        % the reductions $x'$ and $z',$ respectively, of $x$ and $z$ modulo $2^{k+1}$ satisfy either $x'+z'=2^{k+1}-1$ or $x'+z'=2^{k+1}+2^a-1.$
    \end{itemize}
\end{theorem}
Now that $\ulam[y]$ has been fully described for all Zumkeller numbers $y,$ we can fully ignore this case and assume that the set is fully biperiodic.
\begin{figure}[H]
    \centering
     \begin{subfigure}[b]{0.32\textwidth}
         \centering
         \includegraphics[width=\textwidth]{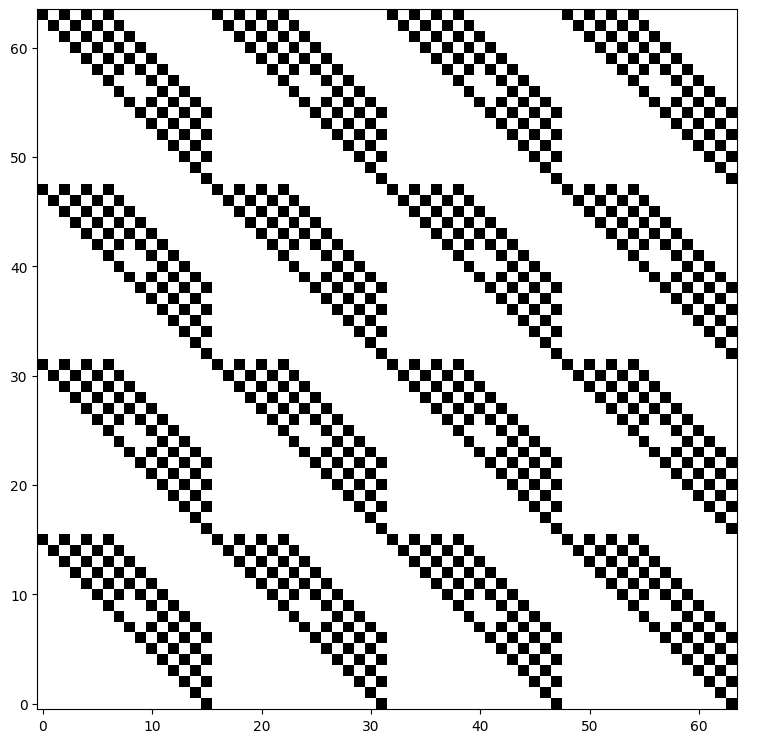}
         \caption{$\ulam[9]$}
         \label{fig:9}
     \end{subfigure}
     \hfill
    %  \hspace{-0.9in}
     \begin{subfigure}[b]{0.32\textwidth}
         \centering
         \includegraphics[width=\textwidth]{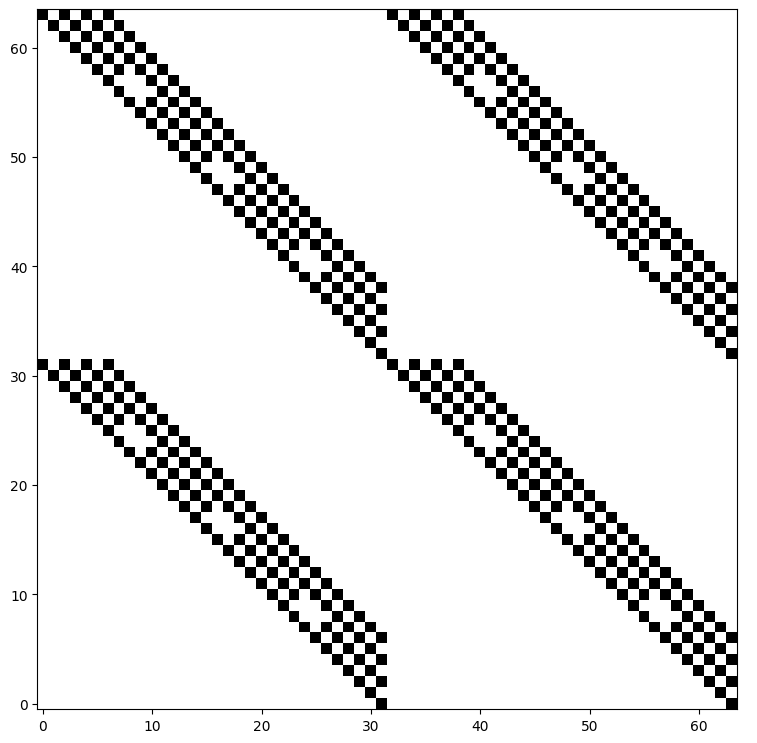}
         \caption{$\ulam[25]$}
         \label{fig:25}
     \end{subfigure}
     \hfill
    %  \hspace{-0.9in}
     \begin{subfigure}[b]{0.32\textwidth}
         \centering
         \includegraphics[width=\textwidth]{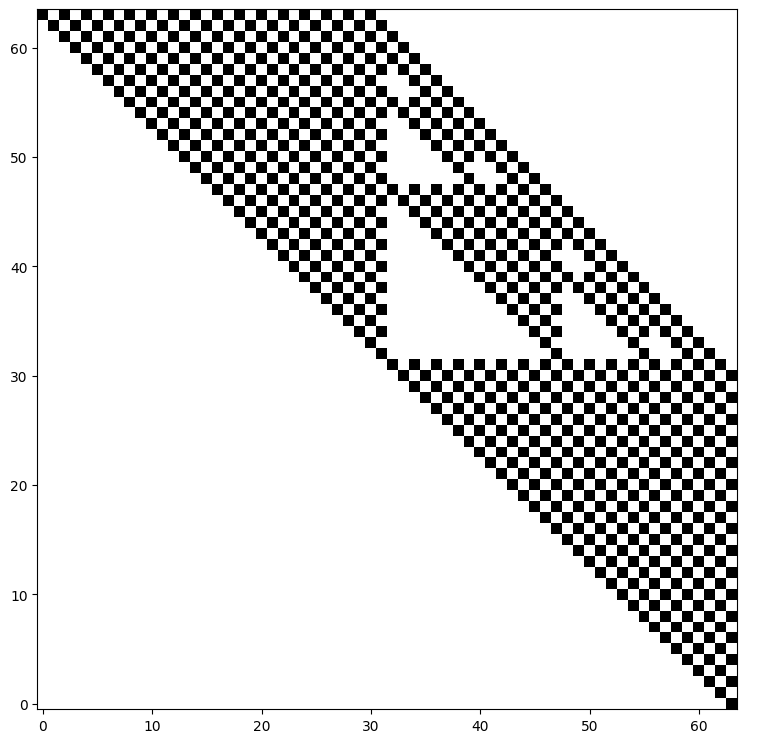}
         \caption{$\ulam[32]$}
         \label{fig:32}
     \end{subfigure}
        \caption{%$\ulam[y]$ for (a) $y=9,$ (b) $y=25,$ (c) $y=32.$}
        Biperiodic blocks of $\ulam[y]$ for various $y.$
        }
        \label{fig:three graphs}
\end{figure}
The more images we look at, the more evident it becomes that a large number of $0$'s in the binary representation of $y$ creates a more fractal-like structure. If there are $1$'s interspersed sporadically between the $0$'s, then the structure becomes even richer, looking like a fractal made out of fractals.

It is clear, then, that $\ulam$ labelings and, even more noticeably, $\culam$ labelings feature interesting fractal patterns. To describe them precisely, we introduce procedures of `adding' and `multiplying' $\culam$ labelings. However, since we are assigning the number $2$ to any word representable as a concatenation of Ulam words in more than one way, we replace simple addition and multiplication by a modified, `caveman'-style arithmetic:
\begin{definition}\label[definition]{caveman-operations}
For integers $m,n\in\{0,1,2\},$ define the operations $$
m\widehat{+}n:=\min(m+n,2)
$$
and $$
m\widehat{\cdot}n:=\min(m\cdot n,2).
$$
\end{definition}
In addition, our hierarchical description of $\culam[y]$ depends fundamentally on splitting into smaller squares and establishing a given structure on each. In particular, we are reminded of the tensor product operation, with $M\otimes N$ being a matrix tiled by blocks of the same dimensions as $N,$ and with $(i,j)$-th entry equal to $m_{ij}N.$ Analogously, we can define $M\widehat{\otimes}N$ being a matrix tiled by blocks $m_{ij}\widehat{\cdot}N.$ 

Since we are now working with tilings by finite matrices, it is useful to define the basic block used in these tilings. Precisely, let $B_d=\{0,1,2,\dots,2^d-1\}^2$ be the fundamental periodic block of $\culam[y]$ (for $y$ with $2^{d-1}\leq y<2^d$ as in \Cref{culam-biperiodic}). The following algorithm is then able to inductively construct $\culam[y],$ implying, in particular, that for any $y,$ Ulam words $0^x10^y10^z$ can also be determined in a logarithmic fashion.
\begin{theorem}\label{inside}
There are universal dyadic patterns $\eone,\etwo,\oone,\otwo$ such that for all $d$ they give labelings $\eone[d],\etwo[d],\oone[d],\otwo[d]:B_d\to\{0,1,2\}$ computable in $O(d)$ time and for any odd $y$ we have $$
\culam[2^dy]=\culam[y-1]\widehat{\otimes}\eone[d]\widehat{+}\culam[y]\widehat{\otimes}\etwo[d]
$$
and $$
\culam[2^dy+1]=\culam[y-1]\widehat{\otimes}\oone[d]\widehat{+}\culam[y]\widehat{\otimes}\otwo[d].
$$
% 
% 
% For every $d,$ there exist four $2^d\times2^d$ blocks, with any coordinate computable in $O(d)$ time, such that for any odd $y,$ the set $\culam[2^dy]$ is a tiling by linear combinations of two of these blocks corresponding to the labelings of respective points in $\culam[y-1]$ and $\culam[y].$ Similarly, $\culam[2^dy+1]$ is a tiling by linear combinations of the two other blocks. This gives rise to a nested fractal structure for $\culam[y]$ for any $y.$
% 
% If $y=2^dy'$ where $y'$ is odd and $2^k\leq y<2^{k+1},$ then aside from some notable exceptions, there exists a $2^m$
% 
% we can use the following reductive algorithm to determine whether $0^x10^y10^z$ is Ulam: \begin{enumerate}
%     \item Define $x',z'$ to be the reductions of $x$ and $z$ modulo $2^{k+1}$; $x''=\lfloor \frac{x'}{2^{d+1}}\rfloor$ and $z''=\lfloor \frac{z'}{2^{d+1}}\rfloor$; and $x'''$ and $z'''$ to be the reductions of $x$ (or $x'$) and $z$ (or $z'$), respectively, modulo $2^{d+1}.$ 
%     \item If $x''+z''=2^{k-d+1}-\frac{y'+1}{2},$ then $0^x10^y10^z$ is Ulam if and only if $x''',z'''<2^d$ and $0^{x'''+2^d}10^{2^d}10^{z'''+2^d}$ is Ulam.
%     \item Otherwise, assuming inequality, $0^x10^y10^z$ is Ulam if and only if %$x''+z''<2^{k-d+1}-\frac{y'+1}{2}$ and 
%     both words $0^{x'''}10^{2^d}10^{z'''}$ and $0^{x''}10^{y'}10^{z''}$ are Ulam.
% \end{enumerate}
\end{theorem}
The patterns $\eone,\etwo,\oone,\otwo$ are presented in the following table:
\begin{table}[H]\label[table]{odd-even-patterns}
    \centering
    \begin{tabular}{c|cc|c}
        $\eone$ & \raisebox{-.5\height}{\includegraphics[width=0.2\textwidth]{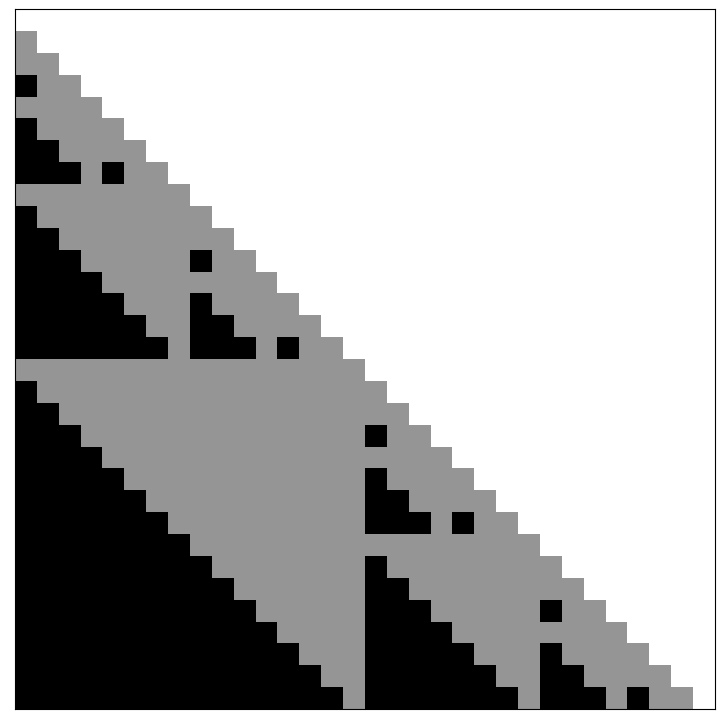}} & $\oone$ & \raisebox{-.5\height}{\includegraphics[width=0.2\textwidth]{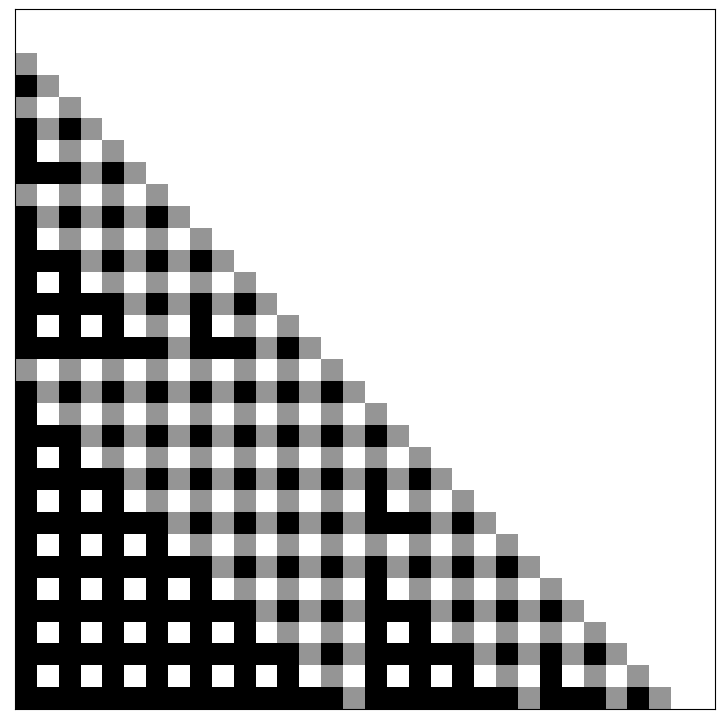}} \\
        $\etwo$ & \raisebox{-.5\height}{\includegraphics[width=0.2\textwidth]{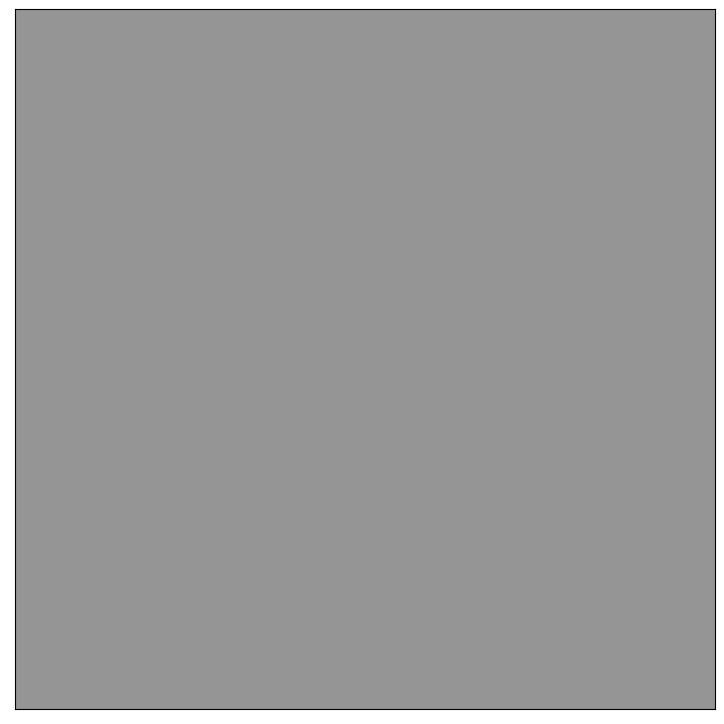}} & $\otwo$ & \raisebox{-.5\height}{\includegraphics[width=0.2\textwidth]{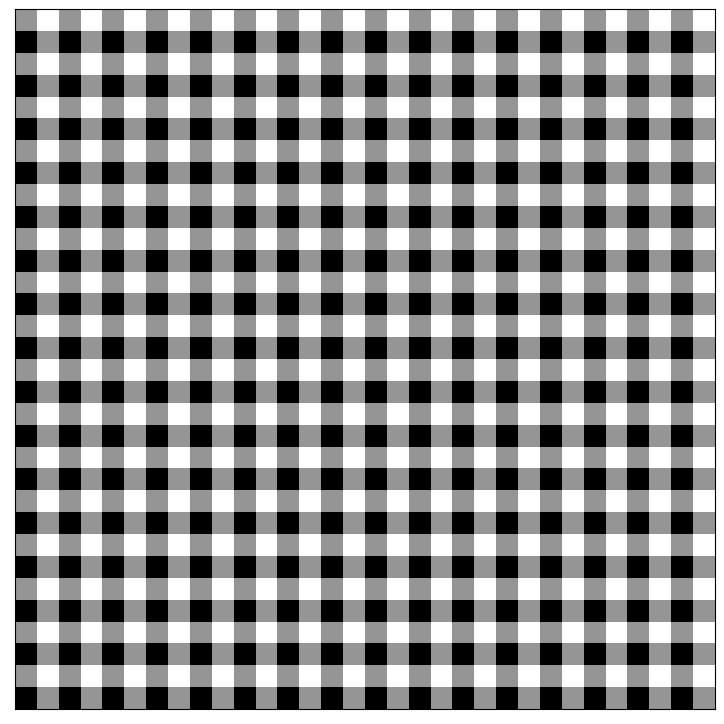}}
    \end{tabular}
    \caption{$E_i[d]$ and $O_i[d]$ for $d=5.$}
    \label{tensor-summands}
\end{table}
\noindent The tensor product operation implies that every combination of two squares in $\culam[y-1]$ and $\culam[y]$ gives two different unique patterns in $\culam[2^dy]$ and $\culam[2^dy+1].$ \Cref{building-blocks} provides the precise conversion chart implied by \Cref{inside}:
\begin{table}[H]%\label[table]{coloring-conversion}
    \centering
    \begin{tabular}{c|c|c|c}
         $\culam[y-1]$ & $\culam[y]$ & $\culam[2^dy]$ & $\culam[2^dy+1]$  \\
         \fcolorbox{black}{white}{\rule{0pt}{1pt}\rule{1pt}{0pt}} & \fcolorbox{black}{white}{\rule{0pt}{1pt}\rule{1pt}{0pt}} & \raisebox{-.5\height}{\includegraphics[width=0.12\textwidth]{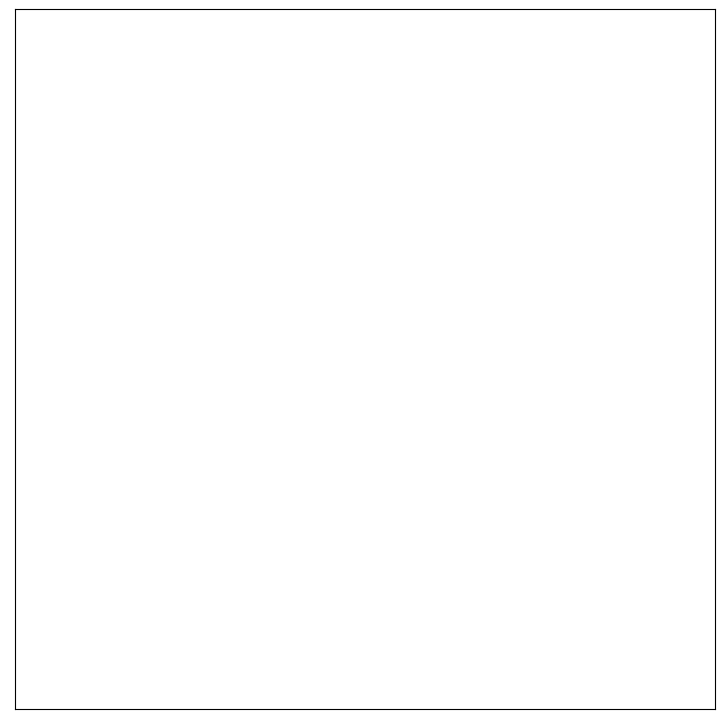}} & \raisebox{-.5\height}{\includegraphics[width=0.12\textwidth]{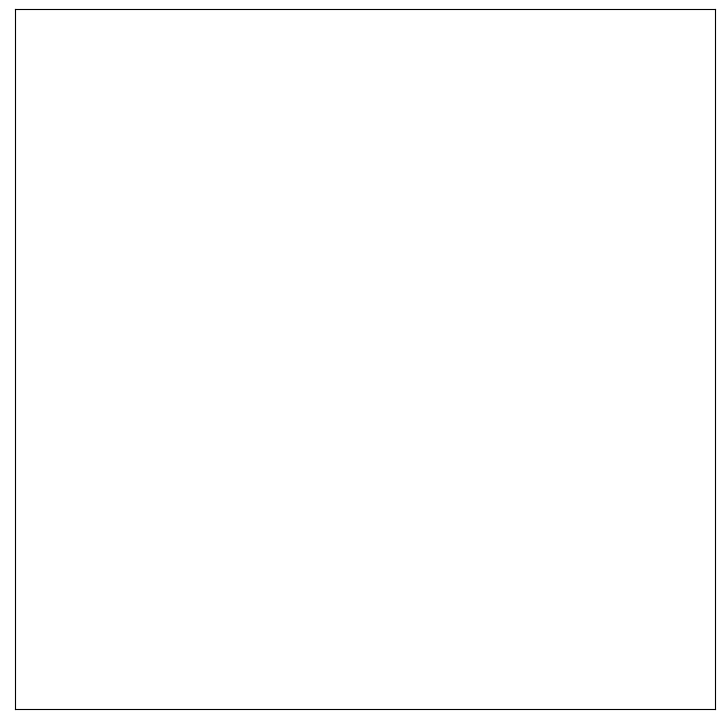}} \\
         \fcolorbox{black}{gray!60}{\rule{0pt}{1pt}\rule{1pt}{0pt}} & \fcolorbox{black}{white}{\rule{0pt}{1pt}\rule{1pt}{0pt}} & \raisebox{-.5\height}{\includegraphics[width=0.12\textwidth]{tile_10_even.png}} & \raisebox{-.5\height}{\includegraphics[width=0.12\textwidth]{tile_10_odd.png}} \\
         \fcolorbox{black}{gray!60}{\rule{0pt}{1pt}\rule{1pt}{0pt}} & \fcolorbox{black}{gray!60}{\rule{0pt}{1pt}\rule{1pt}{0pt}} & \raisebox{-.5\height}{\includegraphics[width=0.12\textwidth]{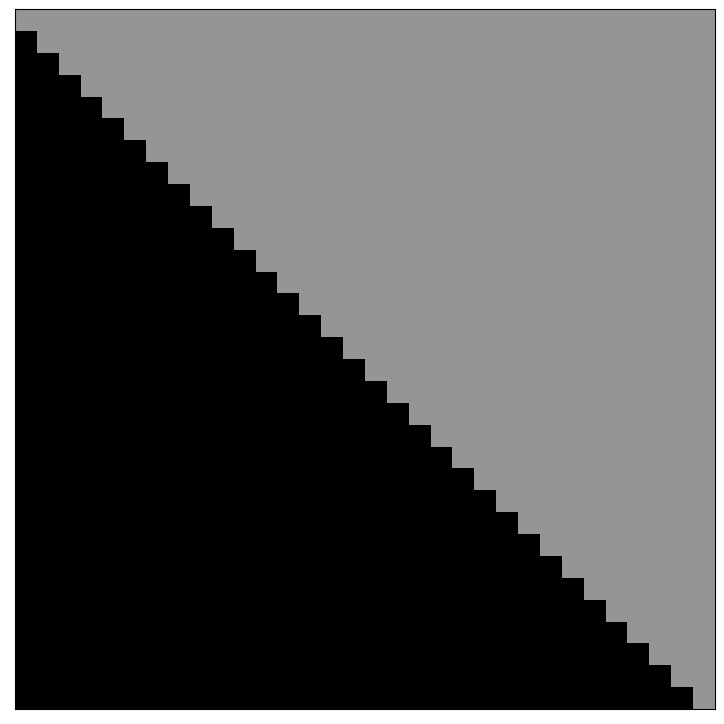}} & \raisebox{-.5\height}{\includegraphics[width=0.12\textwidth]{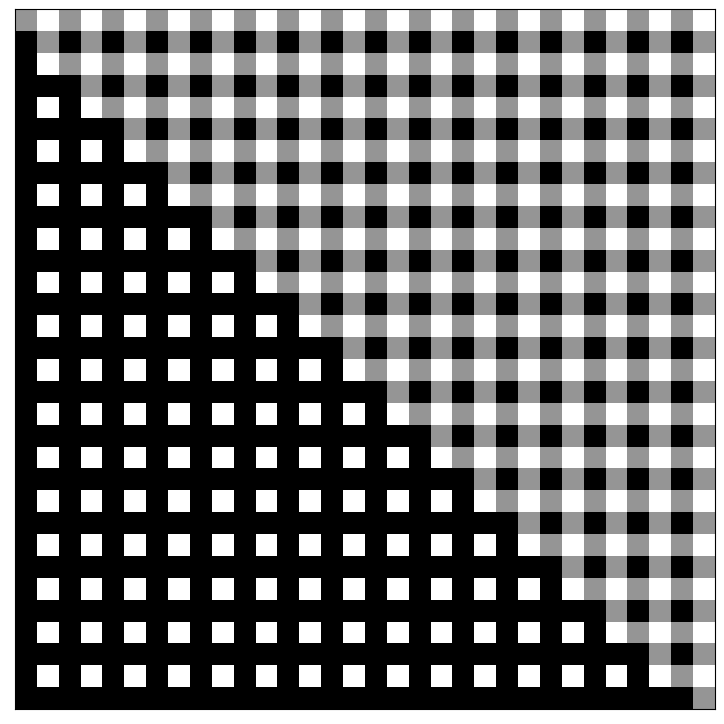}} \\
         \fcolorbox{black}{gray!60}{\rule{0pt}{1pt}\rule{1pt}{0pt}} & \fcolorbox{black}{black}{\rule{0pt}{1pt}\rule{1pt}{0pt}} & \raisebox{-.5\height}{\includegraphics[width=0.12\textwidth]{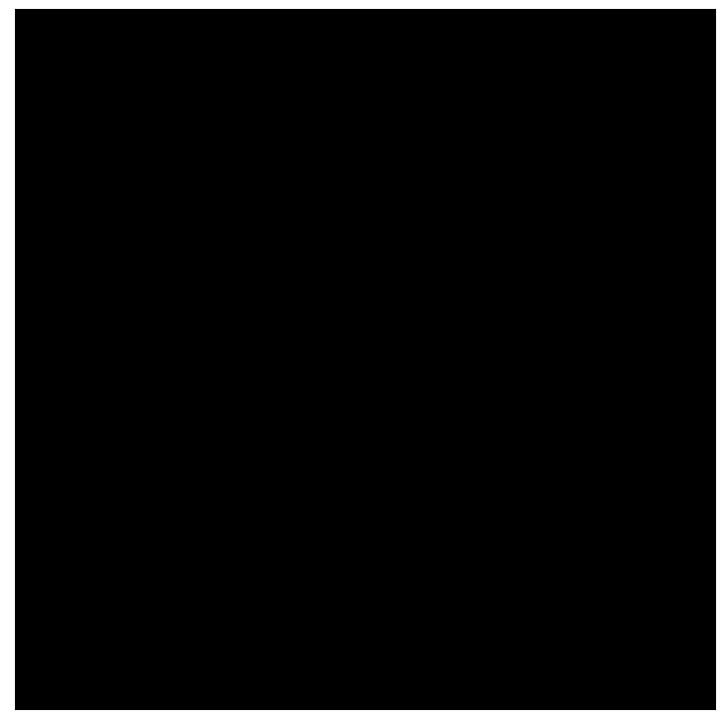}} & \raisebox{-.5\height}{\includegraphics[width=0.12\textwidth]{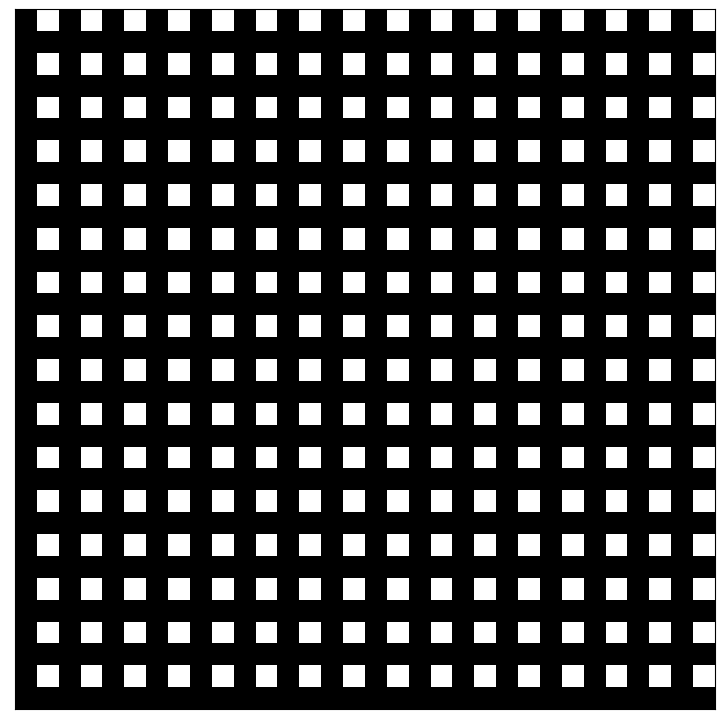}} \\
         \fcolorbox{black}{black}{\rule{0pt}{1pt}\rule{1pt}{0pt}} & \fcolorbox{black}{white}{\rule{0pt}{1pt}\rule{1pt}{0pt}} & \raisebox{-.5\height}{\includegraphics[width=0.12\textwidth]{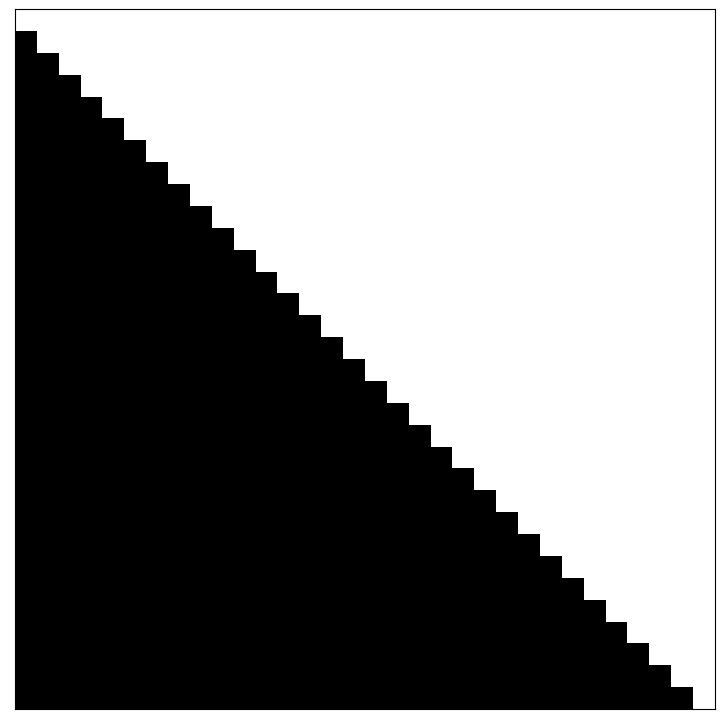}} & \raisebox{-.5\height}{\includegraphics[width=0.12\textwidth]{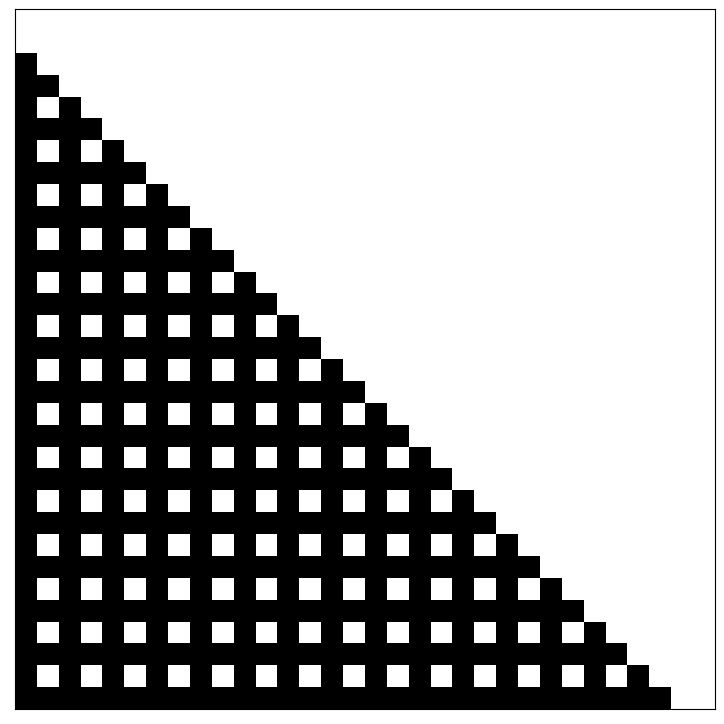}} \\
         \fcolorbox{black}{black}{\rule{0pt}{1pt}\rule{1pt}{0pt}} & \fcolorbox{black}{black}{\rule{0pt}{1pt}\rule{1pt}{0pt}} & \raisebox{-.5\height}{\includegraphics[width=0.12\textwidth]{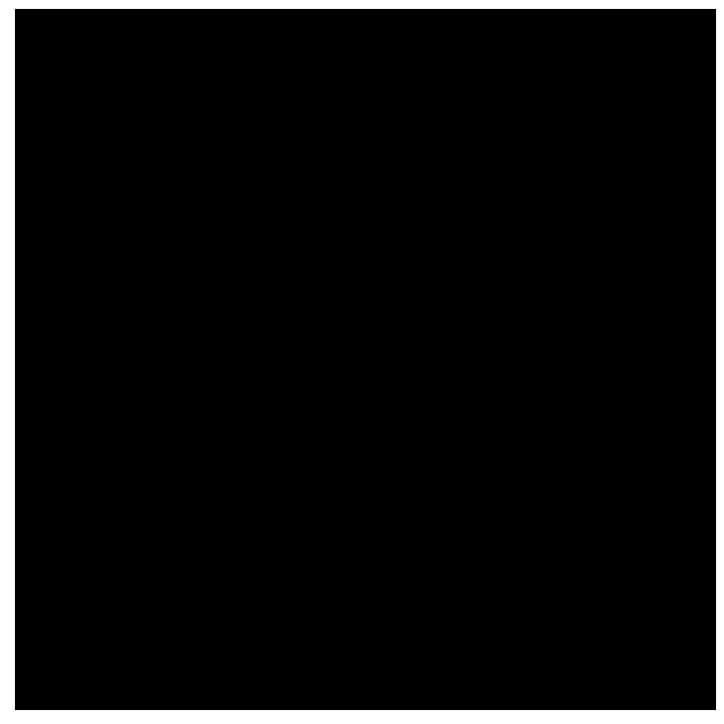}} & \raisebox{-.5\height}{\includegraphics[width=0.12\textwidth]{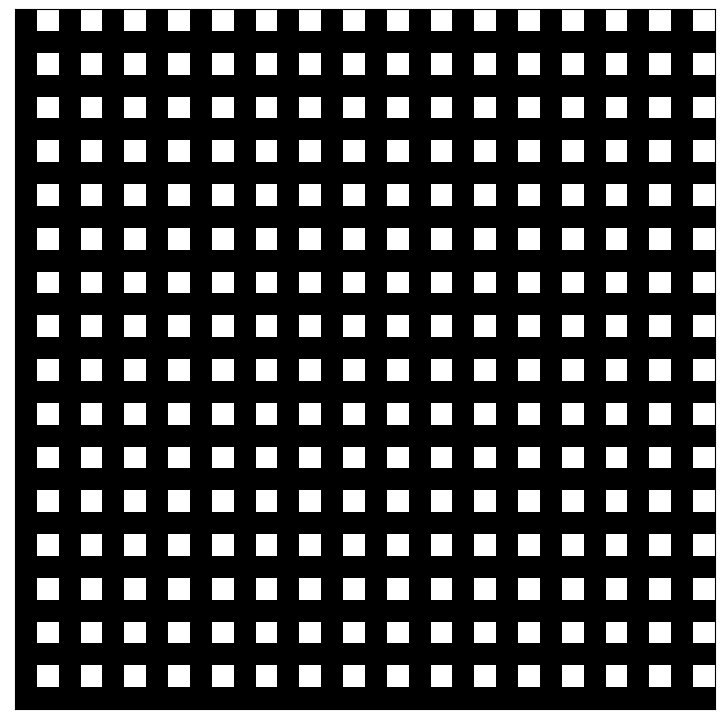}}
    \end{tabular}
    \caption{The labelings of $(x,z)$ in $\culam[2^dy]$ and $\culam[2^dy+1]$ given the labelings of $(x',z')$ in $\culam[y]$ and $\culam[y-1]$; $x'=\lfloor\frac{x}{2^d}\rfloor,$ $z'=\lfloor\frac{z}{2^d}\rfloor.$ Case shown: $d=5.$}
    \label{building-blocks}
\end{table}
For example, the periodic patterns of $\culam[2^d]$ and $\culam[2^d+1]$ are tiled by three blocks in the third and fourth columns, respectively:
\begin{figure}[H]
    \centering
     \begin{subfigure}[b]{0.32\textwidth}
         \centering
         \includegraphics[width=\textwidth]{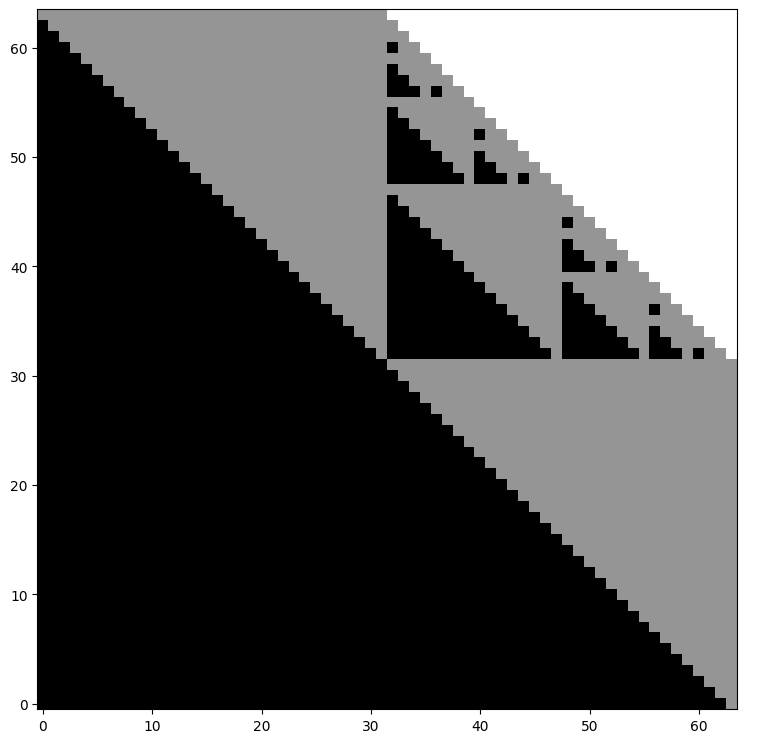}
         \caption{$\culam[32]$}
         \label{fig:power32}
     \end{subfigure}
    %  \hfill
    %  \hspace{-0.9in}
     \begin{subfigure}[b]{0.32\textwidth}
         \centering
         \includegraphics[width=\textwidth]{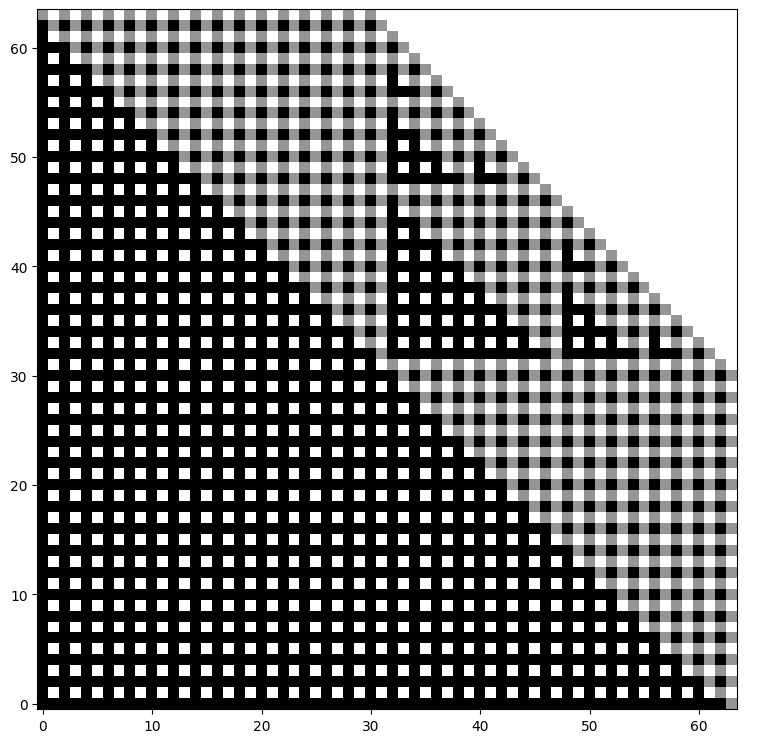}
         \caption{$\culam[33]$}
         \label{fig:power33}
     \end{subfigure}
        \label{fig:2^k-and-2^k+1}
\end{figure}
We can analogously flip our tensor operation. However, we must then restrict ourselves to single periodic blocks.
\begin{theorem}\label{outside}
Suppose $2^k\leq y<2^{k+1}.$ Define the labelings $L,S:B_k\to\{0,1,2\}$ by $$
L(x,z):=\culam[y](x,2^k+z)=\culam[y](2^k+x,z)
$$ 
and $$
S(x,z):=\culam[y](2^k+x,2^k+z).
$$ 
Then we have $$
\culam[2^\ell+y]_{\big|_{B_{\ell+1}}}=\culam[2^{\ell-k}]_{\big|_{B_{\ell-k+1}}}\widehat{\otimes}S\widehat{+}\culam[2^{\ell-k}+1]_{\big|_{B_{\ell-k+1}}}\widehat{\otimes}L.
$$
% where both are defined, i.e. on the set $\{0,1,2,\dots,2^{\ell+1}-1\}^2.$
\end{theorem}
\begin{figure}[H]
    \centering
     \begin{subfigure}[b]{0.55\textwidth}
         \centering
         \includegraphics[width=\textwidth]{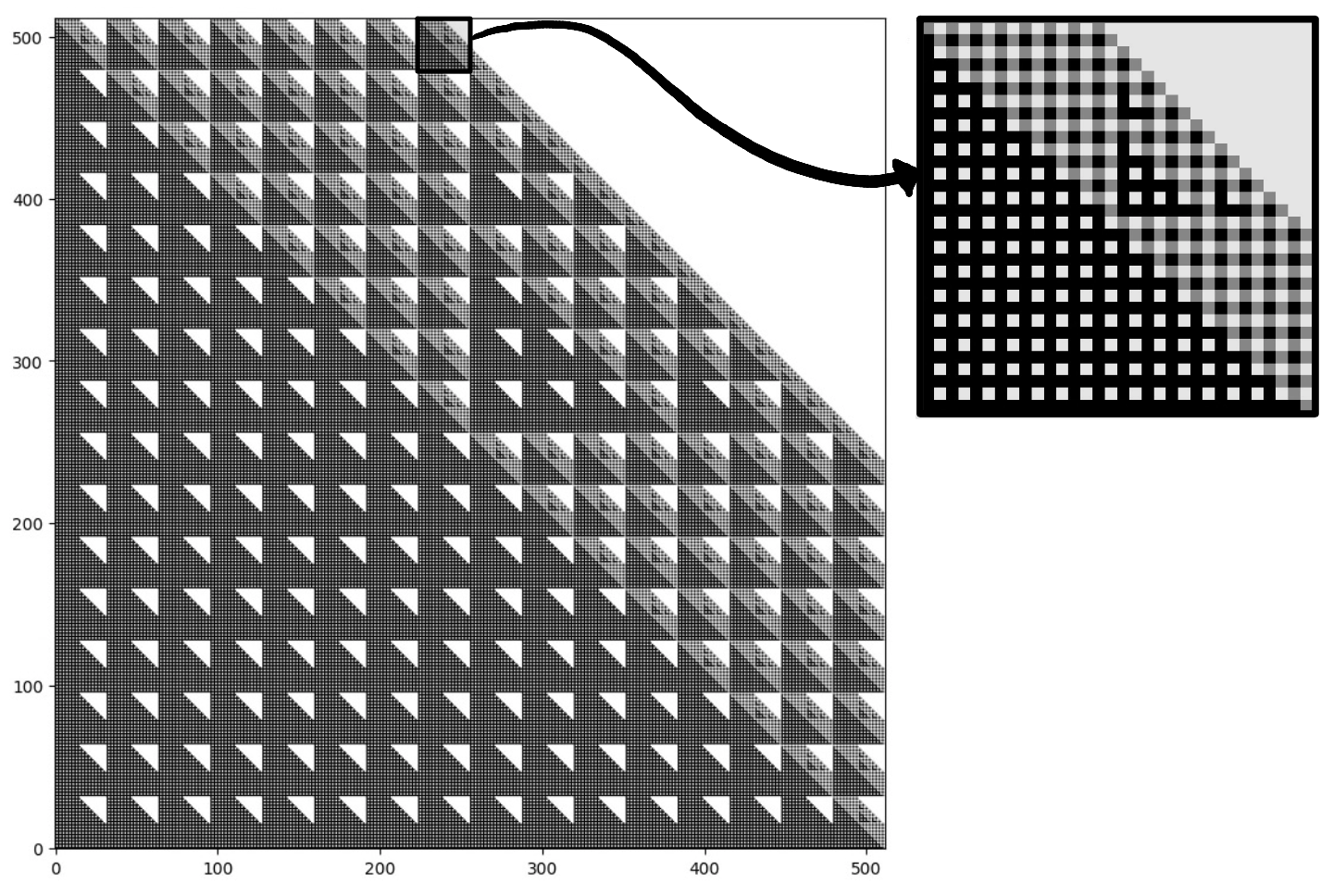}
         \caption{$\culam[273]\quad\quad\quad\quad\qquad\qquad$}
         \label{fig:273}
     \end{subfigure}
     \hspace{0.3in}
    %  \hfill
    %  \hspace{-0.9in}
     \begin{subfigure}[b]{0.35\textwidth}
         \centering
         \includegraphics[width=\textwidth]{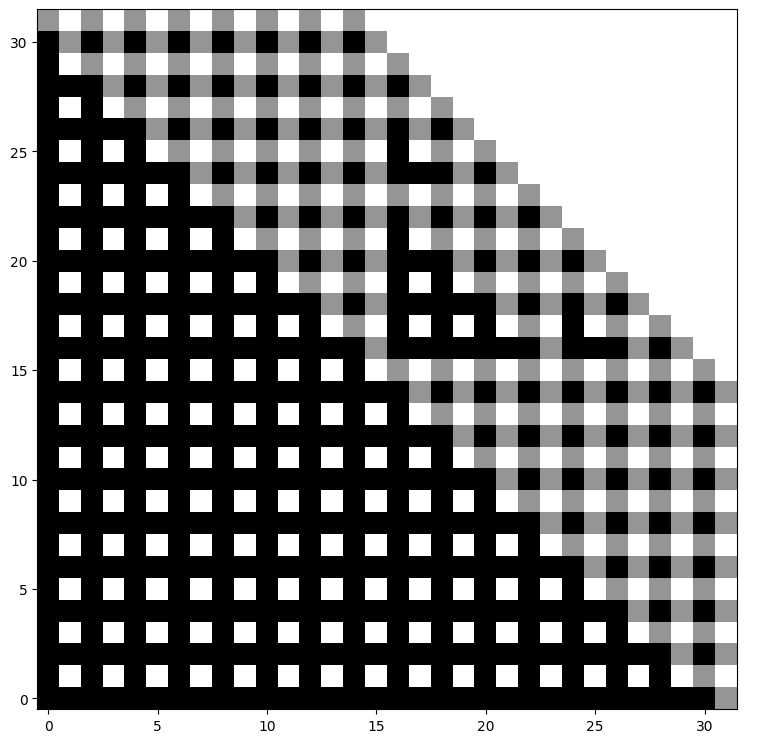}
         \caption{$\culam[17]$}
         \label{fig:17}
     \end{subfigure}
        \caption{Since $273=16*17+1=256+17,$ the resulting labeling $\culam[273]$ has the structure of $\culam[17],$ from both the external and internal viewpoint.}
        \label{fig:double-ended-hierarchy}
\end{figure}
These theorems allow us to also prove several sharp and surprising properties of the structure of $\ulam[y],$ determining simple litmus tests for the set.

\quad
\begin{wrapfigure}{r}{0.4\textwidth}
    \centering
    \vspace{-0.4in}
    \includegraphics[width=0.42549\textwidth]{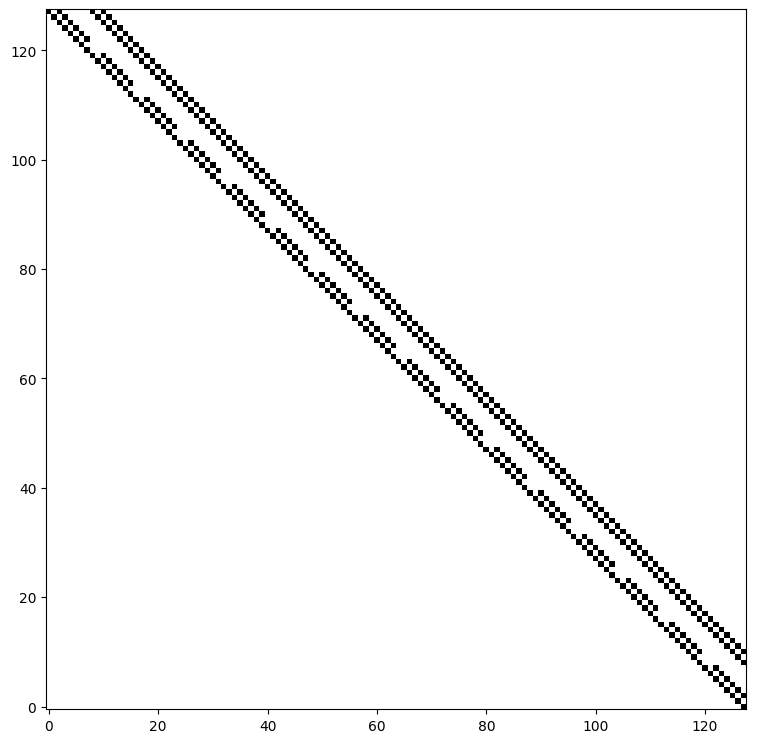}
    \caption{The periodic building block for $\ulam[117].$ Diagonals $(x,z)$ satisfying $x+z=127$ and $x+z=137$ are completely filled, and all other elements are in between them.}
    \label{fig:117}
    \vspace{-0.6in}
\end{wrapfigure}
\vspace{-0.4in}

\begin{prop}\label{consequential-properties}
Suppose $2^k\leq y<2^{k+1},$ and let $x'$ and $z'$ be the reductions of $x$ and $z,$ respectively, modulo $2^{k+1}.$
\begin{enumerate}
    \item If there exist $(x,z)\in\ulam[y]$ with $x+z<2^{k+1}-1,$ then $y$ is a Zumkeller number.\label{not-too-small-improved}
    % 
    % If $y$ is not a Zumkeller number, then $(x,z)\in\ulam[y]$ only if $x+z\geq2^{k+1}-1.$\label{not-too-small-improved}
    \item If $y$ is not a Zumkeller number, then $(x,z)\in\ulam[y]$ only if $2^{k+1}-1\leq x'+z'\leq 2^{k+2}-y-2.$ If $y$ is a Zumkeller number, then the previous inequality is true unless $x+z\in\{y,y+1,y+2\}.$
    \item If $y\not\in\{2^{k+1}-1,2^{k+1}-2\},$ then $(x,z)\in\ulam[y]$ for all $x+z\equiv-1\pmod{2^{k+1}}$ and for all $x'+z'=2\floor{\frac{2^{k+2}-y-3}{2}}+1.$ If $y=2^{k+1}-1$ or $y=2^{k+1}-2,$ then the same result holds except when $x+z=2^{k+1}-1.$
\end{enumerate}
\end{prop}
\begin{remark}
In particular, \Cref{not-too-small-improved} improves a bound set in \cite{student-paper} (see \Cref{not-too-small}).
\end{remark}
Notably, we show the existence of two full diagonals in the periodic block of $\ulam[y]$---one corresponding to $x+z=2^{k+1}-1$ and one corresponding to $x+z=2^{k+2}-2\ceil{\frac{y+1}{2}}-1,$ and show that the entire contents of the block are in between these two diagonals. However, these theorems do not immediately make obvious the beautifully dual hierarchical nature of Ulam sets. To understand it better, we first need to define exactly what we mean by said hierarchy. The precise definitions themselves are somewhat technical and presented fully in \Cref{internal,external}. But in short, an `internal' $i_{d_1,d_2,\dots,d_s}$ hierarchical structure consists of taking entries in a given labeling and `blowing them up'; based on two previous labelings of a given entry, it is replaced by a new (universal) dyadic pattern of size $2^{d_i}\times 2^{d_i}$; this is done inductively for $i$ from $s$ to $1.$ Almost inversely, an `external' $e_{d_1,d_2,\dots,d_s}$ hierarchical structure consists of taking a given pattern, dividing it into four equal blocks, and tiling a given $2^{d_i}\times 2^{d_i}$ block $B$ with them, one empty pattern, and one pattern filling all non-empty labels in the same (in our case black) way, based on the parity and initial labeling of the entries of $B$; this is done inductively for $i$ from $1$ to $s.$ Interestingly, our $\culam$ labeling presents both of these types of hierarchies!
\begin{theorem}[Hierarchical structure]\label{hierarchy}
Suppose that $$
y=\sum_{i=1}^s2^{d_1+\cdots+d_i}=1\underbrace{00\cdots0}_{d_s-1}1\cdots1\underbrace{0\cdots00}_{d_1},
$$
or $$
y=1+\sum_{i=1}^s2^{d_1+\cdots+d_i}=1\underbrace{00\cdots0}_{d_s-1}1\cdots1\underbrace{0\cdots00}_{d_1-1}1,
$$
where $d_i\geq1$ for all $i.$ Then $\culam[y]$ has both $i_{d_1,d_2,\dots,d_s}$ and $e_{d_1,d_2,\dots,d_s}$ hierarchical structure.
\end{theorem}
While the internal hierarchical structure $i_{d_1,d_2,\dots,d_s}$ does make a finite-step fractal shape, after any given step, no label is necessarily fixed, as each small square is blown up to create a large square with a relatively arbitrary labeling. Yet, with the external structure $e_{d_1,d_2,\dots,d_s},$ the situation is different. Instead of being tiled arbitrarily within, the blocks (and their parts) now become the tiles themselves. Continuing this process infinitely, we can in fact make infinitely-hierarchical sets corresponding to \textit{$2$-adic integers}.

Precisely, for any $2$-adic integer $\Tilde{y}\not\in\N$ with approximants $y_i\in\N,$ if we extend the periodic structure of $\culam[y_i]$ into the fourth (or second) quadrant, then they naturally converge to a fractal as $y_i\to_2\Tilde{y}.$ 

Namely, set $Q=\N\times\{-1-\N\}\subset\Z^2$ (where $\{-1-\N\}=\{-1,-2,\cdots\}$). Equip the space $S$ of labelings of $Q$ with a natural metric as in \Cref{metric}. First, suppose $2^k\leq y<2^{k+1},$ and define $\Tilde{U}(y):Q\to\{0,1,2\}$ by $$
\Tilde{U}(y)(x,z)=\culam[y](x,z')
$$ 
for some $z'$ satisfying $z'\equiv z\pmod{2^{k+1}}$ and $z'\geq0$ (this is well-defined because $\culam[y]$ is periodic in any direction with period $2^{k+1}$). For any $\Tilde{y}\in\Z_2\backslash\N,$ define $$
\Tilde{U}(\Tilde{y})(x,z)=\Tilde{U}(y')(x,z)
$$ 
for any $y'\geq x,\abs{z}$ such that $\Tilde{y}\equiv y'\pmod{2^{k+1}}$ for some $2^{k+1}>y'.$ Such a $y'$ can indeed be found because the binary expansion of $\Tilde{y}$ has infinitely many $1$'s.
\begin{theorem}\label{fractal}
    The map $\Tilde{U}:\Z_2\to S$ is well-defined everywhere and continuous at all $y\in\Z_2\backslash\N.$
\end{theorem}
\begin{remark}It is additionally possible to naturally extend the definition of external hierarchical structure to an infinite setting. Then, it is not hard to see that for a $2$-adic number $y=\{0,1\}+2^{d_1}+2^{d_1+d_2}+\cdots,$ the set $\Tilde{U}(y)$ possesses $e_{d_1,d_2,\dots}$ hierarchical structure.\end{remark}
\begin{figure}[H]
    \centering
     \begin{subfigure}[b]{0.3\textwidth}
         \centering
         \includegraphics[width=\textwidth]{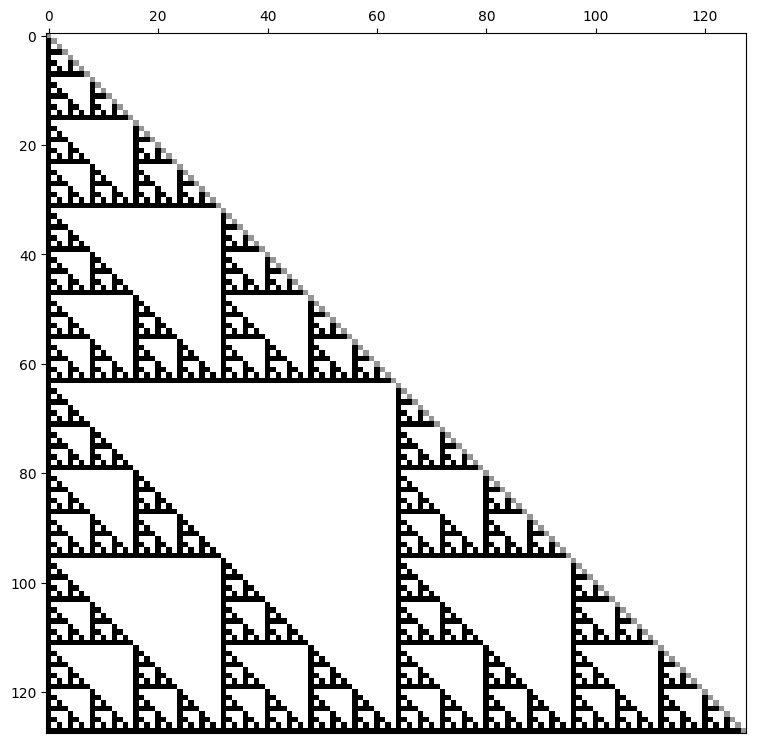}
         \caption{$\Tilde{U}(-1)$}
         \label{fig:-1}
     \end{subfigure}
     \hfill
     \begin{subfigure}[b]{0.3\textwidth}
         \centering
         \includegraphics[width=\textwidth]{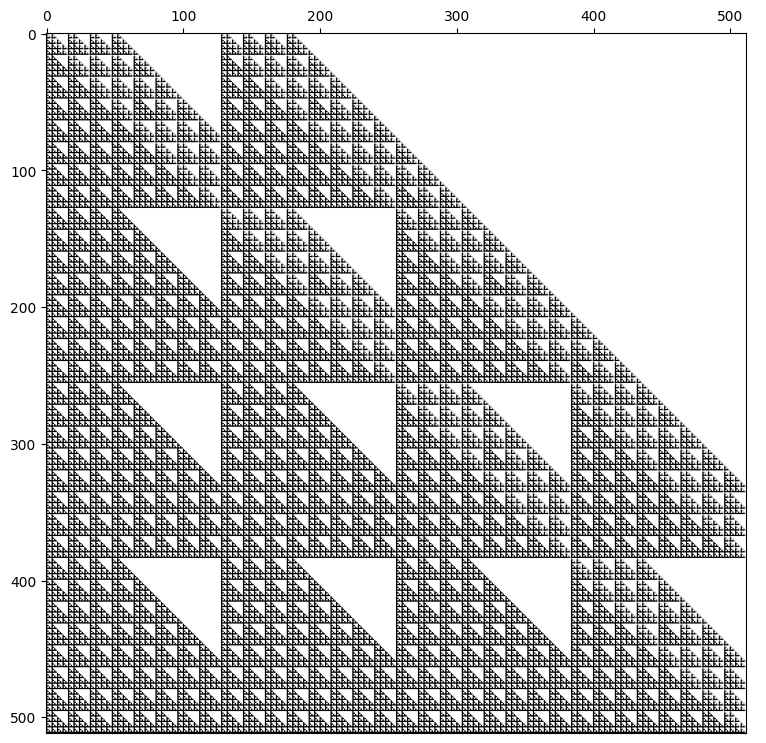}
         \caption{$\Tilde{U}(\sqrt{-7})$}
         \label{fig:sqrt-7}
     \end{subfigure}
     \hfill
    %  \hspace{-0.9in}
     \begin{subfigure}[b]{0.3\textwidth}
         \centering
         \includegraphics[width=\textwidth]{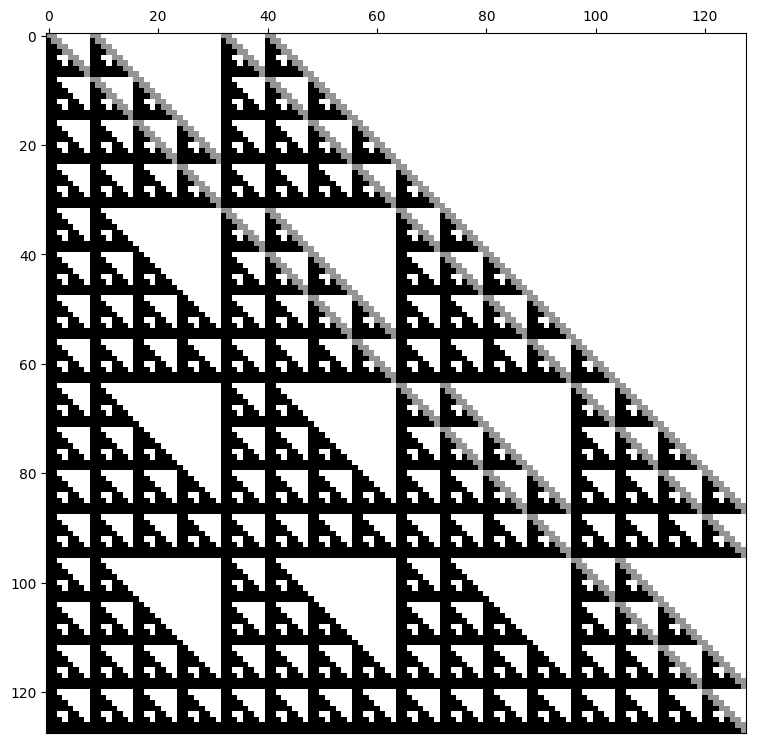}
         \caption{$\Tilde{U}(2/3)$}
         \label{fig:0.66}
     \end{subfigure}
        \caption{Fractal shapes for choices of $y\in\Z_2\backslash\N.$}
        \label{fig:fractals}
\end{figure}
\begin{remark}
The fractal $\Tilde{U}(-1)$ simply gives the label $2$ to all points except the top diagonal, which is labeled with $1,$ of the regular Sierpinski gasket.
\end{remark}

This paper is organized as follows. In \Cref{Background-Section,my-background-section}, we discuss prior relevant results and introduce some definitions. In \Cref{parity-section}, we discuss the curious effect of the update step and how it is essentially a triviality. As a result, this allows us to prove \Cref{culam-to-ulam,odd-even-equal-intro}. However, \Cref{intro-biperiodic} is delayed to \Cref{special-section}, where we fully describe the special case of Zumkeller numbers listed as exceptions in the previous theorems, with our full description given in \Cref{zumkeller}. In \Cref{main-theorem}, we prove the general descriptions of $\culam[y]$ for all $y,$ in particular proving \Cref{inside,outside,hierarchy,fractal}. In \Cref{consequences}, we use the general description to prove several curious facts in \Cref{consequential-properties}.
\section*{Acknowledgements}
The research was conducted at the 2022 University of Minnesota Duluth REU and fully supported by the generosity of Jane Street Capital, the National Security Agency (grant number H98230-22-1-0015), and the Stanford Mathematics Research Center. %{\color{red}(include this?)}.
The author is deeply grateful to Joe Gallian for organizing this REU and for his comments on this paper. A great thank you to Noah Kravitz for suggesting to study Ulam structures and providing a plethora of guidance and support throughout. Thank you also to Amanda Burcroff and Colin Defant, as well as the students and visitors of the Duluth REU, for a great program and research environment. Finally, the author appreciates the suggestions of Stefan Steinerberger and Toufik Mansour on the final version of the paper.
\section{Preliminaries}\label{Background-Section}
We begin with some notation that we will use extensively throughout the paper.
\begin{definition}
For nonnegative integers $x_1,x_2,\dots,x_n,$ let the word $w(x_1,x_2,\dots,x_n)$ consist of $x_1$ $0$'s followed by a $1,$ followed by $x_2$ $0$'s followed by another $1,$ and so on, ending with $x_n$ $0$'s. More concisely, it is $0^{x_1}10^{x_2}1\cdots10^{x_n}.$
\end{definition}
Prior to this paper, Ulam words of type $w(x,y)$ were fully classified, as stated in the introduction: $w(x,y)$ is Ulam if and only if $\binom{x+y}{x}$ is odd. However, this binomial coefficient quickly becomes large and difficult to compute, so we now recall Lucas's criterion:
\begin{prop}[Lucas]\label[proposition]{lucas}
    The binomial coefficient $\binom{x+y}{x}$ is odd if and only if the binary expansions of $x$ and $y$ never both have a $1$ in the same place value.
\end{prop}
Thus, in reality we can simply look at the binary expansions of $x$ and $y$ to determine whether $w(x,y)$ is Ulam. Notably, if $w(x,y)$ is Ulam and either $x$ or $y$ is odd, then the other is even.

Ulam words $w(x,y,z)$ have also been studied, and the following facts are known.
\begin{lemma}[\cite{student-paper}]\label[lemma]{easy-cases}
    The word $w(x,0,z)$ is Ulam if and only if $x+z$ is odd. The word $w(x,1,z)$ is Ulam if and only if $x+z$ is odd and $x+z>1.$
\end{lemma}
\begin{lemma}[\cite{student-paper}]\label[lemma]{not-too-small}
    If $w(x,y,z)$ is Ulam, then $x+z\geq y.$
\end{lemma}
% \subsection{Toolbox: basic lemmas}
Our results, particularly those of \Cref{parity-section}, rely repeatedly on basic facts about the Ulamness of some words in relation to others. Precisely, we have the following useful fact:
\begin{lemma}\label[lemma]{ulam-from-less}
    An even number (zero or two) of the words $w(x,y),w(x,y-1),w(x-1,y)$ are Ulam.
\end{lemma}
\begin{proof}
The identity $\binom{x+y}{x}=\binom{x+y-1}{x}+\binom{x+y-1}{x-1}$ tells us that either zero or two of these binomial coefficients are odd; now recall that %In particular, since 
$w(a,b)$ is Ulam if and only if $\binom{a+b}{a}$ is odd.%, we see that if one of the terms in the identity is odd, then exactly one of the other two is, implying the conclusion.
\end{proof}
\section{Defining hierarchical structures}\label{my-background-section}
Typically, a self-similar fractal structure features a sequence of scales, at each of which the same pattern is repeated. Here, however, the situation is more complex, so we need to introduce more general definitions to describe the hierarchies inherent in $\culam$ labelings. Essentially, the two below definitions correspond to tiling a set in $\N^2$ either from the inside out or from the outside in, and account for several universal patterns of tiling present in the structure.
\begin{definition}[Internal hierarchy]\label[definition]{internal}
Fix some (labeling) set $T$\footnote{In our case, $T=\{0,1,2\}.$}. For any nonnegative integer $d,$ fix maps $I_d^e,I_d^o:T\times T\to\PP^d,$ where $\PP^d$ is a set of some labelings $L:B_d\to T.$ %Typically, we may assume that the labelings $\PP^d=I(T\times T)$ all come from the same ($d$-independent) dyadic patterns, so all belong to a common family $\PP.$
We may typically assume that the labelings $\PP^d$ all belong to a common family $\PP,$ e.g. as in \Cref{building-blocks}.

For any nonnegative integers $d_1,d_2,\dots,d_s,$ a labeling $\LL:B_{d_1+d_2+\cdots+d_s+1}\to T$ exhibits \textit{$i_{d_1,d_2,\dots,d_s}$ hierarchical structure} with respect to $\PP$ if we have $c_j^e,c_j^o:B_{d_j+d_{j+1}+\cdots+d_s+1}\to T,j=1,\dots,s,$ and $c_{s+1}^e,c_{s+1}^o:B_1\to T$ satisfying the inductive relations \begin{align*}
c_j^e(x,z)=I_{d_j}^e(c_{j+1}^e(x',z'),c_{j+1}^o(x',z'))(x'',z'')\\
c_j^o(x,z)=I_{d_j}^o(c_{j+1}^e(x',z'),c_{j+1}^o(x',z'))(x'',z'')
\end{align*}
and $\LL=c_1^e$ or $c_1^o,$ where $x',z'$ are the first (counting $0$'s) $d_{j+1}+d_{j+2}+\cdots+d_s+1$ digits and $x'',z''$ are the last $d_j$ digits, respectively, of $x,z$ in binary. 
\end{definition}
In our definition, $\LL=c_1^o$ occurs in the case when the final hierarchy is `odd'. Specifically, it occurs in $\culam[y]$ when $y$ is odd.
\begin{definition}[External hierarchy]\label[definition]{external}
Fix some (labeling) set $T$ and `empty' label $t\in T,$ as well as `full' label $f\in T.$ Fix a map $I:T\times T\to\{0,1,2,3,4,5\}.$ Also, for every $d$ and $i\in\{0,1,2,3\}$ fix labelings $L_{d,i}^e,L_{d,i}^o:B_d\to T.$

For any nonnegative integers $d_1,d_2,\dots,d_s,$ a labeling $\LL:B_{d_1+d_2+\cdots+d_s+1}\to T$ exhibits \textit{$e_{d_1,d_2,\dots,d_s}$ hierarchical structure} with respect to $L$ if we have, for $i\in\{0,1,2,3,4,5\},$ labelings $c_j^i:B_{d_1+d_2+\cdots+d_j}\to T$ satisfying $c_j^4(x,z)=t$ and $c_j^5(x,z)=\begin{cases}t&c_j^3(x,z)=t,\\f&c_j^3(x,z)\neq t,\end{cases}$ and, for $i<4,$ we have $$
c_j^i(x,z)=c_{j-1}^{I(L_{d_j,i}^e(x'',z''),L_{d_j,i}^o(x'',z''))}(x',z')
$$ 
and $$
c_s^0(x,z)=\LL(x,z),\quad c_s^1(x,z)=\LL(\overline{1x},z),\quad c_s^2(x,z)=\LL(x,\overline{1z}),\quad c_s^3(x,z)=\LL(\overline{1x},\overline{1z})
$$ 
where $x',z'$ are the trailing $d_1+\cdots+d_{j-1}$ digits and $x'',z''$ are the first $d_j$ digits (counting $0$'s), respectively, of $x,z$ in binary, and $\overline{1x}=2^{d_1+d_2+\cdots+d_s}+x$ (resp., $\overline{1z}=2^{d_1+d_2+\cdots+d_s}+z$) is the binary number defined by inserting a $1$ in the $(d_1+d_2+\cdots+d_s+1)$-th slot from the right of the binary representation of $x.$ 
\end{definition}
In our case, $T$ is simply the set $\{0,1,2\},$ and the `trivial' label is $t=0,$ while the `full' label is $s=2.$ The labelings $\PP^d$ are exactly those in \Cref{building-blocks}, and $L_{d,i}^e$ and $L_{d,i}^o$ each give four labelings of $B_d$ partitioning $\culam[2^d].$ Precisely, they are the fourth, third, third, and second figures, respectively, of the third and fourth columns, respectively, of \Cref{building-blocks}.

Recall from the introduction that \Cref{external} allows us to construct an infinite, self-similar fractal for every $2$-adic integer. However, it is essential that we define a form of convergence for these fractal shapes.
\begin{definition}\label[definition]{metric}
Set $Q=\N\times\{-1-\N\}.$ For a given set $T,$ define a metric $\rho$ on the set of labelings $Q\to T$ as follows: $\rho(L_1,L_2)=\frac{1}{d}$ where $d$ is the largest integer (perhaps $\infty$) such that $L_1(m,n)=L_2(m,n)$ for all $m,\abs{n}\leq d.$
\end{definition}
\section{Parity phenomena in \texorpdfstring{$\ulam[y]$}{}}\label{parity-section}
First, we discuss a more understandable general procedure to determine whether a word of the form $0^x10^y10^z$ is Ulam. Observe that a sequence of $0$'s (or $1$'s) is Ulam if and only if it has length exactly $1.$ An Ulam word must be expressible uniquely as a concatenation of two smaller (Ulam) words, so we must understand in what ways this can happen. The words are determined by `slicing' $w$ between two letters: $$ 
w=00\cdots010\cdots0\mid0\cdots010\cdots0.
$$ 
By the previous reasoning, the slice may not be past the second $0$ but before the first $1,$ or past the second $1$ but before the second-to-last $0.$ Thus, the slice is either directly after the first $0,$ directly before the last $0,$ or somewhere in between the two $1$'s. In the former two cases, the word is formed from a $0$ and another word of the form $0^{x'}10^y10^{z'}$ (where either $x'=x-1$ and $z'=z,$ or $x'=x$ and $z'=z-1$; if $x=0$ or $z=0$ then one or both of these cases may disappear). The latter case, however, is more complicated. There are $y+1$ total positions for the slice, each of which splits the $y$ $0$'s into two groups: one of size $a$ and another of size $y-a$ (where either may be $0$). The corresponding words are then $0^x10^a$ and $0^{y-a}10^z.$ Note also that we require the words to be distinct. If one of the words is $0$ and the other contains both $1$'s, then this is not an issue. However, we may run into problems if $w(x,a)=w(y-a,z),$ which happens if $x=y-a$ and $a=z.$ This implies that $x+z=(y-a)+a=y.$ Note that this is the equality case of \Cref{not-too-small}. It seems from the outset as if both words in the concatenation having a single $1$ dominates the arguments, and indeed it does; the remaining two possible representations can in fact be considered separately. Observe that in all of these representations, either the words in the concatenation representation have one $1$ each, or the number $y$ of $0$'s in between the two $1$'s does not change. Thus, it makes sense to layer Ulam words with two $1$'s based on the distance between the $1$'s; separate layers will not interact whatsoever.

Whether $w(x,y,z)$ is an Ulam word then depends, inductively, only on several previous words: $w(x-1,y,z)$ and $w(x,y,z-1),$ as well as $w(x,a)$ and $w(y-a,z)$ for $0\leq a\leq y.$ Instead of considering all of these words as once, we can split the process of determining Ulam words in a \textit{calculation step} and an \textit{update step}. 

The calculation step will incorporate only the words $w(x,a)$ and $w(y-a,z)$ in the determination. After this, each word $w(x,y,z)$ is expressible as a concatenation in either zero, one, or at least two ways. If there are at least two expressions, then the word can be immediately discarded; there is no chance for it to be Ulam, even after considering $w(x-1,y,z)$ and $w(x,y,z-1).$ If, on the other hand, there is exactly one representation, then the word is Ulam if and only if neither $w(x-1,y,z)$ nor $w(x,y,z-1)$ is Ulam. Finally, if there are no representations, then the word is Ulam if and only if exactly one of $w(x-1,y,z)$ and $w(x,y,z-1)$ is Ulam; we will show, in particular, that this case in fact never occurs. 

All in all, we can therefore create a tentative scheme for Ulam and non-Ulam words:
\begin{definition}\label{types-of-ulam}
The word $w(x,y,z)$ is \textit{non-Ulam} if there are at least two representations prior to considering $w(x-1,y,z)$ and $w(x,y,z-1)$; it is \textit{pseudo-Ulam} if there is exactly one representation, and it is \textit{probably not} \textit{Ulam} if there are no representations. 
\end{definition}
\begin{remark}
The choice to label the final case of words $w(x,y,z)$ with no representations as \textit{probably not} Ulam is due to the fact (that we will prove later) that they are, in fact, never Ulam.
\end{remark}
In the update step, we inductively determine whether $w(x',y,z')$ is Ulam for all $x'\leq x,z'\leq z$ satisfying $x'+z'<x+z,$ and then $w(x,y,z)$ is Ulam if and only if either it is pseudo-Ulam and $w(x-1,y,z)$ and $w(x,y,z-1)$ (if they exist) are non-Ulam, or it is probably not Ulam and exactly one of $w(x-1,y,z)$ and $w(x,y,z-1)$ is Ulam.

The first calculation step of the algorithm unfortunately has the caveat that it has to consider the case when $w(x,y,z)=w(x,z)*w(x,z)$ is split into two equal Ulam words, violating the condition that the words must be distinct, which slightly derails some further arguments. Nevertheless, we can in fact significantly bound the instances when this affects further calculations.
\begin{lemma}\label[lemma]{x+z-case}
Suppose $y=x+z$ is not Zumkeller. If $w(x,z)$ is Ulam, then there are at least two decompositions of $w(x,y,z)$ as a concatenation of distinct Ulam words $w(x,a)*w(y-a,z).$ In particular, $w(x,y,z)$ is not Ulam and $\culam[y](x,z)=2.$
\end{lemma}
\begin{proof}
    As noted above, $x$ and $z$ share no $1$'s in their binary representation. By assumption, $y$ has at least two $0$'s in its binary representation and hence is of the form $$
    y=c_mc_{m-1}\cdots c_10b_\ell b_{\ell-1}\cdots b_10a_ka_{k-1}\cdots a_1
    $$ 
    for some 
    $a_h,b_i,c_j\in\{0,1\}.$ 
    
    Assume first that there are $1$'s between the two $0$'s, so then we may assume without loss of generality that 
    $b_1=c_1=1.$ Observe that $y=x+z$ and the sum of $x$ and $z$ in binary has no carrying. Hence, the entries of $x$ and $z$ at all the (at least two) entries of $y$ that are $0$ are also $0$ (since of the pairs $0,1$; $1,0$; $0,0$ only two $0$'s sum to $0$). Therefore, we can write $$
    x=c_m^xc_{m-1}^x\cdots c_1^x0b_\ell^xb_{\ell-1}^x\cdots b_1^x0a_k^xa_{k-1}^x\cdots a_1^x,\hspace{0.05in}
    z=c_m^zc_{m-1}^z\cdots c_1^z0b_\ell^zb_{\ell-1}^z\cdots b_1^z0a_k^za_{k-1}^z\cdots a_1^z.
    $$ 
    Then consider slightly altering both as follows: {\small$$
    x'=c_m^xc_{m-1}^x\cdots c_2^x01b_\ell^xb_{\ell-1}^x\cdots b_1^x0a_k^xa_{k-1}^x\cdots a_1^x,\hspace{0.02in}
    z'=c_m^zc_{m-1}^z\cdots c_2^z01b_\ell^zb_{\ell-1}^z\cdots b_1^z0a_k^za_{k-1}^z\cdots a_1^z.
    $$}
    
    \noindent Observe that $x'+z'=y$ and $x,z'$ still share no $1$'s in their binary representations, and neither do $x',z.$ Therefore, we have $w(x,y,z)=w(x,z')*w(x',z)$ is a representation of $w(x,y,z)$ as a concatenation of (distinct) Ulam words. Similarly, it is possible to define {\small$$
    x''=c_m^xc_{m-1}^x\cdots c_1^x0b_\ell^xb_{\ell-1}^x\cdots b_2^x01a_k^xa_{k-1}^x\cdots a_1^x,\hspace{0.02in}
    z''=c_m^zc_{m-1}^z\cdots c_1^z0b_\ell^zb_{\ell-1}^z\cdots b_2^z01a_k^za_{k-1}^z\cdots a_1^z
    $$}
    
    \noindent and then $w(x,y,z)=w(x,z'')*w(x'',z)$ is another representation of $w(x,y,z)$ as a concatenation of Ulam words. Hence, in fact the existence of $w(x,y,z)=w(x,z)*w(x,z)$ as a representation in this case is not relevant; regardless, $w(x,y,z)$ is not Ulam (since it has two other representations).
    
    Now, suppose $y$ has two consecutive zeroes together; consider the leftmost such pair so that there is a $1$ to the left of them. We can then write $$
    y=b_\ell b_{\ell-1}\cdots b_2100a_ka_{k-1}\cdots a_1
    $$
    and
    $$
    x=b_\ell^xb_{\ell-1}^x\cdots b_1^x00a_k^xa_{k-1}^x\cdots a_1^x,\hspace{0.05in}
    z=b_\ell^zb_{\ell-1}^z\cdots b_1^z00a_k^za_{k-1}^z\cdots a_1^z.
    $$ 
    Note that $b_1^x+b_1^z=1$ so exactly one of the two is $1.$ Then consider $$
    x'=b_\ell^xb_{\ell-1}^x\cdots b_2^x011a_k^xa_{k-1}^x\cdots a_1^x,\hspace{0.05in}
    z'=b_\ell^zb_{\ell-1}^z\cdots b_2^z001a_k^za_{k-1}^z\cdots a_1^z.
    $$ 
    Observe that $x'+z'=y,$ $x$ and $z'$ share no $1$'s in their binary representations, and neither do $z$ and $x'.$ Thus, $w(x,y,z)=w(x,z')*w(x',z)$ is a representation of $w(x,y,z)$ as a concatenation of Ulam words. Similarly, if we define $$
    x''=b_\ell^xb_{\ell-1}^x\cdots b_2^x001a_k^xa_{k-1}^x\cdots a_1^x,\hspace{0.05in}
    z''=b_\ell^zb_{\ell-1}^z\cdots b_2^z011a_k^za_{k-1}^z\cdots a_1^z,
    $$ 
    then $w(x,y,z)=w(x,z'')*w(x'',z)$ is a second such representation. Hence, by the same reasoning as before, $w(x,y,z)$ is not Ulam (independent of $w(x,z)$ being Ulam).
\end{proof}
The splitting of the algorithm into two steps is particularly useful because, except in one very special case (when $y$ is $2$ less than a power of $2$), the update step is in fact quite simple to describe; it will just remove all pseudo-Ulam words satisfying $x+z\equiv0\pmod2.$ However, in order to understand how this works, we need to understand the fundamental reason why these words are removed.
\begin{lemma}\label[lemma]{pseudo-ulam-even}
Suppose that $x+z\equiv0\pmod2$ and $y\neq2^{k+1}-2.$ If $w(x,y,z)$ is pseudo-Ulam, then the words $w(x-1,y,z)$ and $w(x,y,z-1)$ are both also pseudo-Ulam.
\end{lemma}
\begin{proof}
    The proof fundamentally relies on constructing parity-based bijections between the single representation of $w(x,y,z)$ and the resulting unique representations of $w(x-1,y,z)$ and $w(x,y,z-1).$
    
    % First, by \Cref{not-too-small}, any pseudo-Ulam word $w(x,y,z)$ satisfies $x+z\geq y.$ Since we also have
    First, by the condition that $x+z\equiv0\pmod2,$ the only possibility for $x+z=y$ is if $y$ is even. However, we then see that $y$ is not Zumkeller as by assumption it does not equal $2^{k+1}-2.$ In particular, by \Cref{x+z-case}, a representation of $w(x,y,z)$ as a concatenation of equal Ulam words does not affect the final state of $w(x,y,z).$ If $x+z\neq y,$ then there are no such representations so this is a non-issue regardless. Thus, we will assume without loss of generality that any representation of $w(x,y,z)$ consists of distinct Ulam words.
    
    By assumption, there is exactly one such representation, i.e., $0\leq a\leq y$ such that $w(x,a)$ and $w(y-a,z)$ are distinct Ulam words. Without loss of generality, assume that $a\neq z.$ We now split into cases based on the parity of $x+z.$
    \begin{case}
    If $x,z$ are odd, then observe that $$
    w(x-1,y,z)=w(x-1,a)*w(y-a,z)
    $$
    is a representation as a concatenation of two (distinct since $a\neq z$) Ulam words. Suppose by way of contradiction that there is also some other representation $$
    w(x-1,y,z)=w(x-1,b)*w(y-b,z).
    $$
    Then $b$ must be odd, as otherwise $w(x,b)*w(y-b,z)$ is an additional representation of $w(x,y,z).$ On the other hand, $y-b$ is even because $z$ is odd. Therefore, $y$ is odd and so one of $a,y-a$ is odd. But this is a contradiction since then one of $w(x,a)$ and $w(y-a,z)$ is non-Ulam (because both $x$ and $z$ are odd). Hence, there can be no other representation and so $w(x-1,y,z)$ is pseudo-Ulam. Observe that $x\neq y-a$ as well because otherwise $y-a$ is odd and so $w(y-a,z)$ is not Ulam. Therefore, by similar reasoning $w(x,y,z-1)$ is also pseudo-Ulam.
    \end{case}
    \begin{case}
    Now, suppose instead that $x,z$ are both even. Since there is a unique $a$ such that $w(x,a)$ and $w(y-a,z)$ are Ulam and distinct, we must have that for any other $a'$ the words $w(x,a')$ and $w(y-a',z)$ cannot both be Ulam. If $y$ were odd, then exactly one of $a$ and $y-a$ would be odd, without loss of generality $a,$ so then $$
    w(x,y,z)=w(x,a-1)*w(y-a+1,z)
    $$
    would be an alternative representation as a concatenation of (distinct due to parity) Ulam words, contradicting uniqueness. Thus, $y$ is even. It is now important to consider the parity of $a.$
    \begin{subcase}
    If $a$ and $y-a$ are both even, then observe that $w(y-a+1,z)$ is a valid Ulam word since $z$ is even and $y-a$ is even as well. But $w(x,y,z)=w(x,a-1)*w(y-a+1,z),$ so either the words are the same or one is not Ulam. Trivially the words are distinct because $x$ and $y-a+1$ are of different parities. Therefore, we must have that $w(x,a-1)$ is not Ulam. Hence, by \Cref{ulam-from-less}, $w(x-1,a)$ is Ulam. In particular, we obtain a representation $$
    w(x-1,y,z)=w(x-1,a)*w(y-a,z)
    $$
    as a concatenation of distinct (since $a\neq z$) Ulam words. We now show that no other such representation exists. Supposing we also had $$
    w(x-1,y,z)=w(x-1,b)*w(y-b,z)
    $$
    for some $b\neq a,$ we immediately notice that $b$ must be even (since $x-1$ is odd), and so $y-b$ is as well. But observe that $y-b$ is even and so is $z,$ so $w(y-b+1,z)$ is also Ulam. Again by \Cref{ulam-from-less}, exactly one of $w(x,b)$ and $w(x,b-1)$ is Ulam. In particular, we have either an additional (since $b\neq a$ by assumption and $b-1\neq a$ because $b$ is even) representation $$
    w(x,y,z)=w(x,b)*w(y-b,z)\text{ or }w(x,y,z)=w(x,b-1)*w(y-b+1,z)
    $$ 
    of $w(x,y,z)$ as a concatenation of distinct (by the argument in the second paragraph of the proof) Ulam words, a contradiction. %In the latter representation the words are trivially distinct by a parity argument, and therefore we have an alternate (because $a\neq b-1$ since $a,b$ are both even) representation of $w(x,y,z),$ contradicting the fact that $w(x,y,z)$ is Ulam. In the former case, if the words were the same, then remark that $x$ and $z$ are both even so that $y=x+z$ is even. But $y$ is not $2$ less than a power of $2$ as in our conditions, so it in fact has at least two $0$'s in its binary representation. Therefore, by \Cref{x+z-case}, we have that $w(x,y,z)$ is not (pseudo-)Ulam, contradiction. Therefore, the words are not the same, so this gives us an alternate representation of $w(x,y,z)$ as a concatenation of distinct Ulam words $w(x,y,z)=w(x,b)*w(y-b,z)$ (we assumed that $b\neq a$), also a contradiction. Hence, regardless, we reach a contradiction, implying that there is no other representation for $w(x-1,y,z)$ as a concatenation of two distinct Ulam words with one $1,$ implying that it is pseudo-Ulam and thus by the inductive hypothesis in fact Ulam. Similarly, $w(x,y,z-1)$ is also Ulam. In particular, there are in fact three distinct representations of $w(x,y,z)$ as a concatenation of distinct Ulam words, so it is certainly not Ulam.
    
    Hence, $w(x-1,y,z)$ has a unique representation as a concatenation of distinct Ulam words, so it is pseudo-Ulam. If $x\neq y-a,$ then the same argument as above works to show that $w(x,y,z-1)$ is also pseudo-Ulam. Otherwise, $w(x,y,z)=w(x,a)*w(x,z)$ is the unique representation. In particular, again we see that $w(x,z-1)$ is Ulam. The only issue in the above argument that may arise is if $a=z-1,$ so that the resulting representation $w(x,y,z-1)=w(x,a)*w(x,z-1)$ consists of two equal words. But this is impossible, as both $a$ and $z$ were assumed to be even. Hence, $w(x,y,z-1)$ is still pseudo-Ulam, as desired.
    \end{subcase}
    \begin{subcase}
    If, now, $a$ and $y-a$ are both odd, then observe that $w(y-a-1,z)$ is Ulam since $y-a$ is odd. But since $w(x,y,z)$ is pseudo-Ulam, we cannot have $w(x,y,z)=w(x,a+1)*w(y-a-1,z)$ as an alternate representation of $w(x,y,z)$ as a concatenation of two (distinct by assumption) Ulam words. Therefore, one of the words in the representation must not be Ulam---specifically, $w(x,a+1).$ Hence, by \Cref{ulam-from-less}, $w(x-1,a+1)$ is Ulam. This therefore gives a representation $$
    w(x-1,y,z)=w(x-1,a+1)*w(y-a-1,z).
    $$ 
    The words are distinct because $x-1\neq y-a-1$ by parity. We now show that no other such representation exists. To this end, suppose we have an alternate decomposition $$
    w(x-1,y,z)=w(x-1,b)*w(y-b,z),
    $$ 
    where $b\neq a+1.$ Observe that $b$ is even since $x-1$ is odd, and thus since $y$ is even, we must have that $y-b$ is even as well. But therefore, since $z$ is even, we must have that $w(y-b+1,z)$ is also Ulam. As before, by \Cref{ulam-from-less}, exactly one of $w(x,b)$ or $w(x,b-1)$ is Ulam. In particular, we have either an additional (since $b\neq a+1$ by assumption and $b-1\neq a+1$ by parity) representation $$
    w(x,y,z)=w(x,b)*w(y-b,z)\text{ or }w(x,y,z)=w(x,b-1)*w(y-b+1,z)
    $$
    of $w(x,y,z)$ as a concatenation of Ulam words, in any case a contradiction. %Since $b\neq a+1$ and is even, both are distinct from the representation we already have. We then repeat the previous argument to see that in either of these cases, we will obtain another valid representation of $w(x,y,z)$ as a concatenation of two distinct Ulam words, contradicting its pseudo-Ulamness. Hence, $w(x-1,y,z)$ cannot have any other representations as a concatenation of two Ulam words with one $1,$ so it is pseudo-Ulam and by the inductive hypothesis in fact Ulam. Similarly, $w(x,y,z-1)$ is also Ulam. In particular, $w(x,y,z)$ has three representations as a concatenation of smaller Ulam words, so it is certainly not Ulam. 
    Thus, $w(x-1,y,z)$ has no additional representations, so it must be pseudo-Ulam. Similarly, so is $w(x,y,z-1).$
    \end{subcase}
    \end{case}
    Hence, if $w(x,y,z)$ is pseudo-Ulam, then both $w(x-1,y,z)$ and $w(x,y,z-1)$ are as well, as desired.
\end{proof}
Therefore, if pseudo-Ulam words $w(x,y,z)$ with odd $x+z$ are Ulam, then those with even $x+z$ are not. However, we could still have issues with $w(x,y,z)$ with even $x+z$ that are probably not Ulam. We now show that these are in fact unaffected.
% \todo{insert beforehand lemma about borders being non ulam except for $2^{k+1}-1$}
\begin{lemma}\label[lemma]{checkerboard}
    Suppose $w(x,y,z)$ is Ulam and $y\neq2^{k+1}-2.$ Then $x+z\equiv1\pmod2.$
\end{lemma}
This means, in particular, that the update step cannot alter the state of any $w(x,y,z)$ with $x+z\equiv1\pmod2,$ as neither $w(x-1,y,z)$ nor $w(x,y,z-1)$ can be Ulam.
\begin{proof}
    Our proof incorporates \Cref{pseudo-ulam-even}, from which it is evident that it suffices by induction (the base case being words of the form $w(0,y,0),$ all of which are trivially not Ulam because of the representations $1*w(y,0)$ and $w(0,y)*1$) to show that probably not Ulam words $w(x,y,z)$ with $x+z\equiv0\pmod2$ are not Ulam. Indeed, assuming this fact, suppose that all words $w(x,y,z)$ with $x+z<n$ and $x+z\equiv0\pmod2$ are not Ulam. If $w(x,y,z)$ with $x+z=n\equiv0\pmod2$ is Ulam, then it must be pseudo-Ulam. But then it is not Ulam since both $w(x-1,y,z)$ and $w(x,y,z-1)$ are pseudo-Ulam and hence, by the inductive hypothesis, Ulam (thereby completing the inductive step). As in the proof of \Cref{pseudo-ulam-even}, we may assume that any potential representation of $w(x,y,z)$ consists of distinct Ulam words.

    % Fix an arbitrary $y\neq2^{k+1}-2,$ and we proceed by induction on $x$ and $z.$ The base case occurs when $x=z=0.$ By \Cref{not-too-small}, if $w(0,y,0)$ is Ulam then $y\leq0$ and thus $y=0.$ But $w(0,0,0)=11$ is trivially not Ulam. Thus, $w(0,y,0)$ is not Ulam and the lemma holds in this case.

    % Now, for a fixed $x$ and $z,$ suppose that for all $x'\leq x,z'\leq z$ with $x'+z'<x+z$ we have that $w(x',y,z')$ is Ulam only 
    % if $x'+z'\equiv1\pmod2.$ In particular, for any $x',z'$ with $x'+z'\equiv1\pmod2,$ the inductive hypothesis gives that $w(x',y,z')$ is Ulam if and only if it is pseudo-Ulam (since neither $w(x'-1,y,z')$ nor $w(x',y,z'-1)$ can be Ulam).
    
    % Suppose first that $x+z\equiv0\pmod2.$ We split into cases on the state of $w(x,y,z)$ prior to the update step. If it is pseudo-Ulam, then \Cref{pseudo-ulam-even} implies that $w(x-1,y,z)$ and $w(x,y,z-1)$ are both pseudo-Ulam. But by the inductive hypothesis, we know that they are both in fact Ulam, implying that we have two more representations $w(x,y,z)=0*w(x-1,y,z)=w(x,y,z-1)*0.$ Therefore, in this case, $w(x,y,z)$ is not Ulam.
    
    Assume, therefore, that $w(x,y,z)$ is probably not Ulam (it cannot be expressed as $w(x,a)*w(y-a,z)$ for any $0\leq a\leq y$). Note that this implies that $x,z>0$; if not, then without loss of generality $z=0,$ so then we have $w(x,y,0)=w(x,0)*w(y,0),$ which is a decomposition of $w(x,y,0)$ into two Ulam words that are distinct as long as $x\neq y.$ Therefore, $w(x,y,0)$ is either pseudo-Ulam or not Ulam. If, on the other hand, $x=y,$ then first recall that $x+z=x\equiv0\pmod2$ (by our assumption), so $x$ is even. In particular, $w(x,1)$ is Ulam, and so we have the decomposition into Ulam words $w(x,x,0)=w(x,1)*w(x-1,0),$ which works as long as $x\neq0,$ which holds because $x=0$ was the base case. Hence, $w(x,x,0)$ is also either pseudo-Ulam or not Ulam. 
    
    Assuming now that $x,z>0,$ we wish to show that among $w(x-1,y,z)$ and $w(x,y,z-1),$ either neither or both are Ulam. By the induction hypothesis, the latter condition is equivalent to these words both being pseudo-Ulam. We in fact construct a bijection between their representations. %We now split into cases based on the parity of $x$ (the same as that of $z$).
    Consider, therefore, any decomposition of $$
    w(x-1,y,z)=w(x-1,a)*w(y-a,z)
    $$
    into (distinct) Ulam words. %Remark that $z$ is odd and thus $y-a$ is even. 
    Now by assumption, there are no representations of $w(x,y,z)$ as a concatenation of Ulam words with one $1.$ %In particular, the only way for the word $w(x,a)$ to be Ulam (and thus for the decomposition $w(x,y,z)=w(x,a)*w(y-a,z)$ to exist) is if $w(x,a)=w(y-a,z),$ so $y=x+z$ and $a=z.$ However, then $x-1+z=y-1<y$ but $w(x-1,y,z)$ is Ulam, contradicting \Cref{not-too-small}. 
    But then $w(x,a)$ cannot be Ulam, as otherwise we would have the (distinct by assumption) representation $w(x,y,z)=w(x,a)*w(y-a,z).$ Thus, by \Cref{ulam-from-less}, $w(x,a-1)$ is Ulam. Moreover, considering the decomposition $w(x,y,z)=w(x,a-1)*w(y-a+1,z),$ we see that $w(y-a+1,z)$ is not Ulam, implying again by \Cref{ulam-from-less} that $w(y-a+1,z-1)$ is Ulam. Then we get the representation $$
    w(x,y,z-1)=w(x,a-1)*w(y-a+1,z-1).
    $$ 
    The words in the decomposition are distinct because (from the above decomposition of $w(x-1,y,z)$) either $x-1\neq y-a$ or $a\neq z.$ In particular, we have constructed a map from representations of $w(x-1,y,z)$ as a concatenation of distinct Ulam words to those of $w(x,y,z-1).$ But the same reverse map can be constructed, and it is easy to see that they are inverses. Thus, decompositions of $w(x-1,y,z)$ and $w(x,y,z-1)$ are in bijection and so they have the same state, as desired.

    In summary, we see that when $x+z\equiv0\pmod2,$ if $w(x,y,z)$ is pseudo-Ulam or probably not Ulam, then after the update step it is simply not Ulam. This completes the inductive step.
    
    If, on the other hand, $x+z\equiv1\pmod2,$ then by the inductive hypothesis we have that neither $w(x-1,y,z)$ nor $w(x,y,z-1)$ is Ulam, so $w(x,y,z)$ is Ulam if and only if it has a unique representation as a concatenation of distinct Ulam words with one $1$---i.e. it is pseudo-Ulam. This also completes the inductive step.
    % 
    % Thus, induction holds and the statement is proved.
    \end{proof}
    % In particular, we are now in shape to prove \Cref{culam-to-ulam}.
    \begin{corollary}[\Cref{culam-to-ulam}]
    If $y$ is not a Zumkeller number, then $\ulam[y]$ is the set of $(x,z)$ such that $x+z\equiv1\pmod2$ and $\culam[y](x,z)=1.$
    \end{corollary}
    \begin{proof}
        Suppose $w(x,y,z)$ is Ulam. By \Cref{checkerboard}, we know immediately that $x+z\equiv1\pmod2,$ and by \Cref{x+z-case}, if $x+z=y,$ then $w(x,z)$ is not Ulam. In addition, since $x+z\equiv1\pmod2,$ we now see that neither $w(x-1,y,z)$ nor $w(x,y,z-1)$ is Ulam, again by \Cref{checkerboard}, and so all possible representations of $w(x,y,z)$ as a concatenation of two Ulam words are of the form $w(x,a)*w(b,z).$ Moreover, we know that the words must be distinct since $w(x,z)*w(x,z)$ cannot be a valid representation (even if $x+z=y$). But this means that there are exactly $\culam[y](x,z)$-many representations of $w(x,y,z)$ as a concatenation of two distinct Ulam words, implying since $w(x,y,z)$ is Ulam that $\culam[y](x,z)=1.$ Thus, one direction is proved. 
        
        Conversely, if we do indeed have $x+z\equiv1\pmod2$ and $\culam[y](x,z)=1,$ then again by \Cref{x+z-case}, the word $w(x,z)$ is not Ulam if $x+z=y.$ Thus, there exists a unique representation $w(x,y,z)=w(x,a)*w(b,z),$ and the words must be distinct. Since $x+z\equiv1\pmod2,$ we know that neither $w(x-1,y,z)$ nor $w(x,y,z-1)$ is Ulam. Hence, there exists a unique representation of $w(x,y,z)$ as a concatenation of two distinct Ulam words, and thus $w(x,y,z)$ is indeed Ulam, as desierd.
    \end{proof}
    For $y=2^{k+1}-2,$ things are slightly more complicated. It is Zumkeller, so the case of $x+z=y$ becomes relevant. In particular, in the set of Ulam words $w(x,y,z),$ there are exceptional cases with $x+z$ even.
    \begin{lemma}\label[lemma]{almost-checkerboard}
    Suppose $y=2^{k+1}-2$ for some $k>0$ and say $w(x,y,z)$ is Ulam. Then $x+z\equiv1\pmod2$ unless $x+z=2^{k+1}-2$ or $x+z=2^{k+1}.$ In the former case, $w(x,y,z)$ is Ulam if and only if $x,z\equiv0\pmod2$ and in the latter, if and only if $x,z>0.$
    \end{lemma}
    \begin{proof}
    In this proof, we quantify the error produced by ignoring representations of words $w(x,y,z)$ as concatenations of equal Ulam words in our proof of \Cref{checkerboard}. We also observe that this error propagates past the case $x+z=y=2^{k+1}-2,$ in fact all the way to $x+z=2^{k+1}.$ As a consequence, we need to investigate each of these cases separately.
    
    % Note that it remains to consider separately the case $y=2^{k+1}-2.$ 
    The only place in our proof where we actually used that $y\neq2^{k+1}-2$ was when arguing that in representations of $w(x,y,z)$ as a concatenation of Ulam words, the words can be assumed to be distinct. This only arises as an issue when: 
    \setcounter{case}{0}
    \begin{case}
    $x+z=y=2^{k+1}-2$ (and note that $w(x,y,z)$ is not Ulam whenever $x+z<y$) and moreover only if both $x$ and $z$ are even; %$w(x,y,z)=w(x,a)*w(y-a,z)$ with $x,a,y-a,z$ all even and $x+z=y=2^{k+1}-2.$ 
    if $x,z$ are both odd, then $w(x,z)$ is not Ulam. If, on the other hand, $x,z$ are both even with $x+z=2^{k+1}-2,$ then we now prove that $w(x,2^{k+1}-2,z)$ is indeed Ulam. To see this, note that $\frac{x}{2}+\frac{z}{2}=2^{k-1}-1.$ In particular, if we write the binary representation of $$
    \frac{x}{2}=a_ka_{k-1}\cdots a_1,
    $$ 
    then $$
    \frac{z}{2}=b_kb_{k-1}\cdots b_1,
    $$
    where $b_i=1-a_i.$ Observe, then, that if we have a decomposition into Ulam words $$
    w(x,2^{k+1}-2,z)=w(x,a)*w(2^{k+1}-2-a,z)
    $$
    where $a$ is even, then $2^{k+1}-2-a$ is even and the words being Ulam is equivalent (since the binary representations of $x,a,2^{k+1}-2-a,z$ all end with $0$) to the words $w(x/2,a/2)$ and $w(2^k-1-a/2,z/2)$ being Ulam. By the same reasoning, we can write the binary representations $$
    \frac{a}{2}=c_kc_{k-1}\cdots c_1
    $$ 
    and $$
    2^k-1-\frac{a}{2}=d_kd_{k-1}\cdots d_1\quad\quad\quad\hspace{0.075in}
    $$
    where $d_i=1-c_i.$ Since $w(x/2,a/2)$ and $w(2^k-1-a/2,z/2)$ are both Ulam, the corresponding two elements share no $1$'s in their binary representations. In particular, for every $i,$ we have that $a_i+c_i\leq1$ and $b_i+d_i=2-a_i-c_i\leq1.$ Thus, both sums in fact equal $1$ and so $a_i=1-c_i=d_i$ and $b_i=1-d_i=c_i.$
    %If $a_i=1,$ then $c_i=0,$ $b_i=1-a_i=0,$ and $d_i=1-c_i=1.$ If $a_i=0,$ then $b_i=1-a_i=1,$ so $d_i=0,$ and thus $c_i=1-d_i=1.$ Regardless, we see that $a_i=d_i$ and $b_i=c_i.$ 
    Hence, $x/2=2^k-1-a/2$ and $a/2=z/2,$ so $x=2^{k+1}-2-a$ and $a=z.$ Thus, the decomposition splits into two equal Ulam words and is not valid. Hence, any valid decomposition into \textit{distinct} Ulam words must have $a$ (and thus $2^{k+1}-2-a$) be odd. 
    
    Say then that $a=b+1$ and so $2^{k+1}-2-a=2^{k+1}-3-b.$ Then we can once again cut off the last digits to see that $w(x,a)$ is Ulam if and only if $w(x/2,b/2)$ is Ulam and $w(2^{k+1}-2-a,z)$ is Ulam if and only if $w(2^k-2-b/2,z/2)$ is Ulam. If both are Ulam, then we have an Ulam decomposition of $w(x/2,2^k-2,z/2)$: $$
    w\left(\frac{x}{2},2^k-2,\frac{z}{2}\right)=w\left(\frac{x}{2},\frac{b}{2}\right)*w\left(2^k-2-\frac{b}{2},\frac{z}{2}\right).
    $$
    Now, observe that $x/2+z/2=2^k-1,$ in particular exactly one of the two is odd, say without loss of generality that it is $x/2.$ Then note that $b/2$ must be even in order for $w(x/2,b/2)$ to be Ulam, and therefore so must $2^k-2-b/2.$ But then observe that $w(x/2,b/2)$ is Ulam if and only if $w((x/2-1)/2,b/4)$ is (since the last two digits are $1$ and $0$ and thus not both $1$) and $w(2^k-2-b/2,z/2)$ is Ulam if and only if $w(2^{k-1}-1-b/4,z/4)$ is. Observe, however, that $$
    w\left(\frac{x-2}{4},2^{k-1}-1,\frac{z}{4}\right)=w\left(\frac{x-2}{4},\frac{b}{4}\right)*w\left(2^{k-1}-1-\frac{b}{4},\frac{z}{4}\right)
    $$ 
    is a decomposition into Ulam words, and thus decompositions of $w((x-2)/4,2^{k-1}-1,z/4)$ into (not necessarily distinct) Ulam words correspond to those of $w(x/2,2^k-2,z/2),$ which in turn correspond to decompositions of $w(x,2^{k+1}-2,z)$ into \textit{distinct} Ulam words (the words are distinct by a parity argument). Recall the binary representations $$
    \frac{x-2}{4}=a_ka_{k-1}\cdots a_2
    $$
    and $$
    \frac{z}{4}=b_kb_{k-1}\cdots b_2,
    $$
    where $b_i=1-a_i.$ If we have a decomposition $$
    w\left(\frac{x-2}{4},2^{k-1}-1,\frac{z}{4}\right)=w\left(\frac{x-2}{4},c\right)*w\left(d,\frac{z}{4}\right),
    $$
    where $$
    c=c_kc_{k-1}\cdots c_2
    $$ and $$
    d=d_kd_{k-1}\cdots d_2,
    $$
    then note that $c+d=2^{k-1}-1$ and thus $c_i=1-d_i.$ By the same reasoning as before, we eventually deduce that in fact, for all $i,$ we have $a_i=d_i$ and $b_i=c_i,$ implying that $\frac{x-2}{4}=d$ and $\frac{z}{4}=c.$ Therefore, the unique decomposition is $$
    w\left(\frac{x-2}{4},2^{k-1}-1,\frac{z}{4}\right)=w\left(\frac{x-2}{4},\frac{z}{4}\right)*w\left(\frac{x-2}{4},\frac{z}{4}\right),
    $$
    % $w((x-2)/4,z/4)*w((x-2)/4,z/4),$ 
    which therefore leads to a unique decomposition of $w(x,2^{k+1}-2,z)$; precisely, in this case---$x\equiv2\pmod4$---the decomposition is $$
    w(x,2^{k+1}-2,z)=w(x,z+1)*w(x-1,z).
    $$
    \end{case}
    From this we deduce that $w(x,2^{k+1}-2,z)$ with $x+z=2^{k+1}-2$ is Ulam, and equivalently by \Cref{not-too-small}, pseudo-Ulam, if and only if $x\equiv z\equiv0\pmod2.$ This discrepancy from what is expected (that $x+z\equiv1\pmod2$), unfortunately, has the adverse effect of propagating slightly further due to the update step. However, we can quantify exactly how far it does so. 
    \begin{case}
    We will now deduce that $w(x,2^{k+1}-2,z)$ with $x+z=2^{k+1}-1$ is not Ulam; in fact, we explicitly construct two representations of it as a concatenation of smaller Ulam words. Luckily, these are not too difficult to construct. Since $x+z$ is odd, exactly one of $x,z$ is, too; say, without loss of generality, that it is $x.$ Then $x-1$ and $z$ are both even and sum to $2^{k+1}-2,$ in particular we know that $w(x-1,2^{k+1}-2,z)$ is Ulam. Thus, we immediately obtain the representation $$
    w(x,2^{k+1}-2,z)=0*w(x-1,2^{k+1}-2,z).
    $$
    The words are trivially distinct, so this is a valid representation. The other representation is also simple; since $x+z=2^{k+1}-1,$ their binary representations are complementary and thus do not share any $1$'s. Hence, $w(x,z)$ is Ulam. Since $x$ is odd, the word $w(x-1,z)$ is also Ulam. Thus, we have the representation $$
    w(x,2^{k+1}-2,z)=w(x,z)*w(x-1,z)
    $$
    as a concatenation of trivially distinct Ulam words.
    \end{case}
    Now, we know that some (in fact, all) words $w(x,2^{k+1}-2,z)$ with $x+z=2^{k+1}-1$ that previously would have been Ulam are not, anymore. Therefore, we have to go one more step further.
    \begin{case}
    Consider words $w(x,2^{k+1}-2,z)$ with $x+z=2^{k+1}.$ First, there are the two cases $x=0$ or $z=0.$ To resolve them, we have the two representations $$
    w(0,2^{k+1}-2,2^{k+1})=w(0,2^{k+1}-2)*w(0,2^{k+1})=w(0,0)*w(2^{k+1}-2,2^{k+1}).
    $$
    These two representations are distinct (since $k>0$), so $w(0,2^{k+1}-2,2^{k+1})$ and, symmetrically, $w(2^{k+1},2^{k+1}-2,0)$ are not Ulam. With this, we may assume that $x,z>0.$ 
    
    Note that $x+z=2^{k+1}$ and neither is $0,$ so both are strictly less than $2^{k+1}.$ In particular, if $2^a$ divides $x,$ then it divides $z=2^{k+1}-x=2^a(2^{k+1-a}-x/2^a),$ and vice versa. Therefore, $x$ and $z$ have the same exponent of $2,$ so say that we can write them as $x=2^ax',z=2^az'$ with $x',z'$ being odd. Observe that any representation as a concatenation of Ulam words $$
    w(x,2^{k+1}-2,z)=w(x,c)*w(d,z)
    $$
    is valid if and only if, correspondingly, $w(x',c')$ and $w(d',z')$ are both Ulam, where $c'=\lfloor\frac{c}{2^a}\rfloor$ and $d'=\lfloor\frac{d}{2^a}\rfloor$ (the last $a$ digits of $x$ and $z$ are $0$ so they are irrelevant; the $i$-th digit from the right cannot be $1$ for both $x$ and $c$ or for both $d$ and $z$). Now, note that $2^ac'\leq c<2^a(c'+1),$ and the same for $d',$ so $2^a(c'+d')\leq 2^{k+1}-2<2^a(c'+d'+2).$ The only way for this to happen is if $c'+d'=2^{k+1-a}-2$ or $2^{k+1-a}-1.$ The latter cannot happen, as $x'$ and $z'$ are both odd, so $c'$ and $d'$ are both even (and so is their sum). Thus, we must have $c'+d'=2^{k+1-a}-2.$ We see that $w(x',c')$ and $w(d',z')$ are then Ulam if and only if, respectively, $w(x'-1,c')$ and $w(d',z'-1)$ are. But observe that $(x'-1)+(z'-1)=c'+d'=2^{k+1-a}-2$ and so we have a decomposition into (not necessarily distinct) Ulam words of $$
    w(x'-1,2^{k+1-a}-2,z'-1)=w(x'-1,c')*w(d',z'-1),
    $$
    where $x'-1,c',d',z'-1$ are even. But we know from the above arguments that only one such decomposition exists---when $c'=z'-1$ and $d'=x'-1.$ Thus, we have fixed all but the last $a$ digits of any $c$ and $d,$ giving a valid decomposition of $w(x,2^{k+1}-2,z)$ into Ulam words with one $1.$ To determine the last $a$ digits, define $c''=c-2^ac'$ and $d''=d-2^ad',$ and observe that these are exactly the last $a$ digits of $c$ and $d,$ respectively. But now, note that $0\leq c'',d''\leq2^a-1$ by definition of $c'$ and $d',$ and we have $$
    c''+d''=(c+d)-2^a(c'+d')=2^{k+1}-2-2^a(x'+z'-2)=2^{k+1}+2^{a+1}-2-(x+z)=2^{a+1}-2.
    $$
    This immediately implies that $c''=d''=2^a-1,$ and so they must be fixed as well. Thus, there is a unique representation of $w(x,2^{k+1}-2,z)$ as a concatenation of (distinct) Ulam words with one $1$---precisely, $$
    w(x,2^{k+1}-2,z)=w(x,z-1)*w(x-1,z);
    $$
    these are trivially distinct. Hence, $w(x,2^{k+1}-2,z)$ is pseudo-Ulam, and, since no words $w(x,2^{k+1}-2,z)$ with $x+z=2^{k+1}-1$ are Ulam, it is in fact Ulam. Note that this holds for all $x,z$ with $x,z>0$ (and $x+z=2^{k+1}$).
    \end{case}
    So far, we now know that there are Ulam words $w(x,2^{k+1}-2,z)$ with either $x+z=2^{k+1}-2$ or $x+z=2^{k+1}.$ One may think that the propagation continues to occur, but luckily, the next step is the last. 
    \begin{case}
    We now prove that there are no representations for $w(x,2^{k+1}-2,z)$ with $x+z=2^{k+1}+1$ and $x,z>1$ as a concatenation of (distinct) Ulam words with one $1.$ We also show that $w(x-1,2^{k+1}-2,z)$ and $w(x,2^{k+1}-2,z-1)$ are both Ulam, and thus $w(x,2^{k+1}-2,z)$ is not Ulam. Thus, the Ulam status of words $w(x,2^{k+1}-2,z)$ with $x+z=2^{k+1}$ does not affect any of these (since they are non-Ulam regardless). However, we must separately consider the case when one of $x,z$ is $0$ or $1.$ Observe that the state of the words $w(0,2^{k+1}-2,2^{k+1}+1)$ and $w(2^{k+1}+1,2^{k+1}-2,0)$ in fact does not matter, since they can only be affected by $w(0,2^{k+1}-2,2^{k+1})$ and $w(2^{k+1},2^{k+1}-2,0),$ respectively, both of which are non-Ulam. Additionally, consider the two representations of $$
    w(2^{k+1},2^{k+1}-2,1)=w(2^{k+1},0)*w(2^{k+1}-2,1)=w(2^{k+1},2^{k+1}-2)*w(0,1)
    $$
    as a concatenation of distinct Ulam words, implying that it is not Ulam, and, symmetrically, neither is $w(1,2^{k+1}-2,2^{k+1}),$ as desired. 
    
    Suppose now that $x,z>1.$ Since $x+z=2^{k+1}+1,$ exactly one of $x,z$ is odd, say without loss of generality it is $x.$ Note that $x,z>1,$ so $x-1,z>0.$ Any valid decomposition of $w(x,2^{k+1}-2,z)=w(x,a)*w(b,z)$ into Ulam words implies that $a$ is even, and thus $w(x-1,a)$ is also Ulam. This, in turn, gives a decomposition $w(x-1,2^{k+1}-2,z)=w(x-1,a)*w(b,z).$ But then the case of $x+z=2^{k+1}$ applies, allowing us to deduce that the only possible valid decomposition of $w(x-1,2^{k+1}-2,z)$ is $$
    w(x-1,2^{k+1}-2,z)=w(x-1,z-1)*w(x-2,z).
    $$
    Yet, this decomposition truly being valid would imply that the original decomposition is $$
    w(x,2^{k+1}-2,z)=w(x,z-1)*w(x-2,z).
    $$
    Since $x$ and $z-1$ are both odd, the word $w(x,z-1)$ is not Ulam, and thus this is not a valid decomposition. In other words, $w(x,2^{k+1}-2,z)$ has no valid decompositions as a concatenation of two Ulam words with one $1.$ On the other hand, both $w(x-1,2^{k+1}-2,z)$ and $w(x,2^{k+1}-2,z-1)$ are Ulam, since $x+z-1=2^{k+1}$ and $x,z>1$ so $x-1,z-1>0.$ 
    \end{case}
    Thus, $w(x,2^{k+1}-2,z)$ is not Ulam regardless of the update step, implying that for any $x,z$ with $x+z>2^{k+1},$ the update step once again begins to behave as it normally would---removing all words $w(x,y,z)$ with $x+z\equiv0\pmod2$ and ignoring words with $x+z\equiv1\pmod2,$ as desired.
\end{proof}
This indicates that the update step is essentially null; unless $y=2^{k+1}-2$ (in which case we can explicitly describe exceptions), we can simply ignore $w(x,y,z)$ with $x+z\equiv0\pmod2,$ as we immediately know that they are not Ulam, and all other words are Ulam if and only if they have exactly one representation as a sum of distinct Ulam words $w(x,a)*w(y-a,z).$

Meanwhile, the calculation step simply examines the interactions of binomial coefficients modulo $2.$ In particular, it will preserve the fascinating phenomenon of biperiodicity.  
\begin{lemma}\label[lemma]{biperiodic}
Suppose $2^k\leq y<2^{k+1}$ and $y$ is not a Zumkeller number. Then $\ulam[y]$ is biperiodic with biperiod $2^{k+1}\times2^{k+1}.$
\end{lemma}
\begin{remark}
The question may arise as to whether any smaller biperiod is possible, but in fact $2^{k+1}\times2^{k+1}$ is minimal; for example, it is not hard to see that $w(x,y,0)$ and $w(0,y,z)$ are Ulam if and only if $x,z\equiv-1\pmod{2^{k+1}}.$
\end{remark}
\begin{proof}
    By \Cref{checkerboard}, it suffices to prove biperiodicity following the calculation step. Specifically, we will show that $w(x,y,z)$ for $x+z\equiv1\pmod2$ is Ulam (equivalently, pseudo-Ulam) if and only if its reduction $w(x',y,z')$ is Ulam (pseudo-Ulam), where $x'\equiv x,z'\equiv z\pmod{2^{k+1}}.$
    
    Suppose $\binom{x+a}{x}$ and $\binom{z+y-a}{z}$ are both odd. Also, either $x+z\neq y$ or we also have the condition $x\neq y-a$ (in fact, the case $x+z=y$ is exactly where biperiodicity may break). By \Cref{lucas}, we see that $a$ and $x$ share no $1$'s in their binary representations, and neither do $y-a$ and $z.$ But observe that $a,y-a\leq y<2^{k+1},$ so they can only have $1$'s in the last $k+1$ digits of the binary expansion. In particular, adding or subtracting multiples of $2^{k+1}$ to $x$ and $z$ does not change the condition that $\binom{x+a}{x}$ and $\binom{z+y-a}{z}$ are both odd. Thus, it is true if and only if it is true when $x$ and $z$ are considered modulo $2^{k+1}.$ But pseudo-Ulam and probably not Ulam is exactly counting instances of $a$ such that $\binom{x+a}{x}$ and $\binom{z+y-a}{z}$ are both odd (and said instances occur $1$ and $0$ times, respectively). Since this does not depend on $x,z\pmod{2^{k+1}},$ we therefore can deduce that $w(x,y,z)$ is pseudo-Ulam (resp. probably not Ulam) if and only if its reduction $w(x',y,z')$ is as well, \textit{except} perhaps when $x'+z'=y.$ Precisely, the issue of words needing to be distinct becomes significant when $w(x',z')$ is Ulam (since $w(x',y,z')=w(x',z')*w(x',z')$). However, by \Cref{x+z-case} (noting that $y$ satisfies the lemma conditions), the state of $w(x,y,z)$ is unchanged whether $w(x,z)$ is Ulam or not; relevant representations will still continue to carry over modulo $2^{k+1}.$ In particular, we indeed have that in all cases, $w(x,y,z)$ is in the same state as its reduction $w(x',y,z'),$ and so the desired set is biperiodic.
\end{proof}
This result implies that, in order to more properly understand Ulam words $w(x,y,z),$ in reality we only have to understand those satisfying $x,z\leq2^{k+1}$---the `first square.' %In particular, it may (and indeed turns out to) be of use to graph $(x,z)$ such that $w(x,y,z)$ is Ulam. We may then define $\ulam[y]$ to be precisely the set of such $(x,z).$
Looking more precisely at graphs of $\ulam[y]$ for various $y,$ we notice that increasing even $y$ by $1$ does not change the resulting set, except for the occasional imperfection. We now prove this result.
\begin{lemma}\label[lemma]{odd-even-equal}
If $y$ is odd and not a Zumkeller number, then $\ulam[y]=\ulam[y-1].$
\end{lemma}
\begin{proof}
    Note that $y-1$ only has more $0$'s in its binary representation and thus is also not a Zumkeller number. In particular, by \Cref{checkerboard}, both $\ulam[y]$ and $\ulam[y-1]$ contain only coordinates $(x,z)$ satisfying $x+z\equiv1\pmod2,$ and that $w(x,y,z)$ and/or $w(x,y-1,z)$ satisfying the modular condition are Ulam if and only if they are pseudo-Ulam. It remains to show that these coordinates are the same. To this end, for any $(x,z)$ with $x+z\equiv1\pmod2,$ fix an arbitrary representation of $w(x,y,z)$ as a concatenation of (not necessarily distinct) Ulam words: $w(x,y,z)=w(x,a)*w(y-a,z).$ Without loss of generality, $x$ is odd. Then $a$ is even, and so, since $y$ is odd, $y-a$ must be, too. In particular, $w(y-a-1,z)$ is also Ulam, so we have the representation of $w(x,y-1,z)=w(x,a)*w(y-a-1,z)$ as a concatenation of Ulam words. Since this is true for any $a,$ there are at least as many representations of $w(x,y-1,z)$ as there are of $w(x,y,z).$ Conversely, suppose that we have a representation $w(x,y-1,z)=w(x,a)*w(y-a-1,z).$ Then $a$ is still even since $x$ is odd, and so $y-a-1$ is also even. But $z$ is even, and thus $w(y-a,z)$ is also Ulam. Hence, we obtain a decomposition of $w(x,y,z)=w(x,a)*w(y-a,z).$ In particular, there are at least as many representations of $w(x,y,z)$ as there are of $w(x,y-1,z),$ and thus there are equally many of each. But, since neither $y$ nor $y-1$ are Zumkeller, any representation with the Ulam words in question being equal is irrelevant. This immediately implies that either both have $0$ representations, in which case neither is Ulam, or both have $1,$ in which case both are Ulam, or both have $2$ or more, in which case neither is Ulam. In any case, both are of the same type---so $(x,z)\in\ulam[y]$ if and only if $(x,z)\in\ulam[y-1],$ and thus the two sets are equal, as desired.
\end{proof}
Despite $\ulam[y]$ and $\ulam[y-1]$ having the same structure, the graphs $\pulam[y]$ and $\pulam[y-1]$ (of pseudo-Ulam words $w(\cdot,y,\cdot)$ and $w(\cdot,y-1,\cdot),$ respectively) are still very different---$\pulam[y-1]$ generally has points with an even sum of coordinates that $\pulam[y]$ simply does not have to begin with. These differences begin to play even greater of a role once we attempt to describe $\ulam[2^dy].$ In fact, it becomes important to consider not just coordinates in $\pulam[y],$ but also $(x,z)\in\pulam[y-1]$ corresponding to words $w(x,y,z)$ that have either zero or more than one representation as a concatenation of Ulam words.

\section{Special cases}\label{special-section}
Recall that the aforementioned numbers with at most one zero in their binary representation are known as \textit{Zumkeller} numbers; equivalently, they are the numbers expressible as $2^m-2^k-1$ for some nonnegative integers $m$ and $k<m$ (note that $2^{k+1}-1=2^{k+2}-2^{k+1}-1,$ so it is also expressible in this form) \cite[A089633]{oeis}. As can be seen from the proof of \Cref{biperiodic}, biperiodicity essentially continues to hold even when $y$ is a Zumkeller number; the only impurities arise in the first square of the period---words $w(x,y,z)$ satisfying $x,z<2^{k+1}.$ We will now describe these impurities precisely.
\begin{theorem}[\Cref{zumkeller}]\label{special-case}
\hfill
\begin{enumerate}
    \item The set $\ulam[2^{k+1}-1]$ consists of $(x,z)$ such that $x+z\equiv-1\pmod{2^{k+1}}$ and $x+z\neq2^{k+1}-1.$
    \item The set $\ulam[2^{k+1}-2]$ consists of $(x,z)$ such that either $x+z\equiv-1\pmod{2^{k+1}}$ and $x+z\neq2^{k+1}-1$ or $x+z=2^{k+1}-2$ and $x,z\equiv0\pmod2$ or $x+z=2^{k+1}$ and $x,z>0.$
    \item Suppose $y=2^{k+1}-2^a-1$ is a Zumkeller number for some $0<a<k.$ Then $\ulam[y]$ is the set of $(x,z)$ such that either $x+z=y$ and $x\pmod{2^{a+1}}<2^a$ or the reductions $x'$ and $z',$ respectively, of $x$ and $z$ modulo $2^{k+1}$ satisfy either $x'+z'=2^{k+1}-1$ or $x'+z'=2^{k+1}+2^a-1.$
\end{enumerate}
\end{theorem}
Although the above parts are lumped together for conciseness, each separate part may be viewed as a sub-theorem of its own, and correspondingly we will prove them independently.
\begin{proof}[Proof of part 1]
To begin, suppose that $(x,z)\in\ulam[2^{k+1}-1].$ Then, by \Cref{checkerboard}, there exists a unique representation of $w(x,2^{k+1}-1,z)$ as a concatenation of distinct Ulam words $w(x,a)*w(2^{k+1}-1-a,z).$

We remark that $a,2^{k+1}-1-a\leq2^{k+1}-1<2^{k+1},$ and thus both can only have $1$'s in the last $k+1$ digits of their binary representations. In particular, $x$ shares no $1$'s with $a$ if and only if its reduction modulo $2^{k+1},$ which is equal to $x',$ does, and the same applies to $z$ and $2^{k+1}-1-a,$ with the reduction being $z'.$ Hence, the word $w(x',2^{k+1}-1,z')$ is either pseudo-Ulam or, if the reduction results in identical words, not Ulam. However, the latter case can only happen if $x'+z'=2^{k+1}-1,$ so by \Cref{not-too-small} we have $x'+z'\geq2^{k+1}-1.$ Now, write the binary representation of $a=a_ka_{k-1}\cdots a_1a_0,$ so that $2^{k+1}-1-a=b_kb_{k-1}\cdots b_1b_0,$ where $b_i=1-a_i$ for all $i.$ In particular, exactly one of $a_i,b_i$ is equal to $1,$ so, since $w(x',a)$ and $w(2^{k+1}-1-a,z')$ are Ulam, correspondingly, either the $i$-th digit from the right of the binary expansion of $x'$ or respectively that of $z'$ is $0.$ In particular, we have $(x')_i+(z')_i\leq1$ for all $i.$ But the digits of $x'$ and $z'$ only go up to the $(k+1)$-th spot from the right (since they are residues modulo $2^{k+1}$), so in fact we see that $$
x'+z'=\sum_i2^i((x')_i+(z')_i)\leq\sum_{i=0}^k2^i=2^{k+1}-1.
$$ 
In other words, $x'+z'=2^{k+1}-1.$ Note that this corresponds to $x+z\equiv-1\pmod{2^{k+1}}.$ In this case, we can write $x'=x_kx_{k-1}\cdots x_1x_0$ and $z'=z_kz_{k-1}\cdots z_1z_0$ where $z_i=1-x_i.$ As previous arguments have shown us, the fact that $w(x',a)$ and $w(2^{k+1}-1-a,z')$ are both Ulam implies that $x_i=b_i$ and $a_i=z_i$ for all $i,$ so that the unique representation of $w(x',2^{k+1}-1,z')$ as a concatenation of Ulam words is $w(x',z')*w(x',z').$ This representation gives a valid representation of $w(x,2^{k+1}-1,z)=w(x,z')*w(x',z)$ as a concatenation of \textit{distinct} Ulam words if and only if either $x\neq x'$ or $z\neq z'$ (or both). Thus, unless $x=x'$ and $z=z',$ which corresponds to the case $x+z=2^{k+1}-1,$ the representation is valid and unique, so $w(x,2^{k+1}-1,z)$ is Ulam. If, on the contrary, $x+z=2^{k+1}-1,$ then the representation is still unique but invalid, so $w(x,2^{k+1}-1,z)$ is not Ulam. We therefore obtain precisely the desired description of $\ulam[2^{k+1}-1].$
\end{proof}
Noting that the description of $\ulam[2^{k+1}-2]$ is itself already more involved, it is natural to expect that the proof itself will be as well, but luckily we have \Cref{almost-checkerboard} coming to our rescue, doing most of the work with $x+z\leq2^{k+1}.$ But it does remain to check the generic, periodic case.
\begin{proof}[Proof of part 2]
As mentioned above, \Cref{almost-checkerboard} pinpoints the desired description for $x+z\leq2^{k+1}.$ As for $x+z>2^{k+1},$ the update step is irrelevant, so $w(x,2^{k+1}-2,z)$ is Ulam if and only if it is pseudo-Ulam. Suppose then that $x+z>2^{k+1}$ and there is a unique representation of $w(x,2^{k+1}-2,z)$ as a concatenation of distinct Ulam words $w(x,a)*w(2^{k+1}-2-a,z).$ We split into two cases: $a$ being even and odd. 
    
If $a$ is even, then so is $2^{k+1}-2-a,$ and thus the last digits of $x$ and $z$ in any representation are arbitrary. We can thus cut them off, looking at the representation $w(x',2^k-1,z')=w(x',a/2)*w(2^k-1-a/2,z'),$ where $x'=\lfloor\frac{x}{2}\rfloor$ and $z'=\lfloor\frac{z}{2}\rfloor.$ But now recall from above that a representation exists and is unique, hence giving a unique representation for $w(x,2^{k+1}-2,z)=w(x,a)*w(2^{k+1}-2-a,z)$ with $a$ even, if and only if $x''+z''=2^k-1,$ where $x''=x'\pmod{2^k}$ and $z''=z'\pmod{2^k}.$
    
If, on the other hand, $a$ is odd, then so is $2^{k+1}-2-a,$ and thus the last digits of $x$ and $z$ in any representation are both $0.$ We can still cut the last digits off, remembering the condition that $x$ and $z$ are both even, to obtain the representation $w(x/2,2^k-2,z/2)=w(x/2,(a-1)/2)*w(2^k-2-(a-1)/2,z/2).$ Thus, this case adds the number of representations of $w(x/2,2^k-2,z/2)$ as a concatenation of Ulam words with one $1$ to the number of such representations of $w(x,2^{k+1}-2,z).$
    
Now, a simple inductive argument allows us to conclude that $w(x,2^{k+1}-2,z)$ has exactly one representation as a concatenation of (not necessarily distinct) Ulam words with one $1$ if and only if $x+z\equiv-1\pmod{2^{k+1}}.$ The base case considers all $(x,z)$ with $x+z\leq2^k,$ discussed above. Now, assume that all $(x',z')$ with $x'+z'<x+z$ have one representation if and only if $x'+z'\equiv-1\pmod{2^{k+1}}.$ If, now, $x+z\equiv-1\pmod{2^{k+1}},$ then first observe that $x'+z'\equiv-1\pmod{2^k}$ and so $x''+z''=2^k-1.$ In particular, there exists a unique representation of $w(x,2^{k+1}-2,z)=w(x,a)*w(2^{k+1}-2-a,z)$ for $a$ even. Moreover, there are no representations for $a$ odd, as this would imply that $x$ and $z$ are both even. Hence, $w(x,2^{k+1}-2,z)$ has a unique representation, as desired. Suppose now that $x+z\not\equiv-1\pmod{2^{k+1}}.$ Then observe by the inductive hypothesis that there are $0$ or at least $2$ representations of $w(x,2^{k+1}-2,z)=w(x,a)*w(x,2^{k+1}-2-a,z)$ with $a$ odd, and the same with $a$ even. Combining these, we see that there are $0$ or at least $2$ representations of $w(x,2^{k+1}-2,z)$ total, so it is not (pseudo-)Ulam, completing the inductive step.
    
In particular, for any $x,z$ with $x+z\geq2^{k+1},$ we see that $(x,z)\in\ulam[2^{k+1}-2]$ if and only if $w(x,2^{k+1}-2,z)$ is pseudo-Ulam, i.e., if and only if by the above $x+z\equiv-1\pmod{2^{k+1}},$ as desired.
\end{proof}
The final description, and proof as well, of $\ulam[2^{k+1}-2^a-1]$ is in fact remarkably similar in essence to that of $\ulam[2^{k+1}-2],$ which is due to the fact that both are spawned by `similar types' of Zumkeller numbers: the ones with exactly one zero in them. For the reader's convenience, we include it in the appendix.
\begin{corollary}[\Cref{intro-biperiodic}]
Suppose that $2^k\leq y<2^{k+1}.$
\begin{itemize}
    \item The set $\ulam[y]$ is biperiodic if and only if $y$ is not a Zumkeller number, and in this case its biperiod is $2^{k+1}\times2^{k+1}.$
    \item If $y$ is a Zumkeller number, then $\ulam[y]$ is eventually biperiodic with biperiod $2^{k+1}\times2^{k+1}$; moreover, the only impurities are in the bottom-left block: $(x,z)$ satisfying $x,z<2^{k+1}.$
\end{itemize}
\end{corollary}
\begin{proof}
This is an immediate consequence of \Cref{biperiodic,special-case}.
\end{proof}
Visually, in these cases $\ulam[y]$ is either a set of checkerboard-like segments on one extra diagonal line inserted into an infinite biperiodic array of squares with two diagonal segments each, or it is one diagonal removed from an infinite periodic array with only one segment per square (as is in the case of $y=2^{k+1}-1$), or it is that same diagonal removed but some of the points in the two adjacent ones added (as is in the case of $y=2^{k+1}-2$):
\begin{figure}[H]
    \centering
     \begin{subfigure}[b]{0.32\textwidth}
         \centering
         \includegraphics[width=\textwidth]{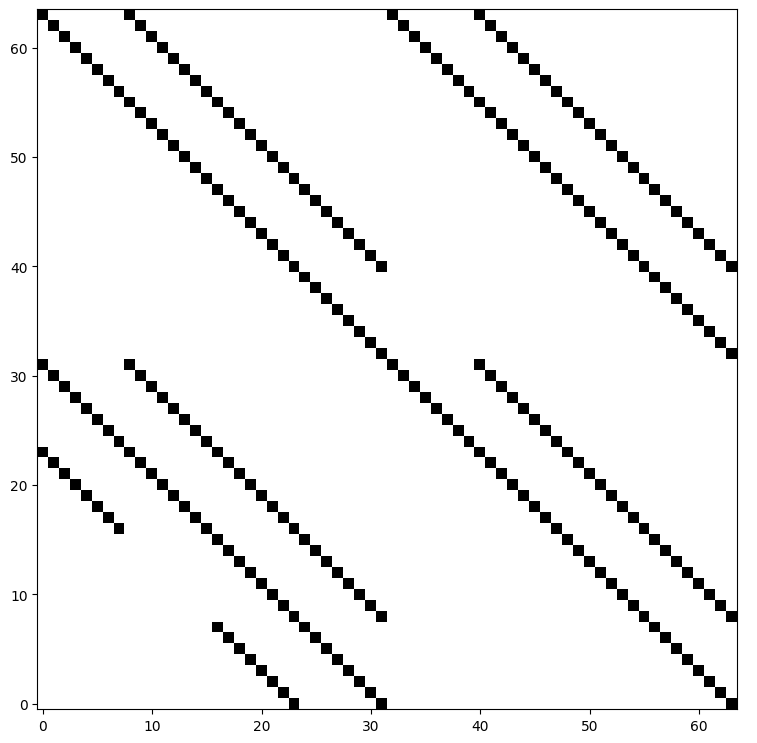}
         \caption{$\ulam[23]$}
         \label{fig:23}
     \end{subfigure}
     \hfill
    %  \hspace{-0.9in}
     \begin{subfigure}[b]{0.32\textwidth}
         \centering
         \includegraphics[width=\textwidth]{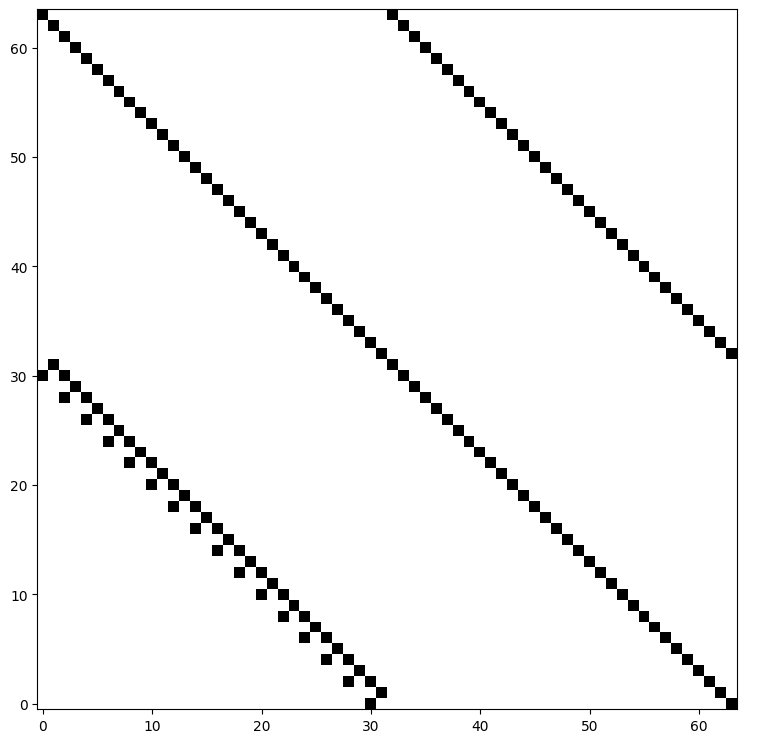}
         \caption{$\ulam[30]$}
         \label{fig:30}
     \end{subfigure}
     \hfill
    %  \hspace{-0.9in}
     \begin{subfigure}[b]{0.32\textwidth}
         \centering
         \includegraphics[width=\textwidth]{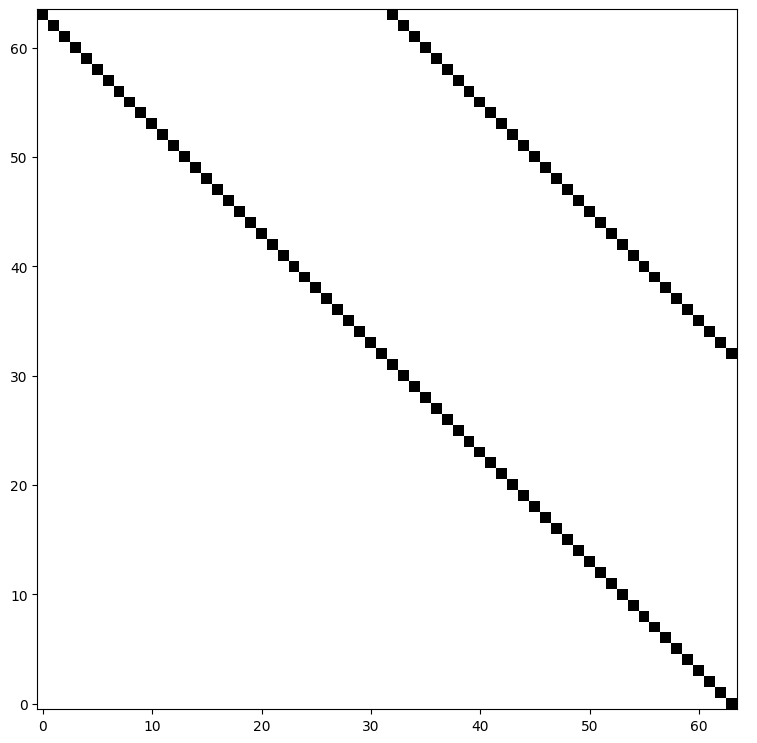}
         \caption{$\ulam[31]$}
         \label{fig:31}
     \end{subfigure}
        \caption{$\ulam[y]$ for (a) $y=23,$ (b) $y=30,$ (c) $y=31.$}
        \label{fig:imperfect graphs}
\end{figure}
\section{Hierarchical and fractal descriptions}\label{main-theorem}
\begin{figure}[H]
    \centering
     \begin{subfigure}[b]{0.32\textwidth}
         \centering
         \includegraphics[width=\textwidth]{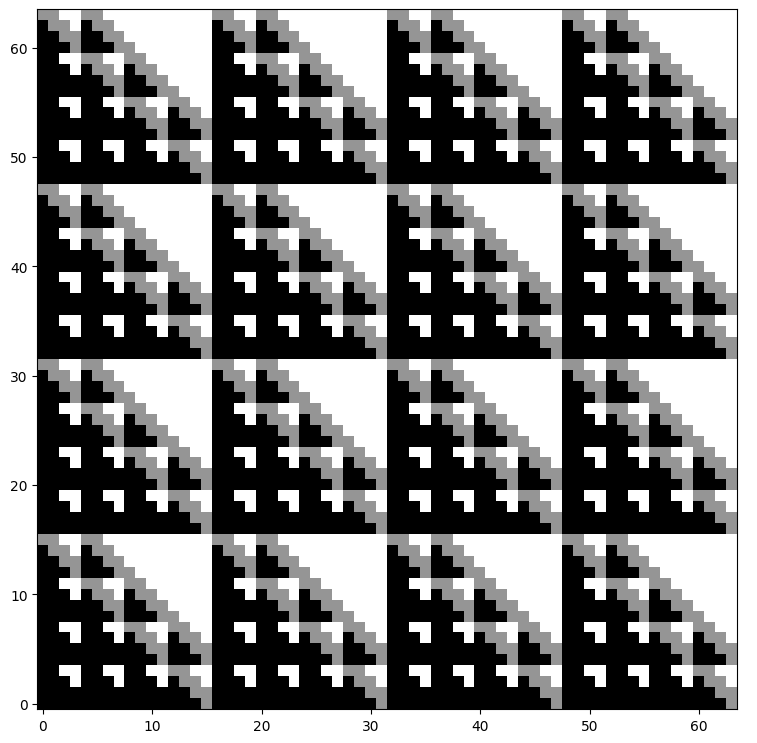}
         \caption{$\culam[10]$}
         \label{fig:10}
     \end{subfigure}
     \hfill
    %  \hspace{-0.9in}
     \begin{subfigure}[b]{0.32\textwidth}
         \centering
         \includegraphics[width=\textwidth]{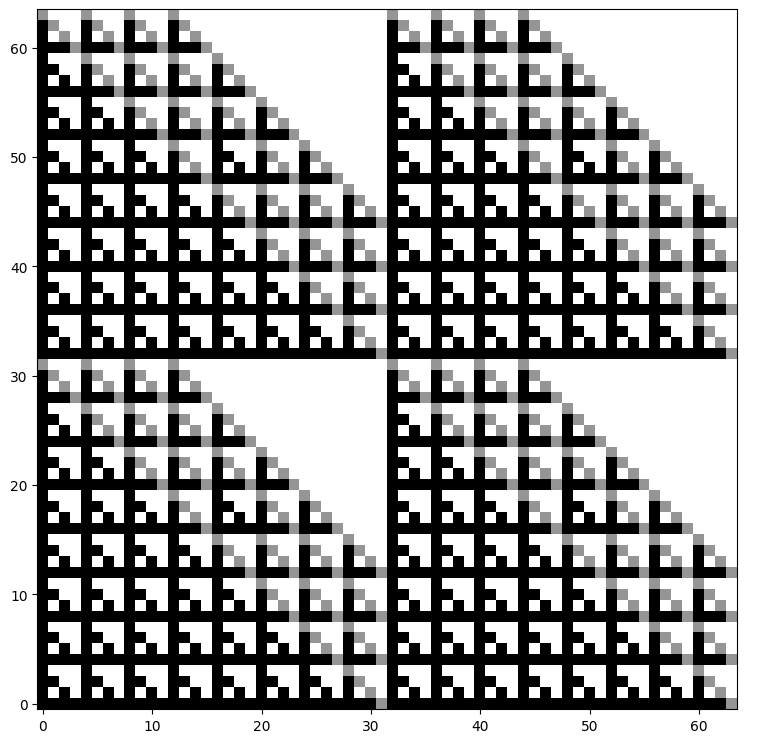}
         \caption{$\culam[19]$}
         \label{fig:19}
     \end{subfigure}
     \hfill
    %  \hspace{-0.9in}
     \begin{subfigure}[b]{0.32\textwidth}
         \centering
         \includegraphics[width=\textwidth]{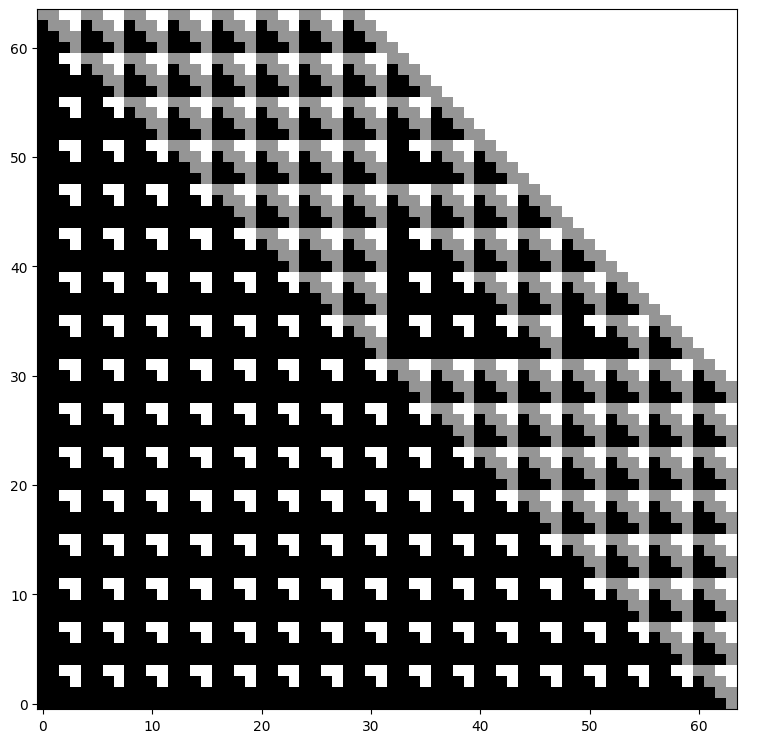}
         \caption{$\culam[34]$}
         \label{fig:34}
     \end{subfigure}
        \caption{$\culam[y]$ for (a) $y=10,$ (b) $y=19,$ (c) $y=34.$}
        \label{fig:caveman graphs}
\end{figure}
We observe that these representations seem quite regular and periodic-looking. At the same time, the coloring for $19,$ an odd number, seems much less dense than the coloring for $10$ and $34.$ Looking through more of these images, we begin to notice more of a pattern: the graphs $\culam[2^dy]$ and $\culam[2^dy+1],$ where $y$ is odd, are tiled by $2^d\times2^d$ blocks whose coloring is determined by the colorings of corresponding squares in $\culam[y-1]$ and $\culam[y].$ Precisely, we in fact have the general conversion table shown in \Cref{building-blocks}.

In fact, all of these tables are simply linear combinations of two tiles each for $2^dy$ and $2^dy+1$---one grey square (with black and white squares interspersed in the case of $2^dy+1$) and one fractal pattern. Corresponding squares are summed and multiplied in the `caveman' way as described in \Cref{caveman-operations}. We can therefore find an inductive algorithm to determine $\culam[y]$ for any $y,$ by successively reducing the number of $1$'s in its binary expansion. But first, we need to explicitly define these blocks and prove that they take the desired form.
\begin{definition}
For every positive integer $d,$ define labelings $\eone[d],\etwo[d],\oone[d],\otwo[d]:\{0,1,2,\dots,2^d-1\}^2\to\{0,1,2\}$ as follows:
\begin{itemize}
    \item If $x+z\geq2^d-1,$ let $\eone[d](x,z)=0.$ Otherwise, suppose that $e$ is the largest $e<d$ satisfying $x,z\pmod{2^{e+1}}<2^e$ (which must exist). If it also satisfies $x\pmod{2^e}+z\pmod{2^e}<2^e-1,$ then $\eone[d](x,z)=2.$ Otherwise, $\eone[d](x,z)=1.$ 
    \item Let $\etwo[d](x,z)=1$ for all $x$ and $z.$
    \item If $x+z\geq2^d-2,$ let $\oone[d](x,z)=0.$ Otherwise, if $x+z$ is even, then $\oone[d](x,z)=0$ if $x$ is odd and $\oone[d](x,z)=2$ if $x$ is even. Finally, if $x+z<2^d-2$ and is odd, then once again define $e$ as in the case of $\eone.$ If $x\pmod{2^e}+z\pmod{2^e}<2^e-1,$ then $\oone[d](x,z)=2.$ Otherwise, $\oone[d](x,z)=1.$
    \item If $x+z$ is even, then $\otwo[d](x,z)=0$ if $x$ is odd and $\otwo[d](x,z)=2$ if $x$ is even. Otherwise, $\otwo[d](x,z)=1.$
\end{itemize}
\end{definition}
Recall that these labelings are visualized in \Cref{odd-even-patterns}. While $\etwo$ and $\otwo$ have very simple definitions, $\eone$ and $\oone$ seem at first glance to be complicated and are fractal-like. However, they actually have very simple descriptions in terms of $\culam$ labelings of powers of $2$:
\begin{lemma}\label[lemma]{culam-power-of-two}
Define the `modified' labelings $\culam[2^d]',\culam[2^d+1]':\{0,1,2,\dots,2^d-1\}^2\to\{0,1,2\}$ as follows. The value $\culam[2^d]'(x,z)$ counts (in the caveman style) representations $$
w(x,2^d,z)=w(x,a)*w(b,z)
$$
satisfying $a,b>0.$ Meanwhile, $\culam[2^d+1]'(x,z)$ counts representations $$
w(x,2^d+1,z)=w(x,a)*w(b,z)
$$
satisfying $a,b>1.$ Then we have \begin{enumerate}
    \item $\culam[2^d]'=\eone[d]$ and
    \item $\culam[2^d+1]'=\oone[d].$
\end{enumerate}
\end{lemma}
\begin{proof}[Proof of part 1]
Fix some representation $$
w(x,2^d,z)=w(x,a)*w(b,z)
$$
with $a,b>0.$ It is not hard to see that the number of $2$'s dividing $a$ is the same as that for $b,$ i.e., we can write $a=2^ea'$ and $b=2^eb',$ where $a'$ and $b'$ are both odd. Then observe that $w(x,a)$ and $w(b,z)$ are both Ulam if and only if, respectively, both $w(x',a')$ and $w(b',z')$ are, where $x'=\lfloor\frac{x}{2^e}\rfloor,z'=\lfloor\frac{z}{2^e}\rfloor$ (since the last $e$ digits of both $a$ and $b$ are $0$ and cannot interfere with the corresponding ones of $x$ and $z$). Now, both $a'$ and $b'$ are odd, so both $x'$ and $z'$ must be even, which is exactly equivalent to $x,z\pmod{2^{e+1}}<2^e.$ Assuming this, we see that $w(x',a')$ and $w(b',z')$ are both Ulam if and only if $w(x'',a'')$ and $w(b'',z'')$ are, where $$
x''=\frac{x'}{2}=\floor{\frac{x}{2^{e+1}}},z''=\frac{z'}{2}=\floor{\frac{z}{2^{e+1}}},a''=\frac{a'-1}{2}=\left\lfloor\frac{a}{2^{e+1}}\right\rfloor,b''=\frac{b'-1}{2}=\left\lfloor\frac{b}{2^{e+1}}\right\rfloor.
$$
Observe that $a''+b''=\frac{a'+b'}{2}-1=\frac{a+b}{2^{e+1}}-1=2^{d-e-1}-1.$ Hence, $w(x,2^d,z)$ has the representation if and only if both $x,z\pmod{2^{e+1}}<2^e$ and we have the representation $$
w(x'',2^{d-e-1}-1,z'')=w(x'',a'')*w(b'',z'').
$$
Say, now, that $x_i,a_i,b_i,z_i$ represent the $(i+1)$-th digit from the right of the binary expansions of $x'',a'',b'',z'',$ respectively. Then note that $x_i+a_i\leq1,b_i+z_i\leq1$ for all $i.$ Additionally, we have $a''+b''=2^{d-e-1}-1,$ from which it is easy to conclude that $a_i+b_i=1$ for all $i<d-e-1.$ Thus, since $x_i+a_i+b_i+z_i\leq2,$ we have $x_i+z_i\leq1$ for all $i,$ and therefore $w(x'',z'')$ is Ulam. Observe that the representation is unique if and only if there is a unique choice of $a'',b''$ and $e.$ But observe that we have {\small$$
w(x'',2^{d-e-1}-1,z'')=w(x'',z'')*w(2^{d-e-1}-1-z'',z'')=w(x'',2^{d-e-1}-1-x'')*w(x'',z''),
$$}

\noindent so the representation is evidently only unique if $z''=2^{d-e-1}-1-x''.$ Thus, $w(x,2^d,z)$ has a representation if there exists a value $e$ such that $x,z\pmod{2^{e+1}}<2^e$ and $w(x'',z'')$ is Ulam. If $e$ is unique and in addition satisfies $x''+z''=2^{d-e-1}-1,$ then the representation is unique.

We see that such an $e$ does not exist if $x+z\geq2^d-1$; for if it did, then we would have $x''+z''\leq2^{d-e-1}-1$ and in particular $x<2^e(x'+1)=2^{e+1}x''+2^e,$ i.e., $x\leq2^{e+1}x''+2^e-1,$ and similarly $z\leq2^{e+1}z''+2^e-1.$ Thus, $$
x+z\leq2^{e+1}(x''+z'')+2^{e+1}-2=2^d-2<2^d-1.
$$
If, on the other hand, we do have $x+z<2^d-1,$ then as in the statement choose the largest $e<d$ satisfying $x,z\pmod{2^{e+1}}<2^e.$ It is not hard to see that we must have $x''_e+z''_e=2^{d-e-1}-1,$ giving at least one representation of $w(x,2^d,z).$ In order for this representation to be unique, for any other $f$ satisfying $x,z\pmod{2^{f+1}}<2^f,$ we observe that $w(x''_f,z''_f)$ must not be Ulam. Hence, there must exist an index $I$ with $f<I<e$ such that $x_I=z_I=1.$ It is not hard to see that these two conditions are equivalent to $x\pmod{2^e}+z\pmod{2^e}\geq2^e-1$; the forward direction is true because if such an $I$ exists, then $$
x\hspace{-0.15in}\pmod{2^e}+z\hspace{-0.15in}\pmod{2^e}\geq\sum\limits_{i=I}^{e-1}(x_i+z_i)2^i\geq2^{I+1}+\sum\limits_{i=I+1}^{e-1}2^i=2^e
$$ 
and the backward direction is true because if $x_f=z_f=0$ and $x_i+z_i\leq1$ for all $i$ in between $f$ and $e,$ then \begin{align*}
x\hspace{-0.15in}\pmod{2^e}+z\hspace{-0.15in}\pmod{2^e}&=\sum\limits_{i=0}^{e-1}(x_i+z_i)2^i=\sum\limits_{i=0}^{f-1}(x_i+z_i)2^i+\sum\limits_{i=f+1}^{e-1}(x_i+z_i)2^i\\
&\leq\sum\limits_{i=0}^{f-1}2^{i+1}+\sum\limits_{i=f+1}^{e-1}2^i=2^{f+1}-2+2^e-2^{f+1}=2^e-2.
\end{align*}
Note that the case $x\pmod{2^e}+z\pmod{2^e}=2^e-1$ is separate because $x_i+z_i=1$ for all $i<e,$ so no $f<e$ satisfying $x_f=z_f=0$ exists. Thus, in this case, $e$ is trivially unique.

Hence, $w(x,2^d,z)$ has a representation as a concatenation of Ulam words if and only if $x+z<2^d-1,$ and this representation is unique if and only if the largest $e$ satisfying $x,z\pmod{2^{e+1}}<2^e$ also satisfies $x\pmod{2^e}+z\pmod{2^e}\geq2^e-1,$ as desired.
\end{proof}
The second part is very similar in essence to the first, so we move its proof to the appendix for readability.

Now, we can inductively add $1$'s to the right of the binary expansion of $y$ as follows:
\begin{theorem}[\Cref{inside}]\label{adding-to-right}
Suppose $y$ is odd and $d$ is a nonnegative integer. For any $x,z,$ suppose that $x'=\lfloor\frac{x}{2^d}\rfloor,z'=\lfloor\frac{z}{2^d}\rfloor,$ and $x''=x\pmod{2^d},z''=z\pmod{2^d}.$ Then {\small$$
\culam[2^dy](x,z)=\culam[y-1](x',z')\widehat{\cdot}\eone[d](x'',z'')\widehat{+}\culam[y](x',z')\widehat{\cdot}\etwo[d](x'',z''),
$$}

\noindent or, concisely, $$
\culam[2^dy]=\culam[y-1]\widehat{\otimes}\eone[d]\widehat{+}\culam[y]\widehat{\otimes}\etwo[d].
$$
Similarly, we have {\small$$
\culam[2^dy+1](x,z)=\culam[y-1](x',z')\widehat{\cdot}\oone[d](x'',z'')\widehat{+}\culam[y](x',z')\widehat{\cdot}\otwo[d](x'',z''),
$$}

\noindent in other words $$
\culam[2^dy+1]=\culam[y-1]\widehat{\otimes}\oone[d]\widehat{+}\culam[y]\widehat{\otimes}\otwo[d].
$$
\end{theorem}
Note that the starred addition and multiplication above are in the caveman style.
\begin{proof}
\setcounter{case}{0}
\begin{case}
Suppose we have a decomposition $$
w(x,2^dy,z)=w(x,a)*w(2^dy-a,z).
$$ 
We have two cases (and these two cases are exactly the ones giving the two summands); either $a\equiv2^dy-a\equiv0\pmod{2^d}$ or both are nonzero. Regardless, define $a'=\lfloor\frac{a}{2^d}\rfloor,b'=\lfloor\frac{2^dy-a}{2^d}\rfloor$ and $a''=a\pmod{2^d},b''=(2^dy-a)\pmod{2^d}.$ In the former case, observe that $a''=b''=0,$ whereas in the latter we have $b''=2^d-a''.$
\begin{subcase}
Suppose first that $a''=b''=0.$ Then, as before, $w(x,a)$ and $w(2^dy-a,z)$ are Ulam if and only if, respectively, $w(x',a')$ and $w(b',z')$ are. %(since the last $d$ binary digits of both $a$ and $2^dy-a$ are $0,$ and therefore do not interfere with the corresponding digits of $x$ and $z$). 
Also, we see that $a'+b'=\frac{a+b}{2^d}=y.$ Thus, representations of $w(x,2^dy,z)$ with $a\equiv0\pmod{2^d}$ are in correspondence with those of $$
w(x',y,z')=w(x',a')*w(b',z').
$$
In particular, there are $\culam[y](x',z')=\culam[y](x',z')\widehat{\cdot}\etwo[d](x'',z'')$-many of them.
\end{subcase}
\begin{subcase}
Now, suppose that $a'',b''>0,$ so that $b''=2^d-a''.$ Then observe that $$
2^dy=a+b=(2^da'+a'')+(2^db'+b'')=2^d(a'+b')+2^d,
$$
from which we get $a'+b'=y-1.$ Now, $w(x,a)$ and $w(2^dy-a,z)$ are Ulam if and only if, respectively, $w(x',a')$ and $w(x'',a''),$ and $w(b',z')$ and $w(b'',z'')$ are. But then observe that we have the representations $$
w(x',y-1,z')=w(x',a')*w(b',z');\quad w(x'',2^d,z'')=w(x'',a'')*w(b'',z'').
$$
Thus, representations of $w(x,2^dy,z)$ with $a\not\equiv0\pmod{2^d}$ are in correspondence with pairs of representations of $w(x',y-1,z')$ and $w(x'',2^d,z''),$ where the latter satisfies $0<a'',b''.$ Hence, we have, by \Cref{culam-power-of-two}, $\culam[y-1](x',z')\widehat{\cdot}\eone[d](x'',z'')$ total representations of $w(x,2^dy,z)$ in this case.
\end{subcase}
In particular, in total there are $$
\culam[2^dy](x,z)=\culam[y](x',z')\widehat{\cdot}\etwo[d](x'',z'')\widehat{+}\culam[y-1](x',z')\widehat{\cdot}\eone[d](x'',z'')
$$
many representations of $w(x,2^dy,z)$ as a concatenation of Ulam words with one $1,$ as desired.
\end{case}
\begin{case}
Suppose, now, that we have a decomposition $$
w(x,2^dy+1,z)=w(x,a)*w(2^dy-a+1,z).
$$
We have two cases; either $a\equiv1-(2^dy-a+1)\equiv0,1\pmod{2^d}$ or both are nonzero and not equal to $1.$ Once again, we can define $a'=\lfloor\frac{a}{2^d}\rfloor,b'=\lfloor\frac{2^dy-a+1}{2^d}\rfloor$ and $a''=a\pmod{2^d},b''=(2^dy-a+1)\pmod{2^d}.$ In the former case, observe that $a''=1-b''=0$ or $1,$ whereas in the latter we have $b''=2^d-a''+1.$ 
\begin{subcase}
Suppose first that $a''=1-b''=0$ or $1.$ Observe, if $a$ is odd, that $x$ must be even, and vice versa if $b$ is odd, then $z$ must be even. In either case, $w(x,a)$ and $w(2^dy-a+1,z)$ are Ulam if and only if $w(x',a')$ and $w(b',z')$ are, and we see that $a'+b'=\frac{a+b-1}{2^d}=y.$ Hence, representations of $w(x,2^dy+1,z)$ with $a\equiv0\pmod{2^d}$ are in correspondence with those of $$
w(x',y,z')=w(x',a')*w(b',z'),
$$
and so are those of $w(x,2^dy+1,z)$ with $a\equiv1\pmod{2^d}.$ In particular, we observe that if $x,z$ are both odd, then there are no such representations. If exactly one of $x,z$ is odd, then representations are in correspondence with those of $w(x',y,z'),$ so there are $\culam[y](x',z')$-many of them. if both $x$ and $z$ are even, then there are two pairs of correspondences with representations of $w(x',y,z'),$ so there are $\culam[y](x',z')\widehat{\cdot}2$ total representations of $w(x,2^dy+1,z)$ with $a\equiv0,1\pmod{2^d}.$ Combining all three of these cases, we see exactly that there are $\culam[y](x',z')\widehat{\cdot}\otwo[d](x,z)$ representations of $w(x,2^dy+1,z)$ in this case.
\end{subcase}
\begin{subcase}
Suppose now that $a'',b''>1$ and so $b''=2^d-a''+1.$ Then observe that $$
2^dy+1=a+b=(2^da'+a'')+(2^db'+b'')=2^d(a'+b')+2^d+1,
$$
from which we get $a'+b'=y-1.$ We observe that $w(x,a)$ and $w(2^dy-a+1,z)$ are Ulam if and only if $w(x',a')$ and $w(x'',a''),$ and $w(b',z')$ and $w(b'',z'')$ are. But we have the representations $$
w(x',y-1,z')=w(x',a')*w(b',z');\quad w(x'',2^d+1,z'')=w(x'',a'')*w(b'',z'').
$$
Thus, representations of $w(x,2^dy+1,z)$ with $a\not\equiv0,1\pmod{2^d}$ are exactly in correspondence with pairs of representations of $w(x',y-1,z')$ and $w(x'',2^d+1,z'')$ where the latter satisfies $0,1<a'',b''.$ Hence, by \Cref{culam-power-of-two}, we have $\culam[y-1](x',z')\widehat{\cdot}\oone[d](x'',z'')$ total representations of $w(x,2^dy+1,z)$ in this case.
\end{subcase}
\noindent In particular, in total there are {\small$$
\culam[2^dy+1](x,z)=\culam[y](x',z')\widehat{\cdot}\otwo[d](x'',z'')\widehat{+}\culam[y-1](x',z')\widehat{\cdot}\oone[d](x'',z'')
$$}

\noindent many representations of $w(x,2^dy+1,z)$ as a concatenation of Ulam words with one $1,$ as desired. 
\end{case}
The corresponding tensor product representations of $\culam[2^dy]$ and $\culam[2^dy+1]$ are trivial by definition of tensor products.
\end{proof}
Alternatively, we can in fact add $1$'s to the left of the binary expansion of $y$ as well. However, for that, we need to use the whole structure of $\culam[2^k]$ and $\culam[2^k+1],$ rather than simply parts of it. As it turns out, however, the entire structure is easy to deduce from $\eone,\etwo,\oone,\otwo$; it is simply a linear combination of them. It is also useful to define `parts' of a periodic block split into fourths:
\begin{definition}\label[definition]{dividing-block-into-parts}
Say that $2^k\leq y<2^{k+1}.$ Then we can partition the periodic block of $\culam[y]$ into four parts: $\culam_0[y],\culam_1[y],\culam_2[y],\culam_3[y]:B_k\to\{0,1,2\},$ defined by
\begin{align*}
\culam_0[y](x,z)&=\culam[y](x,z)&\\
\culam_1[y](x,z)&=\culam[y](2^k+x,z)&\\
\culam_2[y](x,z)&=\culam[y](x,2^k+z)&\\
\culam_3[y](x,z)&=\culam[y](2^k+x,2^k+z).&
\hspace{1.5in}
\end{align*}
\end{definition}
\begin{wrapfigure}{r}{0.33\textwidth}
    \centering
    \vspace{-1.37in}
    \includegraphics[width=0.33\textwidth]{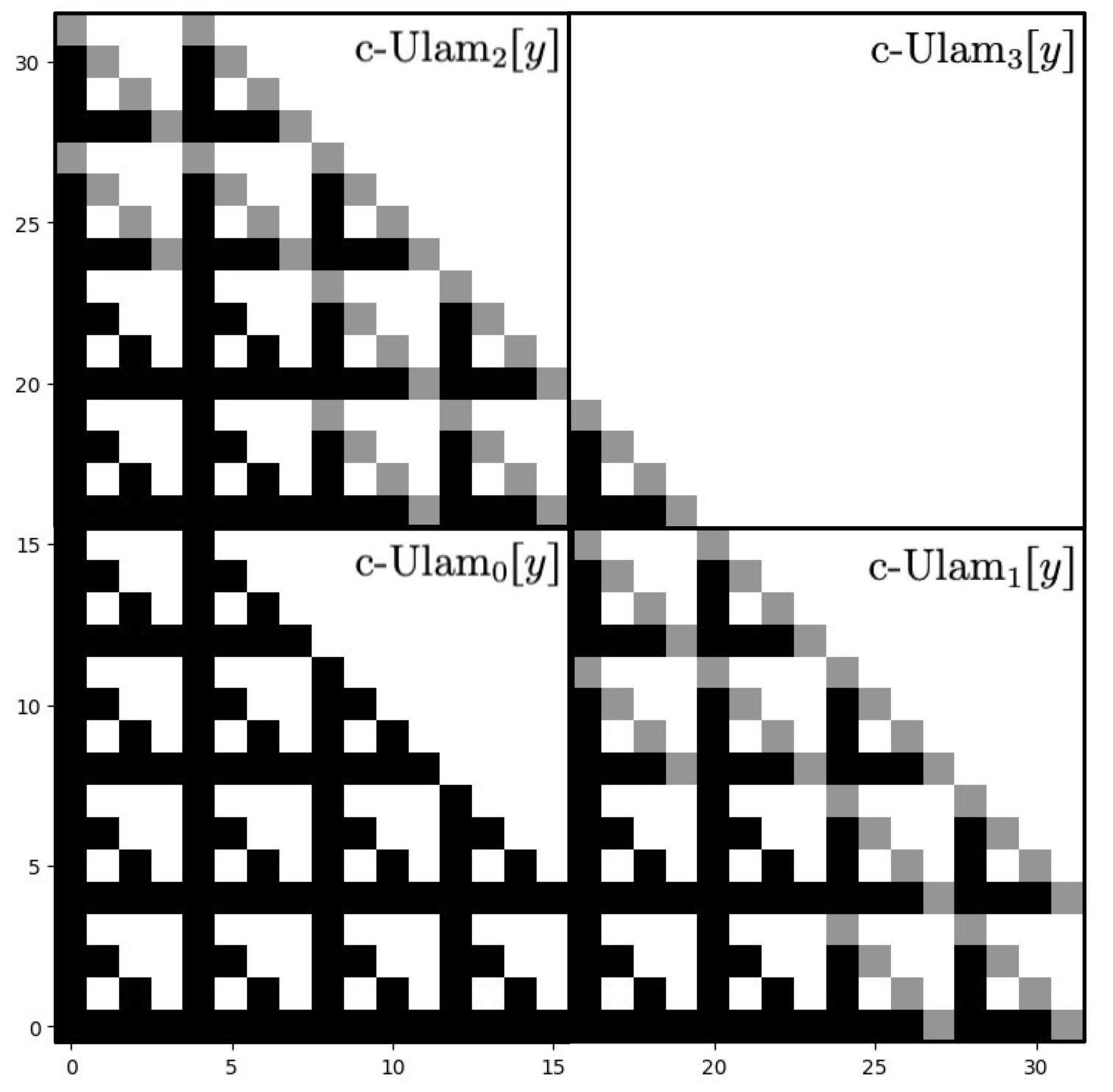}
    \caption{The four parts of $\culam[y],y=27$}
    \label{four-parts}
    \vspace{-0.9in}
\end{wrapfigure}
% \vspace{-0.1in}
We then deduce the following lemma about the parts of $\culam[2^k]$ and $\culam[2^k+1].$
\begin{lemma}\label[lemma]{structure-of-power-two}
We have \begin{flalign*}
    \culam_0[2^k]&=2\widehat{\cdot}\etwo[k]=\eone[k]+2\widehat{\cdot}\etwo[k],\\
    \culam_1[2^k]&=\culam_2[2^k]=\eone[k]\widehat{+}\etwo[k],\\
    \culam_3[2^k]&=\eone[k].
    \hspace{3.775in}
\end{flalign*}
Similarly, \begin{flalign*}
    \culam_0[2^k+1]&=2\widehat{\cdot}\otwo[k]=\oone[k]+2\widehat{\cdot}\otwo[k],\\
    \culam_1[2^k+1]&=\culam_2[2^k+1]=\oone[k]\widehat{+}\otwo[k],\\
    \culam_3[2^k+1]&=\oone[k].
\end{flalign*}
\end{lemma}
\vspace{-0.2in}
\begin{proof}
The proof essentially follows from \Cref{adding-to-right}. Observe that we have $$
\culam[2^k]=\culam[0]\widehat{\otimes}\eone[k]\widehat{+}\culam[1]\widehat{\otimes}\etwo[k].
$$
It is not hard to see that $\culam[0]=1$ and $$
\culam[1](x,z)=(x+1)\pmod2+(z+1)\pmod2.
$$
Thus, $\eone[k]$ is added once to each of $\culam_0[2^k],\culam_1[2^k],\culam_2[2^k],\culam_3[2^k],$ whereas $\etwo[k]$ is added twice to $\culam_0[2^k],$ once to $\culam_1[2^k]$ and $\culam_2[2^k],$ and not at all to $\culam_3[2^k],$ as desired. The proof for $\culam[2^k+1]$ is similar.
\end{proof}
\begin{theorem}[\Cref{outside}]\label{adding-to-left}
Suppose $2^k\leq y<2^{k+1}.$ Consider labelings $L,S:B_k\to\{0,1,2\}$ defined by $$
L(x,z):=\culam_1[y](x,z)=\culam_2[y](x,z)
$$ 
and $$
S(x,z):=\culam_3[y](x,z).
$$ 
For any $x,z<2^{\ell+1},$ let $x'=\floor{\frac{x}{2^k}},z'=\floor{\frac{z}{2^k}},$ and $x''=x\pmod{2^k},z''=z\pmod{2^k}.$ Then {\small$$
\culam[2^\ell+y](x,z)=\culam[2^{\ell-k}](x',z')\widehat{\cdot}S(x'',z'')\widehat{+}\culam[2^{\ell-k}+1](x',z')\widehat{\cdot}L(x'',z''),
$$}

\noindent or equivalently, $$
\culam[2^\ell+y]_{\big|_{B_{\ell+1}}}=\culam[2^{\ell-k}]\widehat{\otimes}S\widehat{+}\culam[2^{\ell-k}+1]\widehat{\otimes}L.
$$
\end{theorem}
\begin{proof}
First, we wish to better understand the labelings $L$ and $S.$ Observe first that $S(x,z)$ exactly counts representations of $w(2^k+x,y,2^k+z)$ as a concatenation of Ulam words $w(2^k+x,a)$ and $w(b,2^k+z).$ Note then that $w(x,a)$ and $w(b,z)$ are also both Ulam, as we are simply decreasing the number of $1$'s in the binary representations. Also note that both $2^k+x$ and $2^k+z$ have a $1$ entry in the $(k+1)$'th entry from the right, hence, correspondingly, $a$ and $b$ must not. But $a+b=y<2^{k+1},$ from which we get that $a,b<2^k.$ Alternatively, if $a,b<2^k$ and $w(x,a)$ and $w(b,z)$ are both Ulam, then so are $w(2^k+x,a)$ and $w(b,2^k+z).$ Thus, $S$ exactly counts representations of $w(x,y,z)$ as concatenations of Ulam words $w(x,a)*w(b,z)$ with $a,b<2^k.$ Similarly, $L$ counts representations of $w(x,y-2^k,z)$ as concatenations of Ulam words $w(x,a)*w(b,z).$

Suppose we have a decomposition $$
w(x,2^\ell+y,z)=w(x,a)*w(b,z).
$$
Define $a'=\lfloor\frac{a}{2^k}\rfloor,b'=\lfloor\frac{b}{2^k}\rfloor$ and $a''=a\pmod{2^k},b''=b\pmod{2^k}.$ There are two cases; either $a''+b''<2^k$ or $a''+b''\geq2^k.$ Observe in the former case that $a''+b''=y-2^k,$ whereas in the latter exactly one of $a''+b''=y.$
\setcounter{case}{0}
\begin{case}
If $a''+b''=y-2^k,$ then we have $$
2^\ell+y=a+b=(2^ka'+a'')+(2^kb'+b'')=2^k(a'+b')+y-2^k,
$$
so $a'+b'=2^{\ell-k}+1.$ Now, $w(x,a)$ and $w(b,z)$ are Ulam if and only if $w(x',a')$ and $w(x'',a''),$ and $w(b',z')$ and $w(b'',z'')$ are. But then we have the representations $$
w(x',2^{\ell-k}+1,z')=w(x',a')*w(b',z');\quad w(x'',y-2^k,z'')=w(x'',a'')*w(b'',z'').
$$
Thus, representations of $w(x,2^\ell+y,z)$ with $a''+b''<2^k$ are exactly in correspondence with pairs of representations of $w(x',2^{\ell-k}+1,z')$ and $w(x'',y-2^k,z'').$ Thus, there are $\culam[2^{\ell-k}+1](x',z')\widehat{\cdot}L(x'',z'')$ total representations of $w(x,2^\ell+y,z)$ in this case.
\end{case}
\begin{case}
If $a''+b''=y,$ then we have $$
2^\ell+y=a+b=(2^ka'+a'')+(2^kb'+b'')=2^k(a'+b')+y,
$$
so $a'+b'=2^{\ell-k}.$ Similarly to the previous case, representations of $w(x,2^\ell+y,z)$ with $a''+b''\geq2^k$ are exactly in correspondence with pairs of representations of $w(x',2^{\ell-k},z')$ and $w(x'',y,z''),$ where the latter additionally satisfies $a'',b''<2^k.$ Thus, there are $\culam[2^{\ell-k}](x',z')\widehat{\cdot}S(x'',z'')$ total representations of $w(x,2^\ell+y,z)$ in this case.
\end{case}
\noindent In particular, in total there are $$ \culam[2^\ell+y](x,z)=\culam[2^{\ell-k}+1](x',z')\widehat{\cdot}L(x'',z'')\widehat{+}\culam[2^{\ell-k}](x',z')\widehat{\cdot}S(x'',z'')
$$
representations of $w(x,2^\ell+y,z)$ as a concatenation of Ulam words with one $1,$ as desired.
The corresponding tensor product representation is simple to deduce.
\end{proof}
While these two results do give an inductive way to see that $\culam[y]$ possesses at least some sort of tiled structure, they do not illuminate the full extent of this structure. We now present it. However, in order to understand external structure, we first need to prove a result regarding the interaction between the four equal parts dividing $\culam[y].$
\begin{lemma}\label[lemma]{four-parts-containment}
We have \begin{enumerate}
    \item $\culam_1[y]=\culam_2[y],$
    \item $\culam_2[y]\geq2\widehat{\cdot}\culam_3[y],$ and 
    \item $\culam_0[y]=2\widehat{\cdot}\culam_2[y].$
\end{enumerate}
\end{lemma}
Note also that $S_y=\culam_3[y]$ and $L_y=\culam_1[y]=\culam_2[y]$ (due to which the labels $S$ and $L$ make sense).
\begin{proof}\let\qed\relax
All of these results are deduced by constructing relevant injections or bijections.
\begin{enumerate}
    \item The first equality $\culam[y](2^k+x,z)=\culam[y](x,2^k+z)$ is evident from the bijection of representations $$
    w(2^k+x,a)*w(b,z)\Longleftrightarrow\begin{cases}w(x,a)*w(b,2^k+z)&a,b<2^k,\\
    w(x,2^k+a)*w(b-2^k,2^k+z)&b\geq2^k.\end{cases}
    $$
    \item As for the containment, suppose we have a representation $w(2^k+x,y,2^k+z)=w(2^k+x,a)*w(b,2^k+z)$ as a concatenation of Ulam words. We wish to construct two representations of $w(x,y,2^k+z)$ using this. First, we trivially have the representation $w(x,a)*w(b,2^k+z).$ As for the second representation, however, we have to work harder. Say that the binary expansions of $a,b$ are $$
    a=a_{k-1}\cdots a_1a_0;\quad b=b_{k-1}\cdots b_1b_0.
    $$
    Then let $I$ be the largest index such that $a_I=b_I.$ Then observe that we must in fact have $a_i=b_i=1,$ as assuming the contrary, we get $$
    2^k\leq y=a+b=\sum_{i=0}^{k-1}(a_i+b_i)2^i\leq\sum_{i=0}^{I-1}2^{i+1}+\sum_{i=I+1}^{k-1}2^i=2^k-2^{I+1}+2^{I+1}-2=2^k-2,
    $$
    a contradiction. In particular, we have $a_I=b_I=1$ and $a_i+b_i=1$ for all $i>I,$ so that $$
    \sum_{i=I}^{k-1}(a_i+b_i)2^i=2^{I+1}+\sum_{i=I+1}^{k-1}2^i=2^k.
    $$
    Define $a'=a\pmod{2^I},b'=b\pmod{2^I},$ and then observe that $a+b=a'+b'+2^k=y.$ From here it is easy for us to deduce the validity of the additional representation $$
    w(x,y,2^k+z)=w(x,2^k+a')*w(b',2^k+z),
    $$
    and so we have $\culam[y](x,2^k+z)\geq2\widehat{\cdot}\culam[y](2^k+x,2^k+z).$ 
    \item As for the final equality, suppose that $w(x,y,2^k+z)=w(x,a)*w(b,2^k+z)$ is a representation. Then we wish to construct two representations of $w(x,y,z).$ Suppose then that $w(x,y,2^k+z)=w(x,a)*w(b,2^k+z)$ is a representation as a concatenation of Ulam words. Then observe that if $a,b<2^k,$ then as before we have an additional representation $w(x,y,2^k+z)=w(x,2^k+a')*w(b',2^k+z),$ and then, correspondingly, we have two representations $$
    w(x,y,z)=w(x,a)*w(b,z)=w(x,2^k+a')*w(b',z).
    $$
    If $a\geq2^k,$ then we have two representations $$
    w(x,y,z)=w(x,a)*w(b,z)=w(x,a-2^k)*w(2^k+b,z).
    $$
    Regardless, we have two representations of $w(x,y,z),$ as desired. Conversely, suppose we have a representation of $w(x,y,z)=w(x,a)*w(b,z).$ If $b<2^k,$ then this also gives a representation $w(x,y,2^k+z)=w(x,a)*w(b,2^k+z).$ If, however, $b\geq2^k,$ then $a=y-b<2^k$ so that we have the representation $w(x,y,2^k+z)=w(x,2^k+a)*w(b-2^k,2^k+z).$ Regardless, we have at least one representation of $w(x,y,2^k+z).$ Thus, $\culam[y](x,z)=2\widehat{\cdot}\culam[y](x,2^k+z),$ as desired.\hspace*{0em plus 1fill}\makebox{\qedsymbol}
\end{enumerate}
\end{proof}
We can now prove both general hierarchical structures.
\begin{theorem}[\Cref{hierarchy}]
Suppose that either $$
y=1\underbrace{00\cdots0}_{d_s-1}1\cdots1\underbrace{0\cdots00}_{d_1},
$$
or $$
y=1\underbrace{00\cdots0}_{d_s-1}1\cdots1\underbrace{0\cdots00}_{d_1-1}1,
$$
where $d_i\geq1$ for all $i.$ Then $\culam[y]$ exhibits both $i_{d_1,d_2,\dots,d_s}$ and $e_{d_1,d_2,\dots,d_s}$ hierarchical structure.
\end{theorem}
\begin{proof}
First, we wish to define all relevant maps and sets for our hierarchical structures. For both structures, let $T=\{0,1,2\}.$ For the external structure, also define the trivial label $t=0$ and the full label $f=2.$ 

We first demonstrate the internal structure. Let $$
I_d^e(a,b)=a\widehat{\cdot}\eone[d]\widehat{+}b\widehat{\cdot}\etwo[d]
$$
and similarly $$
I_d^o(a,b)=a\widehat{\cdot}\oone[d]\widehat{+}b\widehat{\cdot}\otwo[d].
$$
Define also $c_j^e=\culam[y_j]$ and $c_j^o=\culam[y_j+1],$ where $$
y_j=1\underbrace{00\cdots0}_{d_s-1}1\cdots1\underbrace{0\cdots00}_{d_j}
$$
is the `cut-off' of $y$ at the $j$-th step. Finally, let $c_{s+1}^e=\culam[0]$ and $c_{s+1}^o=\culam[1].$ Now, we see that $\culam[y]=c_1^e$ if $y$ is of the first form (even) and $c_1^o$ otherwise.
Moreover, for $j=1,\dots,s,$ we have, by \Cref{adding-to-right}, that {\small\begin{align*}
c_j^e(x,z)&=\culam[y_j](x,z)\\
&=\culam[y_{j+1}](x',z')\widehat{\cdot}\eone[d_j](x'',z'')\widehat{+}\culam[y_{j+1}+1](x',z')\widehat{\cdot}\etwo[d_j](x'',z'')\\
&=I_{d_j}^e(\culam[y_{j+1}](x',z'),\culam[y_{j+1}+1](x',z'))(x'',z'')\\
&=I_{d_j}^e(c_{j+1}^e(x',z'),c_{j+1}^o(x',z'))(x'',z''),
\end{align*}}

\noindent where $x',z'$ are all but the last $d_j$ digits of $x,z,$ respectively, and $x'',z''$ are the last $d_j$ digits. In a similar fashion, we have $$
c_j^o(x,z)=I_{d_j}^o(c_{j+1}^e(x',z'),c_{j+1}^o(x',z'))(x'',z'').
$$
As for the external structure, define $I$ as follows. Let \begin{align*}
I(0,0)=4,I(0,1)=2,I(0,2)=0,\\
I(1,0)=3,I(1,1)=2,I(1,2)=0,\\
I(2,0)=5,I(2,1)=2,I(2,2)=0.
\end{align*}
Now, define \begin{align*}
L_{d,0}^e=2\widehat{\cdot}\etwo[d],L_{d,1}^e=L_{d,2}^e=\eone[d]\widehat{+}\etwo[d],L_{d,3}^e=\eone[d]\\
L_{d,0}^o=2\widehat{\cdot}\otwo[d],L_{d,1}^o=L_{d,2}^o=\oone[d]\widehat{+}\otwo[d],L_{d,3}^o=\oone[d].
\end{align*}
Finally, let $c_j^i=\culam_i[y_j]$ for $i\in\{0,1,2,3\},$ where we now define either $$
y_j=1\underbrace{00\cdots0}_{d_j-1}1\cdots1\underbrace{0\cdots00}_{d_1}
$$
or $$
y_j=1\underbrace{00\cdots0}_{d_j-1}1\cdots1\underbrace{0\cdots00}_{d_1-1}1.
$$
Then we first have that $c_s^i=\culam_i[y],$ and it is easy to see that this can be rewritten in the form described in \Cref{external}. Now, \Cref{four-parts-containment} implies that $c_j^{I(a,b)}=a\widehat{\cdot}\culam_3[y_j]\widehat{+}b\widehat{\cdot}\culam_2[y_j].$ In particular, we have, by \Cref{adding-to-left} and \Cref{structure-of-power-two}, that \begin{align*}
    c_j^i(x,z)&=\culam_0[y_j](x,z)\\
    &=\culam_i[2^{d_j}](x'',z'')\widehat{\cdot}\culam_3[y_{j-1}](x',z')\\
    &\qquad\widehat{+}\culam_i[2^{d_j}+1](x'',z'')\widehat{\cdot}\culam_2[y_{j-1}](x',z')\\
    &=c_{j-1}^{I(\culam_i[2^{d_j}](x'',z''),\culam_i[2^{d_j}+1](x'',z''))}(x',z')\\
    &=c_{j-1}^{I(L_{d_j,i}^e(x'',z''),L_{d_j,i}^o(x'',z''))}(x',z'),
\end{align*}

\noindent as desired.
\end{proof}
Finally, we wish to impose fractal structures on the labelings $\culam[y].$ The hierarchies demonstrate that $\culam[y]$ is generated from a finite series of embeddings, but we wish to extend this infinitely. To do so, we use the map $\Tilde{U}$ from the $2$-adics to the set of all labelings of $Q=\N\times\{-1-\N\},$ as described in the introduction. We now prove that it is well-defined and continuous.
\begin{theorem}[\Cref{fractal}]
The map $\Tilde{U}:\Z_2\to S$ is well-defined everywhere and continuous at all $y\in\Z_2\backslash\N.$
\end{theorem}
\begin{proof}
To prove well-definedness, we need to show that $\Tilde{U}(\Tilde{y})(x,-z)=\Tilde{U}(y')(x,-z)$ is independent of the choice of $y'.$ To this end, suppose that we have two $y',y''\geq x,z.$ Say without loss of generality that $y''>y',$ and that $2^k\leq y'<2^{k+1}.$ In particular, we know that $\Tilde{y}\equiv y'\equiv y''\pmod{2^{k+1}}.$ In particular, we can write in binary \begin{align*}
y'&=1y_{k-1}\cdots y_1y_0,\\
y''&=1\underbrace{0\cdots0}_{c_\ell-1}1\cdots1\underbrace{0\cdots0}_{c_1-1}1y_{k-1}\cdots y_1y_0.
\end{align*}
By induction, it suffices to assume that $\ell=1,$ as we can continue adding greater powers of $2$ while maintaining the condition $y\geq x,z.$ Thus, we suppose that $y''=2^{k+c}+y.$ We now apply \Cref{adding-to-left}. Precisely, we have \begin{align*}
\Tilde{U}(y'')(x,-z)=&\culam[y''](x,2^{k+c+1}-z)\\
=&\culam[2^c](x',(2^{k+c+1}-z)')\widehat{\cdot}S(x'',(2^{k+c+1}-z)'')\widehat{+}\\
\widehat{+}&\culam[2^c+1](x',(2^{k+c+1}-z)')\widehat{\cdot}L(x'',(2^{k+c+1}-z)''),
\end{align*}
where $$
x'=\floor{\frac{x}{2^k}},(2^{k+c+1}-z)'=\floor{\frac{2^{k+c+1}-z}{2^k}}
$$
and $$
x''=x\pmod{2^k},(2^{k+c+1}-z)''=(2^{k+c+1}-z)\pmod{2^k}=(-z)\pmod{2^k}.
$$
Noting $x,z\leq y<2^{k+1}$ and $z>0$ (noting that $(x,-z)\in Q$), we must now split into cases. 
\setcounter{case}{0}
\begin{case}
If $x<2^k$ and $z\leq2^k,$ then $x'=0,(2^{k+c+1}-z)'=2^{c+1}-1,$ and $x''=x,(2^{k+c+1}-z)''=2^k-z.$ But it is easy to see that $\culam[2^c](0,2^{c+1}-1)=\culam[2^c+1](0,2^{c+1}-1)=1$ and thus we simply have, by \Cref{four-parts-containment}, \begin{align*}
\Tilde{U}(y'')(x,-z)&=S(x,2^k-z)\widehat{+}L(x,2^k-z)=L(x,2^k-z)\\
&=\culam[y'](x,2^{k+1}-z)=\Tilde{U}(y')(x,-z),
\end{align*}
as desired.
\end{case}
\begin{case}
If $x<2^k$ and $z>2^k,$ then $x'=0,(2^{k+c+1}-z)'=2^{c+1}-2,$ and $x''=x,(2^{k+c+1}-z)''=2^{k+1}-z.$ It is again easy to see that $\culam[2^c](0,2^{c+1}-2)=\culam[2^c+1](0,2^{c+1}-2)=2$ and thus we have, once again by \Cref{four-parts-containment}, \begin{align*}
\Tilde{U}(y'')(x,-z)&=2\widehat{\cdot}S(x,2^{k+1}-z)\widehat{+}2\widehat{\cdot}L(x,2^{k+1}-z)=2\widehat{\cdot}L(x,2^{k+1}-z)\\
&=2\widehat{\cdot}\culam_2[y'](x,2^{k+1}-z)=\culam(_0)[y'](x,2^{k+1}-z)=\Tilde{U}(y')(x,-z),
\end{align*}
as desired.
\end{case}
\begin{case}
If $x\geq2^k$ and $z\leq2^k,$ then $x'=1,(2^{k+c+1}-z)'=2^{c+1}-1,$ and $x''=x-2^k,(2^{k+c+1}-z)''=2^k-z.$ But $\culam[2^c](1,2^{c+1}-1)=1$ and $\culam[2^c+1](1,2^{c+1}-1)=0$ and thus we have \begin{align*}
\Tilde{U}(y'')(x,-z)&=S(x-2^k,2^k-z)=\culam[y'](x,2^{k+1}-z)=\Tilde{U}(y')(x,-z),
\end{align*}
as desired.
\end{case}
\begin{case}
If $x\geq2^k$ and $z>2^k,$ then $x'=1,(2^{k+c+1}-z)'=2^{c+1}-2,$ and $x''=x-2^k,(2^{k+c+1}-z)''=2^{k+1}-z.$ But now $\culam[2^c](1,2^{c+1}-2)=\culam[2^c+1](1,2^{c+1}-2)=1$ and thus we have \begin{align*}
\Tilde{U}(y'')(x,-z)&=S(x-2^k,2^{k+1}-z)\widehat{+}L(x-2^k,2^{k+1}-z)=L(x-2^k,2^{k+1}-z)\\
&=\culam[y'](x,2^{k+1}-z)=\Tilde{U}(y')(x,-z),
\end{align*}
as desired.
\end{case}
Thus, in all cases we have $\Tilde{U}(y'')(x,-z)=\Tilde{U}(y')(x,-z),$ from which we conclude that $\Tilde{U}(\Tilde{y})$ is well-defined for any $2$-adic integer $\Tilde{y}.$ 

Now, continuity at any $y\in\Z_2\backslash\N$ is evident; for any $\e>0,$ choose $y'$ such that $2^k\leq y'<2^{k+1}$ and $2^k>\frac{1}{\e},$ and $\Tilde{y}\equiv y'\pmod{2^{k+1}}.$ Then for any $Y\in\Z_2$ such that $\norm{Y-y}_2\leq\frac{1}{2^{k+1}},$ we have $Y\equiv y\equiv y'\pmod{2^{k+1}}$ and thus for all $x,\abs{z}\leq2^k,$ definitionally, $\Tilde{U}(Y)(x,z)=\Tilde{U}(y')(x,z)=\Tilde{U}(y)(x,z).$ In particular, the distance $\rho(\Tilde{U}(y),\Tilde{U}(Y))\leq\frac{1}{2^k}<\e,$ as desired.
\end{proof}
\section{Sharp bounds}\label{consequences}
We conclude by showing the power of the general reductive theorems of \Cref{main-theorem} in proving several unexpected deterministic properties of Ulam words with two $1$'s.
\begin{theorem}\label{cute-thm}
Suppose $2^k\leq y<2^{k+1},$ and let $\Tilde{x}$ and $\Tilde{z}$ be the reductions of $x$ and $z,$ respectively, modulo $2^{k+1}.$
\begin{enumerate}
    \item If $y$ is not a Zumkeller number, then $(x,z)\in\ulam[y]$ only if $x+z\geq2^{k+1}-1.$
    \item If $y$ is not a Zumkeller number, then $(x,z)\in\ulam[y]$ only if $2^{k+1}-1\leq \Tilde{x}+\Tilde{z}\leq 2^{k+2}-y-2.$ If $y$ is a Zumkeller number, then this is true unless $x+z\in\{y,y+1,y+2\}.$\label{helpful-upper-bound}
    \item If $y\not\in\{2^{k+1}-1,2^{k+1}-2\},$ then $(x,z)\in\ulam[y]$ for all $x+z\equiv-1\pmod{2^{k+1}}.$ If $y=2^{k+1}-1$ or $y=2^{k+1}-2,$ then the same result holds except when $x+z=2^{k+1}-1.$
    \item For all $y,$ we have $(x,z)\in\ulam[y]$ whenever either $\Tilde{x}+\Tilde{z}=2^{k+2}-y-2$ or $\Tilde{x}+\Tilde{z}=2^{k+2}-y-3,$ whichever is odd.
\end{enumerate}
\end{theorem}
All of these results follow from relatively simple arguments by induction, as we will now show. In the context of $\culam[y]$ for a given $y,$ let $\Tilde{x},\Tilde{z}$ always denote the reductions of $x$ and $z$ modulo $2^{k+1},$ the first power of $2$ greater than $y.$
\begin{proof}[Proof of part 1]
First, since $y$ is not a Zumkeller number, we know that $(x,z)\in\ulam[y]$ if and only if $\culam[y](x,z)=1$ and $x+z\equiv1\pmod2.$ We now prove that this modified fact, in fact, simply the condition $\culam[y](x,z)=1,$ implies $x+z\geq2^{k+1}-1$ by induction for all $y.$ This evidently implies the conclusion. 
    
The base cases $y=2^k$ follow from a simple calculation; if $x+z<2^{k+1}-1,$ then if $x,z<2^k,$ we have by \Cref{structure-of-power-two} that $$
\culam[2^k](x,z)=\culam_0[2^k](x,z)=2\widehat{\cdot}\etwo[k](x,z)\neq1
$$
and otherwise exactly one of $x,z$ is still less than $2^k$; say it is $x,$ so that $$
\culam[2^k](x,z)=\culam_2[2^k](x,z-2^k)=\eone[k](x,z-2^k)\widehat{+}\etwo[k](x,z-2^k).
$$
But $x+z-2^k<2^k-1,$ and thus $\eone[k](x,z-2^k)\geq1$ and $\etwo[k](x,z-2^k)=1,$ so $\culam[2^k](x,z)=2.$ Thus, regardless, $\culam[2^k](x,z)\neq1$ if $x+z<2^{k+1}-1.$
    
Now, suppose $\culam[Y](x,z)=1,$ implies $x+z\geq2^{K+1}-1$ for all $Y<y,$ where $K=\floor{\log_2(Y)},$ and assume that $y$ is not a power of $2.$ Note that if $y$ is odd, then $\culam[y](x,z)=1$ only if $\culam[y-1](x,z)=1$ and then we proceed by induction. Thus, we can assume without loss of generality that $y$ is even. Consider any $x,z$ such that $x+z<2^{k+1}-1.$ We can write $y=2^dy'$ for some odd $y'$ and $d>0.$ Then, by \Cref{adding-to-right}, we have {\small$$
\culam[y](x,z)=\culam[y'-1](x',z')\widehat{\cdot}\eone[d](x'',z'')\widehat{+}\culam[y'](x',z')\widehat{\cdot}\etwo[d](x'',z''),
$$}
    
\noindent where $x'=\lfloor\frac{x}{2^d}\rfloor,z'=\lfloor\frac{z}{2^d}\rfloor,$ and $x''=x\pmod{2^d},z''=z\pmod{2^d}.$ Since $x+z<2^{k+1}-1,$ we have $$
x'+z'\leq\frac{x}{2^d}+\frac{z}{2^d}<\frac{2^{k+1}-1}{2^d}<2^{k-d+1}.
$$ 
Since $2^{k-d}\leq y'<2^{k-d+1},$ we know from our inductive hypothesis that $\culam[y'](x',z')\neq1.$ It is then evident that $\culam[y'](x',z')\widehat{\cdot}\etwo[d](x'',z'')\neq1$ as well. Additionally, we have $2^{k-d}\leq y'-1<2^{k-d+1}$ as well, since $y,$ and thus $y',$ are not powers of $2.$ Hence, we also have $\culam[y'-1](x',z')\neq1$ and thus $\culam[y'-1](x',z')\widehat{\cdot}\eone[d](x'',z'')\neq1.$ From this, we immediately obtain that $\culam[y](x,z)\neq1.$
    
Hence, in any case, if $x+z<2^{k+1}-1,$ then $\culam[y](x,z)\neq1,$ completing the inductive step.
\end{proof}
Note that the part 2 incorporates part 1 within its statement. Naturally, then, its proof will use the already obtained result.
\begin{proof}[Proof of part 2]
Translating to $\culam[y],$ we recall that we can have exceptional aperiodic points $(x,z)\in\ulam[y]$ only if $x+z\in\{y,y+1,y+2\},$ hence the second part of the statement. Otherwise, we once again know and assume that $(x,z)\in\ulam[y]$ if and only if $\culam[y](x,z)=1$ and $x+z\equiv1\pmod2.$ It then suffices to prove specifically that $\culam[y](x,z)=\culam[y](\Tilde{x},\Tilde{z})=1$ implies $\Tilde{x}+\Tilde{z}\leq2^{k+2}-y-2,$ by induction. Note that we have already proved the other inequality $2^{k+1}-1\leq\Tilde{x}+\Tilde{z}$ in the first part. 
    
The base cases $y=2^k$ for some $k$ are trivial; the upper bound for $2^k$ is $2^{k+2}-2^k-2,$ which can only be violated if $\Tilde{x},\Tilde{z}\geq2^k,$ in which case we can set $x'=\Tilde{x}-2^k,z'=\Tilde{z}-2^k$ so that $x'+z'>2^k-2.$ Then we recall by \Cref{structure-of-power-two} that $$
\culam[2^k](x,z)=\culam[2^k](\Tilde{x},\Tilde{z})=\culam_3[2^k](x',z')=\eone[k](x',z')=0.
$$
Suppose now that $\culam[Y](x,z)=1$ implies $\Tilde{x}+\Tilde{z}\leq2^{K+2}-Y-2$ for all $Y<y$ (where $K=\floor{\log_2(Y)}$). Now, fix $x$ and $z$ and say that $\Tilde{x}+\Tilde{z}>2^{k+2}-y-2.$ If $y$ is odd and $\culam[y](x,z)=1,$ then $x+z$ must be odd and also $\culam[y-1](x,z)=1,$ so for any such $x,z,$ we have $$
\Tilde{x}+\Tilde{z}\leq2^{k+2}-(y-1)-2=2^{k+2}-y-1.
$$
But since $\Tilde{x}+\Tilde{z}$ is odd and $2^{k+2}-y-1$ is even, we must in fact have $\Tilde{x}+\Tilde{z}\leq2^{k+2}-y-2,$ contradiction. Thus, we may assume that $y=2^dy'$ is even. Then we have {\small$$
\culam[y](x,z)=\culam[y'-1](x',z')\widehat{\cdot}\eone[d](x'',z'')\widehat{+}\culam[y'](x',z')\widehat{\cdot}\etwo[d](x'',z'').
$$}
    
\noindent Now, observe that $2^{k-d}\leq y'<2^{k-d+1}.$ Moreover, we also have $2^{k-d}\leq y'-1<2^{k-d+1},$ since $y'$ is not a power of $2.$ Since $$
\Tilde{x}+\Tilde{z}=2^d(\Tilde{x'}+\Tilde{z'})+x''+z''>2^{k+2}-y-2,
$$
we must have either $\Tilde{x'}+\Tilde{z'}>2^{k-d+2}-y'-1$ or $x''+z''>2^d-2.$ If, indeed, the first inequality holds, then by the inductive hypothesis we in fact have that neither $\culam[y'-1](x',z')$ nor $\culam[y'](x',z')$ can equal $1.$ Thus, neither can the respective summands, nor their sum ($1$ is only a product of two $1$'s and only a sum of a $1$ and a $0$). Hence, in this case $\culam[y](x,z)\neq1.$ 
    
If $\Tilde{x'}+\Tilde{z'}\leq2^{k-d+2}-y'-2,$ then we see that $x''+z''>2^{d+1}-2,$ which is impossible since $x'',z''<2^d.$ Thus, the only remaining case is that $\Tilde{x'}+\Tilde{z'}=2^{k-d+2}-y'-1$ and then correspondingly $x''+z''=2^d-2.$ It is easy to see from here, since $\Tilde{x'}+\Tilde{z'}>2^{k-d+2}-y'-2,$ that $\culam[y'](x',z')\neq1.$ Also, $\Tilde{x'}+\Tilde{z'}$ is even, and thus $\culam[y'-1](x',z')\neq1$ as well. Thus, the same conclusion as before holds, implying that $\culam[y](x,z)\neq1.$
    
In particular, we have $\culam[y](x,z)\neq1$ whenever $x+z>2^{k+2}-y-2$ and thus the inductive step is proved.
\end{proof}
The next two results, while simple to state, end up being nontrivial consequences of the general tensor structure theorem, as seen further.
\begin{proof}[Proof of part 3]
The result follows in the exact stated form for all Zumkeller $y$ by \Cref{special-case}. Thus, we may assume that $y$ is not Zumkeller, so $(x,z)\in\ulam[y]$ if and only if $\culam[y](x,z)=1$ and $x+z$ is odd. Since we are assuming $x+z\equiv-1\pmod{2^{k+1}},$ the odd condition is tautological. Thus, it suffices to prove, for any $x+z\equiv-1\pmod{2^{k+1}},$ that $\culam[y](x,z)=1.$ Once again, we prove this by induction.
    
First, note that $x+z\equiv-1\pmod{2^{k+1}}$ is in fact equivalent to $\Tilde{x}+\Tilde{z}=2^{k+1}-1.$ We will use this instead.
    
The base cases are $y=2^k$ for any $k.$ Note that we must have either $\Tilde{x}\geq2^k$ or $\Tilde{z}\geq2^k,$ but not both. Without loss of generality, the second is true, so we have by \Cref{structure-of-power-two} that \begin{align*}
\culam[2^k](x,z)&=\culam[2^k](\Tilde{x},\Tilde{z})=\culam_2[2^k](\Tilde{x},\Tilde{z}-2^k)\\
&=\eone[k](\Tilde{x},\Tilde{z}-2^k)\widehat{+}\etwo[k](\Tilde{x},\Tilde{z}-2^k).
\end{align*}
But we know that $\etwo[k](\Tilde{x},\Tilde{z}-2^k)=1$ and $\eone[k](\Tilde{x},\Tilde{z}-2^k)=0$ because $\Tilde{x}+\Tilde{z}-2^k=2^k-1.$ Thus, $\culam[2^k](x,z)=1,$ as desired.
    
Now, suppose that $\culam[Y](x,z)=1$ whenever $\Tilde{x}+\Tilde{z}=2^{K+1}-1$ for all $Y<y$ (where $K=\floor{\log_2(Y)}$). Fix $x$ and $z$ such that $\Tilde{x}+\Tilde{z}=2^{k+1}-1.$ If $y$ is odd, then $$
\culam[y](x,z)=\culam[y-1](x,z)=1,
$$
so we may assume without loss of generality that $y=2^dy'$ is even. Then {\small$$
\culam[y](x,z)=\culam[y'-1](x',z')\widehat{\cdot}\eone[d](x'',z'')\widehat{+}\culam[y'](x',z')\widehat{\cdot}\etwo[d](x'',z'').
$$}
    
\noindent Now, we have $$
\Tilde{x}+\Tilde{z}=2^d(\Tilde{x'}+\Tilde{z'})+x''+z''=2^{k+1}-1,
$$
from which it is easy to deduce that $\Tilde{x'}+\Tilde{z'}=2^{k-d+1}-1$ and $x''+z''=2^d-1.$ Therefore, since $2^{k-d}\leq y'-1,y<2^{k-d+1}$ (as before), we see that $\culam[y'-1](x',z')=\culam[y'](x',z')=1.$ Additionally, $\eone[d](x'',z'')=0$ and $\etwo[d](x'',z'')=1,$ from which we immediately obtain $\culam[y](x,z)=1,$ as desired.
\end{proof}
\begin{proof}[Proof of part 4]
Once again, the result holds for all Zumkeller $y$ by \Cref{special-case}. Thus, we may assume that $(x,z)\in\ulam[y]$ if and only if $\culam[y](x,z)=1$ and $x+z$ is odd. Since we are assuming $\Tilde{x}+\Tilde{z}=2\ceil{2^{k+1}-\frac{y}{2}-1}-1,$ the odd condition is tautological. Thus, it suffices to prove, for any $\Tilde{x}+\Tilde{z}=2\ceil{2^{k+1}-\frac{y}{2}-1}-1,$ that $\culam[y](x,z)=1.$ We prove this by induction. We need to in fact strengthen our inductive statement---that also $\culam[y](x,z)=1$ if $y$ is even and $x+z=2^{k+2}-y-2$ (if $y$ is odd, then $2^{k+2}-y-2=2\ceil{2^{k+1}-\frac{y}{2}-1}-1$).
    
The base cases are once again $y=2^k$ for any $k.$ We may assume $k>0,$ as the case $k=0$ devolves into a triviality. Now, we have $\Tilde{x}+\Tilde{z}\in\{2^{k+2}-2^k-3,2^{k+2}-2^k-2\}.$ Note that we must have at least one of $\Tilde{x}\geq2^k$ or $\Tilde{z}\geq2^k.$ If only one is true, without loss of generality the second, then we have \begin{align*}
\culam[2^k](x,z)&=\culam[2^k](\Tilde{x},\Tilde{z})=\culam_2[2^k](\Tilde{x},\Tilde{z}-2^k)\\
&=\eone[k](\Tilde{x},\Tilde{z}-2^k)\widehat{+}\etwo[k](\Tilde{x},\Tilde{z}-2^k).
\end{align*}
But $\Tilde{x}+\Tilde{z}-2^k\geq2^{k+1}-3\geq2^k-1$ (since $k>0$), and thus $\eone[k](\Tilde{x},\Tilde{z}-2^k)=0.$ But $\etwo[k]=1$ and therefore indeed $\culam[2^k](x,z)=1.$ If, on the other hand, $\Tilde{x},\Tilde{z}\geq2^k,$ then {\small$$
\culam[2^k](x,z)=\culam[2^k](\Tilde{x},\Tilde{z})=\culam_3[2^k](\Tilde{x}-2^k,\Tilde{z}-2^k)=\eone[k](\Tilde{x}-2^k,\Tilde{z}-2^k).
$$}
    
\noindent Now, we have $(\Tilde{x}-2^k)+(\Tilde{z}-2^k)\in\{2^k-3,2^k-2\}$ (which is notably impossible for $k=1$). It thus remains to show that $\eone[k](\Tilde{x}-2^k,\Tilde{z}-2^k)=1,$ which is relatively tedious, but we provide the proof here for completeness. Let $x'=\Tilde{x}-2^k=x_{k-1}x_{k-2}\cdots x_1x_0$ in binary and $z'=\Tilde{z}-2^k=z_{k-1}z_{k-2}\cdots z_1z_0,$ and say that $e<k$ is the largest satisfying $x',z'\pmod{2^{e+1}}<2^e.$ Then $x_i+z_i=1$ for all $i>e,$ and $x_e+z_e=0.$ Say $x''=x\pmod{2^e}$ and $z''=z\pmod{2^e}.$ Then we have $$
2^k-3\leq x'+z'=2^{k-1}+\cdots+2^{e+1}+x''+z'',
$$
so $x''+z''\geq2^{e+1}-3\geq2^e-1$ unless $e=0,$ in which case $2^e-1=0$ and still $x''+z''\geq2^e-1.$ In particular, $\eone[k](x',z')=1,$ and so $\culam[2^k](x,z)=1,$ as desired.
    
Now, suppose that $\culam[Y](x,z)=1$ whenever $\Tilde{x}+\Tilde{z}=2\ceil{2^{K+1}-\frac{Y}{2}-1}-1$ or $2^{K+2}-Y-2$ for all $Y<y$ (where $K=\floor{\log_2(Y)}$). Fix $x$ and $z$ such that $\Tilde{x}+\Tilde{z}=2\ceil{2^{k+1}-\frac{y}{2}-1}-1$ or $2^{k+2}-y-2.$ If $y$ is odd, then $\Tilde{x}+\Tilde{z}=2\ceil{2^{k+1}-\frac{y-1}{2}-1}-1$ as well and so by induction $$
\culam[y](x,z)=\culam[y-1](x,z)=1,
$$
so we may assume without loss of generality that $y=2^dy'$ is even. Then {\small$$
\culam[y](x,z)=\culam[y'-1](x',z')\widehat{\cdot}\eone[d](x'',z'')\widehat{+}\culam[y'](x',z')\widehat{\cdot}\etwo[d](x'',z'').
$$}
    
\noindent Now, we have $$
\Tilde{x}+\Tilde{z}=2^d(\Tilde{x'}+\Tilde{z'})+x''+z''\in\{2^{k+2}-y-3,2^{k+2}-y-2\},
$$
from which we either have $\Tilde{x'}+\Tilde{z'}=2^{k-d+2}-y'-2$ and $x''+z''\in\{2^{d+1}-3,2^{d+1}-2\}$ or $\Tilde{x'}+\Tilde{z'}=2^{k-d+2}-y'-1$ and $x''+z''\in\{2^d-3,2^d-2\}.$ In both cases, note that $2^{k-d}\leq y'-1,y<2^{k-d+1}$ (as before). In the former case, we see that, since $y'$ is odd, $\culam[y'](x',z')=\culam[y'-1](x',z')=1.$ Moreover, $2^{d+1}-2,2^{d+1}-3\geq2^d-1$ (since $d\geq1$) and so $\eone[d](x'',z'')=0.$ Since $\etwo[d]=1,$ we immediately obtain $\culam[y](x,z)=1\widehat{\cdot}0\widehat{+}1\widehat{\cdot}1=1.$ In the latter case, we use \Cref{helpful-upper-bound} to deduce that $\culam[y'](x',z')=0$ and the inductive step to see that $\culam[y'-1](x',z')=1.$ Moreover, we showed above that $\eone[d](x'',z'')=1$ in either case ($x''+z''$ is either $2^d-3$ or $2^d-2$) and so $\culam[y](x,z)=1\widehat{\cdot}1\widehat{+}0=1,$ completing the inductive step.
\end{proof}
\appendix
\section{Appendix}\label{app}
\subsection{Proof of part 3 of \Cref{zumkeller}}
Suppose that $(x,z)\in\ulam[2^{k+1}-2^a-1].$ As in the case of $2^{k+1}-1,$ \Cref{checkerboard} applies, so we simply need to determine which words $w(x,2^{k+1}-2^a-1,z)$ have a unique representation as a concatenation of \textit{distinct} Ulam words $w(x,b)*w(2^{k+1}-2^a-1-b,z).$ Note that we may separately consider the case $x+z=2^{k+1}-2^a-1$ and otherwise remove the condition that the words must be distinct. Note that since $2^{k+1}-2^a-1$ is odd, in order for $w(x,2^{k+1}-2^a-1,z)$ to have a \textit{unique} representation we must have by an argument in \Cref{checkerboard} that $x+z$ is odd.
    
Hence, suppose first that $x+z\neq2^{k+1}-2^a-1.$ Then exactly one of $x,z$ is odd, say without loss of generality that it is $x.$ Now observe that any representation $w(x,2^{k+1}-2^a-1,z)=w(x,b)*w(2^{k+1}-2^a-1-b,z)$ must satisfy that $b$ is even, and thus $2^{k+1}-2^a-1-b$ is odd. Hence, the last digits of the binary expansions of each term are uniquely determined, and so we may chop them off. We then obtain the respective representation $w((x-1)/2,2^k-2^{a-1}-1,z/2)=w((x-1)/2,b/2)*w(2^k-2^{a-1}-b/2-1,z/2).$ In other words, $w(x,2^{k+1}-2^a-1,z)$ has a unique representation if and only if $w((x-1)/2,2^k-2^{a-1}-1,z/2)$ does, and continuing this repetitive operation (noting at each step that there is a unique representation only if exactly one of $x',z'$ is odd), if and only if $w(\Tilde{x},2^{k+1-a}-2,\Tilde{z})$ does, where $\Tilde{x}=\lfloor\frac{x}{2^a}\rfloor$ and $\Tilde{z}=\lfloor\frac{z}{2^a}\rfloor.$ But we already discussed above that $w(\Tilde{x},2^{k+1-a}-2,\Tilde{z})$ has a unique representation (as a concatenation of not necessarily distinct Ulam words with one $1$) if and only if the reductions $\Tilde{x}',\Tilde{z}'$ of $\Tilde{x},\Tilde{z}$ modulo $2^{k+1-a}$ satisfy $\Tilde{x}'+\Tilde{z}'\in\{2^{k+1-a}-1,2^{k+1-a}\}.$ But observe that $x+z=2^a(\Tilde{x}+\Tilde{z})+2^a-1,$ and so if $x',z'$ are the respective reductions of $x,z$ modulo $2^{k+1},$ then either $x'+z'=2^{k+1}-1$ or $x'+z'=2^{k+1}+2^a-1,$ as desired.
    
Now, in the special case that $x+z=2^{k+1}-2^a-1,$ we note that issues may arise only if $w(x,z)$ is Ulam, as then we have the extra representation $$
w(x,2^{k+1}-2^a-1,z)=w(x,z)*w(x,z).
$$
Then $w(x,z)$ is Ulam if and only if there are no carries in adding $x,z$ in binary, i.e., if and only if the entries of $x+z$ that are $0$ are also $0$ in both $x$ and $z.$ Note that the only $0$ in $x+z$ is in the $(a+1)$-th entry from the right, so the only requirement on $x$ and $z$ is that their $(a+1)$-th entry from the right is $0.$ Note also that in all other slots exactly one of $x$ and $z$ has a $1$ in its respective entry (since $x+z$ does). If this is not the case, observe that $x+z\neq2^{k+1}-1$ or $2^{k+1}+2^a-1$ and so $w(x,2^{k+1}-2^a-1,z)$ is not Ulam. If, on the other hand, $x$ has a $0$ in its $(a+1)$-th entry, then observe that we can in fact still apply the same reduction as in the case when $x+z\neq2^{k+1}-2^a-1.$ Therefore, the number of representations of $w(x,2^{k+1}-2^a-1,z)$ as a concatenation of (not necessarily distinct) Ulam words with one $1$ is equal to the same number for $w(\Tilde{x},2^{k-a+1}-2,\Tilde{z}),$ where $\Tilde{x}=\lfloor\frac{x}{2^a}\rfloor$ and $\Tilde{z}=\lfloor\frac{z}{2^a}\rfloor.$ Observe that the condition on the $(a+1)$-th entry from the right in the binary representations of $x$ and $z$ being $0$ is equivalent to the last digits of $\Tilde{x}$ and $\Tilde{z}$ being $0,$ i.e., that both are even. As above, we also know that in every other slot exactly one of $\Tilde{x}$ and $\Tilde{z}$ has its entry equal to $1.$ In particular, $\Tilde{x}+\Tilde{z}=2^{k-a+1}-2.$ But we know by our above characterization that $w(\Tilde{x},2^{k-a+1}-2,\Tilde{z})$ has exactly two representations as a concatenation of Ulam words with one $1$ (with one of the two representations having equal words) whenever $\Tilde{x}+\Tilde{z}=2^{k-a+1}-2$ and $\Tilde{x}$ (hence also $\Tilde{z}$) is even. Thus, $w(x,2^{k+1}-2^a-1,z)$ also does. One of the representations is $w(x,z)*w(x,z),$ hence, there is exactly one representation as a concatenation of \textit{distinct} Ulam words. In other words, when $x+z=2^{k+1}-2^a-1,$ we see that $w(x,2^{k+1}-2^a-1,z)$ is Ulam if and only if $x,z$ have a $0$ in the $(a+1)$-th entry from the right of their binary representation, i.e., $x\pmod{2^{a+1}}<2^a,$ as desired.
\subsection{Proof of part 2 of \Cref{culam-power-of-two}}
Here, we simplify our casework by using \Cref{odd-even-equal}. From this, we observe that if $x+z$ is odd, then $\ulam[2^d+1](x,z)=\ulam[2^d](x,z)$ unless $d=1,$ in which case there are exceptions when $x,z<4.$ However, the periodic blocks are always the same and, correspondingly, since the update step does not affect the state of coordinates with odd sum, we know that $\culam[2^d+1](x,z)=\culam[2^d](x,z)$ as well. It is not hard to see that $\eone[d]$ is simply the description of $\culam[2^d]$ on the square $\{2^d,2^d+1,\dots,2^{d+1}-1\}$ and $\oone[d]$ is simply the description of $\culam[2^d+1]$ on the same square. Thus, if $x+z$ is odd, then the description of $\oone[d](x,z)$ is exactly the same as that of $\eone[d](x,z)$ (which is to be expected). Thus, it remains to consider the case of $x+z$ being even.

Fix a representation $$
w(x,2^d+1,z)=w(x,a)*w(b,z)
$$
with $a,b>1$ and $x+z\equiv0\pmod2.$ Observe that $a+b=2^d+1,$ therefore, exactly one of $a,b$ is odd. In particular, we cannot have both $x$ and $z$ be odd, as otherwise one of $w(x,a)$ or $w(b,z)$ must not be Ulam. In particular, observe that $\oone[d](x,z)=0$ whenever $x$ and $z$ are both odd. Since $x+z$ is even, we therefore have that $x$ and $z$ must both be even.

If $a$ is odd, then observe that $w(x,a)$ is Ulam if and only if $w(x,a-1)$ is, so these representations are in bijection with those of $$
w(x,2^d,z)=w(x,a')*w(b,z)
$$
with $a',b,x,z$ all even, and $a',b>0$ (since $b$ is even, we must have $b>1$ automatically). These, in turn, are in bijection with those of $$
w(x/2,2^{d-1},z/2)=w(x/2,a'/2)*w(b/2,z/2)
$$
with $a'/2,b/2>0,$ i.e., giving the labeled set $\eone[d-1].$

If, instead, $b$ is odd, then our representations are once again in bijection with those of $\eone[d-1].$ In other words, we see that $\oone[d](x,z)=\eone[d-1](x/2,z/2)\widehat{+}\eone[d-1](x/2,z/2).$ If $\frac{x}{2}+\frac{z}{2}\geq2^{d-1}-1,$ or, equivalently, $x+z\geq2^d-2,$ then we know that $\eone[d-1](x/2,z/2)=0,$ and thus $\oone[d](x,z)=0.$ If, on the other hand, $x+z<2^d-2,$ then we know that $\eone[d-1](x/2,z/2)\geq1$ and thus $\oone[d](x,z)\geq2$ and thus is equal to $2.$

In particular, going through all of our cases, we see that if $x+z\geq2^d-2,$ then $\oone[d](x,z)=0.$ If $x+z<2^d-2$ and is even, then if $x$ is odd, we have $\oone[d](x,z)=0,$ and if $x$ is even, then $\oone[d](x,z)=2.$ Finally, if $x+z<2^d-2$ and is odd, then $\oone[d](x,z)=\eone[d](x,z)$ and hence has the same description as $\eone[d](x,z)$ (as is stated in the definition), as desired.
\unappendix
\bibliographystyle{alpha}

% \bibliography{biblio}

\end{document}